\documentclass{amsart}
\usepackage{graphicx,amsmath,amssymb,stmaryrd,amsthm} 
\usepackage[utf8]{inputenc}
\usepackage{subcaption}
\newcommand{\mA}{\mathbf{A}}
\newcommand{\mB}{\mathbf{B}}

\newcommand{\mX}{\mathbf{X}}
\newcommand{\mZ}{\mathbf{Z}}

\newcommand{\mV}{\mathbf{V}}

\newcommand{\mY}{\mathbf{Y}}
\newcommand{\mx}{\mathbf{x}}
\newcommand{\my}{\mathbf{y}}
\newcommand{\mf}{\mathbf{f}}
\newcommand{\mn}{\mathbf{n}}
\newcommand{\R}{\mathbb{R}}
\newcommand{\p}{\partial}
\usepackage[dvipsnames]{xcolor}
\newcommand{\e}{\varepsilon}

\newcommand{\jump}[1]{\llbracket{#1}\rrbracket}

\newtheorem{theorem}{Theorem}[section]
\renewcommand\thetheorem{\arabic{section}.\arabic{theorem}}

\newtheorem{remark}{Remark}[section]

\newtheorem{definition}[theorem]{Definition}

\newtheorem{proposition}[theorem]{Proposition}

\newtheorem{lemma}[theorem]{Lemma}
\renewcommand\thelemma{\arabic{section}.\arabic{theorem}}

\newtheorem{corollary}[theorem]{Corollary}

\begin{document}

\title{Anchored Peskin Problem}
\author{Achyuta Telekicherla Kandalam}
\address{School of Mathematics, University of Minnesota, Minneapolis, MN 55455}
\email{telek003@umn.edu}

\author{Daniel Spirn}
\address{School of Mathematics, University of Minnesota, Minneapolis, MN 55455}
\email{spirn@umn.edu}

\thanks{The second author was supported in part by NSF grants DMS-2009352 and DMS-2511002}

\maketitle

\begin{abstract}
The Immersed Boundary Method has long served as a robust computational framework for fluid-structure interactions, yet the rigorous analysis of 1D Peskin filaments anchored to rigid boundaries remains sparse. In this paper, we generalize the classical Peskin problem to the half-plane by considering an elastic filament whose endpoints are  anchored to a no-slip wall. Moving beyond the algebraic complexity of the traditional Blake image system, we utilize the boundary-symmetric formulation of Gimbutas, Greengard, and Veerapaneni \cite{gimbutas2015simple}. This  representation allows for a transparent decomposition of the hydrodynamic interactions into a free-space principal part and a regularizing reflected component without resorting to hypersingular integral operators. 

Through this framework, we prove that the leading-order evolution of the anchored filament is governed by a fractional Laplacian equipped with homogeneous Dirichlet boundary conditions. We characterize the stationary states of the system, proving that all equilibria are circular arcs connecting the anchor points—a result that holds for a broad class of elastic energy densities. By framing the non-local dynamics in weighted little-H\"{o}lder spaces, we establish local well-posedness and prove that the filament exhibits instantaneous $C^\infty$ regularization in both space and time. This work provides a rigorous analytical foundation for anchored filaments in bounded domains and suggests a spectrally accurate numerical path for simulating tethered biological structures.
\end{abstract}


\section{Introduction}

We consider the dynamics of an active, one-dimensional elastic filament immersed in a highly viscous, incompressible fluid bounded by a rigid wall. Mathematically, this system is governed by the Immersed Boundary (IB) equations, originally introduced by Peskin \cite{peskin1972flow}. Let the fluid domain be the upper half-plane $\mathbb{R}^2_+ = \{ (x,y) \in \mathbb{R}^2 : y > 0 \}$ and let the filament be parameterized by a time-dependent curve $\mathbf{X}(t,s) : [0,\pi] \to \mathbb{R}^2_+$. The fluid velocity $\mathbf{u}$ and pressure $p$ satisfy the forced Stokes equations:
\[
\begin{aligned}
    -\Delta \mathbf{u} + \nabla p &= \int_0^\pi \partial_s \left( \mathcal{T}(\mathbf{X}) \partial_s \mathbf{X} \right) \delta(\mathbf{x} - \mathbf{X}(t,s)) \, ds, \quad \text{in } \mathbb{R}^2_+, \\
    \nabla \cdot \mathbf{u} &= 0, \quad \text{in } \mathbb{R}^2_+, \\
    \mathbf{u}(x, 0) &= \mathbf{0},
\end{aligned}
\]
where $\mathcal{T}$ represents the elastic tension. The filament is advected by the fluid velocity, leading to the kinematic condition:
\begin{equation}\label{eq:kinematic_intro}
    \partial_t \mathbf{X}(t,s) = \mathbf{u}(\mathbf{X}(t,s), t).
\end{equation}
In this paper, we study the \emph{Anchored Peskin Problem} (APP), wherein the endpoints of the filament are strictly attached to the impermeable boundary. Specifically, we impose the stationary Dirichlet boundary conditions $\mathbf{X}(t,0) = \mathbf{e}_1$ and $\mathbf{X}(t,\pi) = -\mathbf{e}_1$.

This geometric configuration serves as a mathematical abstraction for microscopic swimmers and biological appendages, such as cilia and flagella, which are structurally tethered to a massive cell body or tissue wall \cite{dillon1995microscale, fauci1995sperm, lauga2009hydrodynamics}. While the macroscopic effects of these tethered active filaments have been heavily studied computationally, the rigorous mathematical analysis of the resulting non-local, non-linear partial differential equation \eqref{eq:kinematic_intro} near the anchoring points poses new types of analytical and numerical challenges.
\subsection{Mathematical Context and Challenges}

The well-posedness of the 1D Peskin problem in domains without physical boundaries (namely $\mathbb{R}^2$ and $\mathbb{T}^2$) has seen significant recent activity \cite{mori2019well, lin2019solvability}. Further work by  \cite{garcia2020peskin} established that the free-space IB formulation can be recast as a quasilinear parabolic evolution equation, sharpened the regularity theory using critical Wiener-algebra spaces, and obtained global well-posedness for medium-size data. The critical mechanism in these works is the extraction of the principal linear operator, which acts as the fractional Laplacian $(-\Delta)^{1/2}$ on the filament configuration. The resulting parabolicity yields local well-posedness and instantaneous smoothing.

The introduction of the rigid boundary $\partial \mathbb{R}^2_+$ fundamentally alters the functional analytic landscape. To satisfy the no-slip boundary condition $\mathbf{u}(x,0) = \mathbf{0}$, the free-space Stokeslet must be augmented by an image system. While classical formulations rely on the algebraically cumbersome Blake tensor \cite{blake1971note}, we utilize the boundary-symmetric complex-variable formulation of Gimbutas, Greengard, and Veerapaneni \cite{gimbutas2015simple}. Under this framework, the evolution equation \eqref{eq:kinematic_intro} takes the form of a singular integral equation:
\begin{equation}
    \partial_t \mathbf{X}(t,s) = \int_0^\pi \mathbf{G}_{GGV}(\mathbf{X}(t,s), \mathbf{X}(t,s')) \partial_{s'} \left( \mathcal{T}(\mathbf{X}) \partial_{s'} \mathbf{X} \right) ds'.
\end{equation}

When the curve is strictly bounded away from the wall, the reflection part of $\mathbf{G}_{GGV}$ acts as a smooth, lower-order perturbation. However, in the anchored problem, $\mathbf{X}(t,s)$ physically intersects the wall at $s=0$ and $s=\pi$. As $s, s' \to 0$ or $\pi$, the curve interacts intimately with its own hydrodynamic reflection. The image kernels lose their smoothness, producing boundary singularities that match the exact order of the principal free-space fractional Laplacian. The loss of translation invariance precludes the use of standard Fourier multiplier theorems, and the boundary-reflected kernels threaten to destroy the H\"{o}lder continuity of the fluid velocity field at the anchor points.

The analytic difficulty of resolving non-local contour integrals for curves intersecting rigid boundaries represents a major hurdle in fluid mechanics. A notable analog to the Peskin problem is the Muskat problem on the half-plane, which models the interface between two fluids of differing densities above an impermeable bottom. Standard energy estimates for the Muskat interface blow up as the curve approaches the solid boundary. In a recent sequence of papers, Zlatoš \cite{zlatos2024muskat1, zlatos2024muskat2} resolved the local well-posedness of the half-plane Muskat problem even when the interface contacts the boundary. Zlato\v{s}'s analysis  relies on identifying delicate cancellations between the direct and reflected kernels and formulating the problem in specific weighted spaces. Our approach to the APP shares this approach: identifying the exact boundary-compatible principal operator and controlling the non-local image remainder via carefully weighted function spaces.

\subsection{Approach and Main Results}

To overcome the singularities at the anchor points, we conduct a rigorous asymptotic decomposition of the GGV tensor evaluated along the configuration curve. We demonstrate that the leading-order hydrodynamic interaction, including the singular reflections at the boundaries, precisely reconstructs the Dirichlet-to-Neumann operator. 

Specifically, we decouple the evolution into a stiff principal linear operator $\mathcal{L}_D$ and a non-local nonlinear remainder $\mathbf{R}(\mathbf{X})$:
\begin{equation}
    \partial_t \mathbf{X} = \mathcal{L}_D \mathbf{X} + \mathbf{R}(\mathbf{X}), \quad \text{where} \quad \mathcal{L}_D = -\frac{1}{4}(-\Delta_D)^{1/2}.
\end{equation}
Because the fractional Dirichlet operator generates an analytic semigroup, we can frame the problem via a Duhamel integral formulation. However, the classical Hölder spaces $C^{1,\alpha}([0,\pi])$ are incompatible with this semigroup approach at $t=0$, and the boundary reflections further induce logarithmic derivative growth near the anchors. To close the estimates, we deploy \emph{weighted little-H\"{o}lder spaces} $h^{1,\alpha}_\beta([0,\pi])$, which enforce strong continuity of the semigroup while providing the necessary spatial weights to absorb the behavior of the reflected kernels.

Our main results are organized as follows:

\begin{enumerate}
    \item \textbf{Extraction of the Parabolic Structure and Remainder Estimates:}
    We isolate the principal linear operator $\mathcal{L}_D$ and prove that the remaining non-local interactions (both direct and reflected) form a Fr\'{e}chet-differentiable mapping that exhibits subcritical Calder\'{o}n-Zygmund type bounds from $h^{1,\alpha}_\beta$ into $h^{0,\alpha+\gamma}_{-\beta}$. 

    \item \textbf{Local Well-Posedness:}
    By establishing the strict contraction of the Duhamel operator in the weighted little-H\"{o}lder topology, we prove the existence and uniqueness of a local-in-time strong solution to the Anchored Peskin Problem (Theorem \ref{thm:local_existence} and Theorem \ref{thm:strong_solution}).

    \item \textbf{Instantaneous Global Smoothing:}
    In classical treatments of free-space non-local PDEs, higher regularity is obtained via iterated integration by parts against hyper-singular kernels to extract derivative regularity. Instead, following \cite{mori2019well} we develop an algebraic space-time bootstrap. Using the maximal regularity of analytic semigroups and the Fr\'{e}chet differentiability of the remainder, we iterate between temporal differentiation and the elliptic inversion of $\mathcal{L}_D$. We show this implies the filament is instantly smooth, $\mathbf{X} \in C^\infty((0,T_*] \times (0,\pi))$, (Theorem~\ref{thm:infinite_smoothness_full}). The smoothness of the filament in turn induces a classical solution of the half-space Stokes IB system in the fluid domain (Corollary~\ref{cor:classical_stokes}).

    \item \textbf{Numerical Simulation:}
    Finally, we propose a numerical scheme that closely respects the PDE's mathematical structure. By applying an odd reflection to the perturbation from the linear equilibrium, we map the anchored filament to a periodic domain, allowing the exact inversion of $\mathcal{L}_D$ via pseudo-spectral methods. Numerical experiments demonstrate the generic convergence of anchored filaments to arcs of circles, with the enclosed area conserved to spatial-quadrature accuracy.
\end{enumerate}

\subsection{Open Directions and Future Work}

The rigorous resolution of the Anchored Peskin Problem on the half-plane opens several avenues for further mathematical and computational investigation. We highlight five:

{\bf Critical regularity.}
While the weighted little-H\"{o}lder spaces $h^{1,\alpha}_\beta([0,\pi])$ provide a sufficient topological framework to close the contraction mapping and absorb the logarithmic boundary singularities, they may not be optimal. In the free-space setting, Garcia-Juarez, Mori, and Strain~\cite{garcia2020peskin} captured the exact critical functional spaces for the Peskin problem, working in Wiener algebras and critical Sobolev spaces and exploiting the translation invariance and exact Fourier multiplier structure of the free-space fractional Laplacian. The introduction of the rigid boundary breaks this translation invariance, precluding direct application of classical Fourier techniques. Adapting the optimal critical-space framework of \cite{garcia2020peskin} to the boundary-reflected IB equations---particularly capturing the precise endpoint behavior of the fractional Dirichlet operator in appropriately weighted function spaces---is a natural next step.

{\bf Global existence.}
The present work establishes local-in-time strong solutions up to a time $T_*$ dictated by the initial proximity to the geometric constraint set. A natural continuation is the establishment of global-in-time well-posedness. In the free-space setting, global existence for small data relies on exploiting the energy dissipation of the surrounding Stokes fluid. For the anchored problem, proving global existence requires showing that the strongly dissipative fractional Laplacian permanently suppresses the nonlinear remainder, preventing the chord-arc ratio $|\mathbf{X}|_*$ from collapsing and ensuring the filament never self-intersects over infinite time.

{\bf Boundary tangency.}
Our geometric constraint set $\mathcal{O}^{M,m}_\sigma$ enforces that the filament intersects the impermeable wall at a non-zero angle. If the active filament were to dynamically evolve such that it becomes perfectly tangent to the boundary at the anchor points (an angle of incidence approaching zero), the cancellation properties between the direct Stokeslet and the reflected image system could severely degenerate. Whether the non-local hydrodynamics prevent boundary tangency, or instead lead to a finite-time blow-up of the fluid velocity gradient, is reminiscent of contact-angle dynamics in free-boundary problems and remains open.

{\bf Singly anchored microswimmers.}
The doubly anchored filament serves as the foundational boundary-value model. Biologically accurate microswimmers (such as spermatozoa or tethered bacteria) typically feature a single anchored end attached to the cell body, with a free-floating tail. Mathematically, this transitions the system from a pure Dirichlet boundary value problem to a mixed Dirichlet-Neumann (or free-stress) boundary value problem. Re-evaluating the principal linear operator under these mixed boundary conditions---and adapting the numerical spectral reflection scheme to account for the non-periodic free end---represents a vital step toward the rigorous analysis of true biological microswimmers. Related free-boundary formulations for immersed filaments in higher-dimensional Stokes flow have been studied recently by Ohm~\cite{ohm2025freeboundaryproblemimmersed}.

\section{Elastic Filament Anchored to a Wall}

We consider an elastic, flexible filament anchored to a rigid wall at two points $(1,0)$ and $(-1,0)$ (Figure~\ref{fig:filament-wall}).  
The filament is parametrized by $s \in [0,\pi]$ and described by
\[
\mX(\cdot,t) \subset \mathbb{R}^2, \qquad \mX(0,t) = (1,0), \quad \mX(\pi,t) = (-1,0).
\]

\begin{figure}[h]
    \centering
    \includegraphics[width=0.8\linewidth]{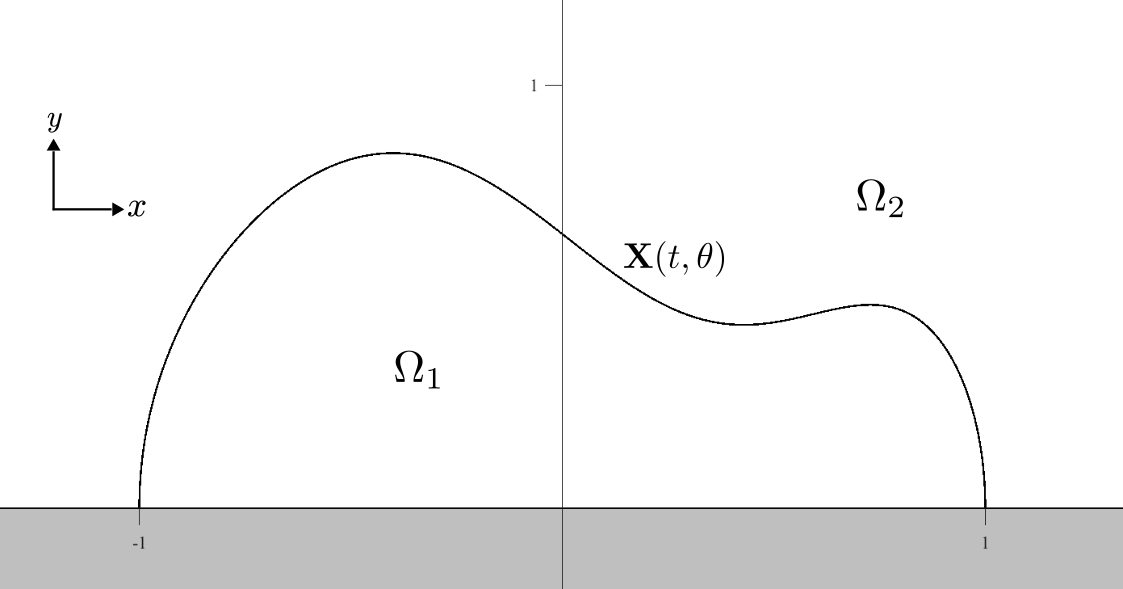}
    \caption{Schematic of an elastic filament anchored at two points on a rigid wall.}
    \label{fig:filament-wall}
\end{figure}

The rigorous study of this problem in $\mathbb{R}^2$ was initiated in \cite{lin2019solvability,mori2019well}, with recent progress on tension determination given in \cite{kuo2023tensiondeterminationprobleminextensible}.

\subsection{Peskin Problem in $\mathbb{R}_+^2$}
\label{ss:peskin_R2_plus}

We now formulate the Peskin problem for a one-dimensional membrane anchored on a wall in the half-space $\mathbb{R}_+^2$.  
The half-space is divided into two domains, $\Omega_1$ and $\Omega_2$, separated by a curve $\Gamma$.  
In each domain, the fluid satisfies the Stokes equations:
\begin{align*}
-\nabla P + \Delta u &= 0, \\
\operatorname{div} u &= 0.
\end{align*}

For any quantity $w$ defined on $\Omega \setminus \Gamma$, we denote the jump across $\Gamma$ by
\[
\jump{w} = \left. w \right|_{\Gamma_i} - \left. w \right|_{\Gamma_e},
\]
where the subscripts $i$ and $e$ indicate the interior and exterior limits approaching $\Gamma$.

The boundary and half-space conditions for the anchored membrane are
\begin{align*}
\jump{u}_{\Gamma} &= 0, \\
\jump{\sigma \mn}_{\Gamma} &= f|_{\Gamma}, \\
u|_{\partial \mathbb{R}_+^2} &= 0,
\end{align*}
where $f$ is the restorative force generated by the membrane.  

For an incompressible fluid at equilibrium (where $u \equiv 0$), the stress tensor reduces to $\sigma = -p \, \mathbb{I}$.  
The membrane response $f(\mX)$ can be linear (Hookean) or nonlinear:
\[
f(\mX) =
\begin{cases}
\partial_s^2 \mX, & \text{linear Hookean},\\[1mm]
\partial_s \big( \mathcal{T}(|\partial_s \mX|)\, \tau \big), & \text{nonlinear},
\end{cases}
\]
where $\tau = \frac{\partial_s \mX}{|\partial_s \mX|}$.  
For a Hookean filament, $\mathcal{T}(p) = p$.

With Hookean response and the membrane parametrized as $\mX(s,t)$ satisfying
\[
\mX(0,t) = (1,0), \quad \mX(\pi,t) = (-1,0),
\]
the Peskin problem reduces to
\[
\partial_t \mX(s,t) = \text{P.V.} \int_0^\pi S^+(\mX(s,t),\mX(s',t))\, \partial^2_{s'} \mX(s',t)\, ds',
\]
where $S^+$ is the half-space Stokeslet operator. The principal value is required because the logarithmic singularity of $S^+$ along the diagonal $s = s'$ makes the integral conditionally convergent; the symmetric truncation $|s - s'| > \varepsilon$ is taken and $\varepsilon \to 0$.

\subsection{Stokeslet in the Half-Space $\mathbb{R}_+^2$}

To make the problem concrete, we require the half-space Stokeslet for vanishing velocity on the boundary.  
Several constructions exist, including Blake's classical method \cite{blake1971note}.  
Here, we follow Gimbutas--Greengard--Veerapaneni \cite{gimbutas2015simple}, which yields a Green's function with a lower-order singularity.  

The operator is
\begin{equation}\label{e:halfspacestokeslet}
\begin{split}
S^+[\mf](\mx,\my)
&= \frac{1}{4\pi} \Big[
-\log|\mx - \my| + \log|\mx - \my^r|
\Big] \mf \\
&\quad + \frac{1}{4\pi} \Big[
\frac{(\mx - \my)\otimes (\mx - \my)}{|\mx - \my|^2} \mf
- \frac{(\mx - \my)\otimes (\mx - \my^r)}{|\mx - \my^r|^2} \mf^r
\Big] \\
&\quad - \frac{\mx_2 \mf_2}{2\pi} \frac{\mx - \my^r}{|\mx - \my^r|^2}
- \frac{\mx_2 \my_2}{2\pi} \Big[
\frac{\mf^r}{|\mx - \my^r|^2} - 2\frac{(\mx - \my^r)\otimes (\mx - \my^r)}{|\mx - \my^r|^4}\mf^r
\Big].
\end{split}
\end{equation}
A derivation of \eqref{e:halfspacestokeslet}, along with the associated pressure $P^+[\mf](\mx,\my)$, is provided in Appendix~\ref{appendix:stokeslet}.  
Although the expression is complicated, the principal part is the logarithmic term; other terms are lower-order or involve $\mx - \my^r$ in the denominator.  
This principal operator defines the semi-group used for well-posedness and numerics.

\subsection{Conservation of Area}

A key consistency check is conservation of area due to the divergence-free condition.  
Let $\Omega_t$ denote the bounded region enclosed between the filament $\Gamma(t)$ and the wall segment $\{(x,0) : -1 \le x \le 1\}$ at time $t$. Its boundary $\partial\Omega_t$ consists of the filament (on which the kinematic condition $u = \partial_t \mX$ holds) and the wall segment (on which $u = 0$).
By Reynolds' transport theorem and the divergence theorem,
\[
\frac{d}{dt}\operatorname{Area}(\Omega_t) 
= \frac{d}{dt} \int_{\Omega_t} 1\, dx 
= \int_{\partial \Omega_t} u \cdot \nu\, d\ell
= \int_{\Omega_t} \operatorname{div} u\, dx = 0.
\]
Hence, the area is conserved.

\subsection{Nonlinear Hooke Laws}
\label{ss:nonlinear_hooke}

More generally, consider a nonlinear elastic response of the form
\[
f(\mX)
= \partial_s \bigl(\mathcal{T}(|\partial_s \mX|)\tau\bigr),
\qquad
\tau=\frac{\partial_s \mX}{|\partial_s \mX|}.
\]
Equivalently,
\[
f(\mX)
= \partial_s\!\left(
|\partial_s \mX|^{-1}\mathcal{T}(|\partial_s \mX|)
\,\partial_s \mX
\right).
\]
If $\mathcal{T}(s)=s^p$, then
\[
f(\mX)
= \partial_s^2 \mX\,|\partial_s \mX|^{p-1}
+ \partial_s \mX\,\partial_s |\partial_s \mX|^{p-1}.
\]
Imposing the equilibrium condition
$f(\mX)\cdot\partial_s \mX=0$ gives
\[
0
= \partial_s \mX\cdot \partial_s^2 \mX\,|\partial_s \mX|^{p-1}
+ |\partial_s \mX|^2\partial_s |\partial_s \mX|^{p-1}
= {p \over p+1} \partial_s |\partial_s \mX|^{p+1}.
\]
Hence $|\partial_s \mX|$ is constant, so we may reparametrize to arc-length. In arc-length coordinates, the unit tangent is $\tau = \partial_s \mX$ and the curvature vector is $\partial_s^2 \mX = \kappa\,\hat{n}$, where $\hat{n}$ is the unit normal. The equilibrium balance $f(\mX) = -(p_+ - p_-)\hat{n}$ then forces $\kappa |\partial_s \mX|^{p-1} = -(p_+ - p_-)$, which is constant along the curve (as the pressure jump is uniform for a curve in Stokes flow). Thus the membrane has constant curvature and must be an arc of a circle, independent of the power $p$.

\section{Equilibria}\label{s:equilibria}

We seek equilibrium configurations for which the fluid velocity vanishes identically,
\[
u(\mX(s,t)) \equiv 0 \qquad \text{in } \mathbb{R}^2_+ \setminus \Gamma .
\]
In this case, the jump condition across the membrane reduces to a balance between the
elastic force and the pressure jump:
\begin{align*}
f(\mX)
&= \jump{\sigma \mn}
 = - (p_+ - p_-)\mn
 = - (p_+ - p_-)\frac{(\partial_s \mX)^\perp}{|\partial_s \mX|}.
\end{align*}
Taking the dot product with the tangent vector $\partial_s \mX$ yields
\begin{align*}
f(\mX)\cdot \partial_s \mX
= - (p_+ - p_-)(\partial_s \mX)^\perp \cdot \partial_s \mX
= 0,
\end{align*}
so the elastic force is purely normal to the membrane.

\begin{lemma}[Hookean equilibria are circular arcs]\label{lem:hookean_equilibria}
For a Hookean elastic response
\[
f(\mX) = \partial_s^2 \mX,
\]
the only equilibria are arcs of circles that connect to the anchoring points. 
\end{lemma}

\begin{proof}
The orthogonality condition, together with the form of $f(\mX)$, implies
\[
0 = \partial_s^2 \mX \cdot \partial_s \mX
= \frac{1}{2}\partial_s |\partial_s \mX|^2.
\]
Thus $|\partial_s \mX|$ is constant, so the parametrization is proportional to arc-length.
Reparametrizing to arc-length (so $|\partial_s \mX| = 1$), the Frenet--Serret formulas give
$\partial_s^2 \mX = \kappa\,\hat{n}$, where $\kappa$ is the signed curvature and $\hat{n}$
is the unit normal. Hence $f(\mX) = \kappa\,\hat{n}$ is automatically normal, consistent with
the equilibrium condition. The Young--Laplace balance
$f(\mX) = -(p_+ - p_-)\hat{n}$ then gives $\kappa = -(p_+ - p_-)$.
Since the pressure jump is constant along the curve for a Stokes flow equilibrium,
$\kappa$ is constant and $\mX$ must be an arc of a circle (Figure~\ref{fig:equilibrium-arc}).
The specific arc is determined by the enclosed area and the anchoring constraints.
\end{proof}

\noindent The Hookean case stated above is the special case $p = 1$ of the more general nonlinear-elastic computation in Section~\ref{ss:nonlinear_hooke}, where the same circular-arc conclusion was shown to hold for any power-law tension $T(s) = s^p$.

\begin{figure}[h]
    \centering
    \includegraphics[width=0.9\linewidth]{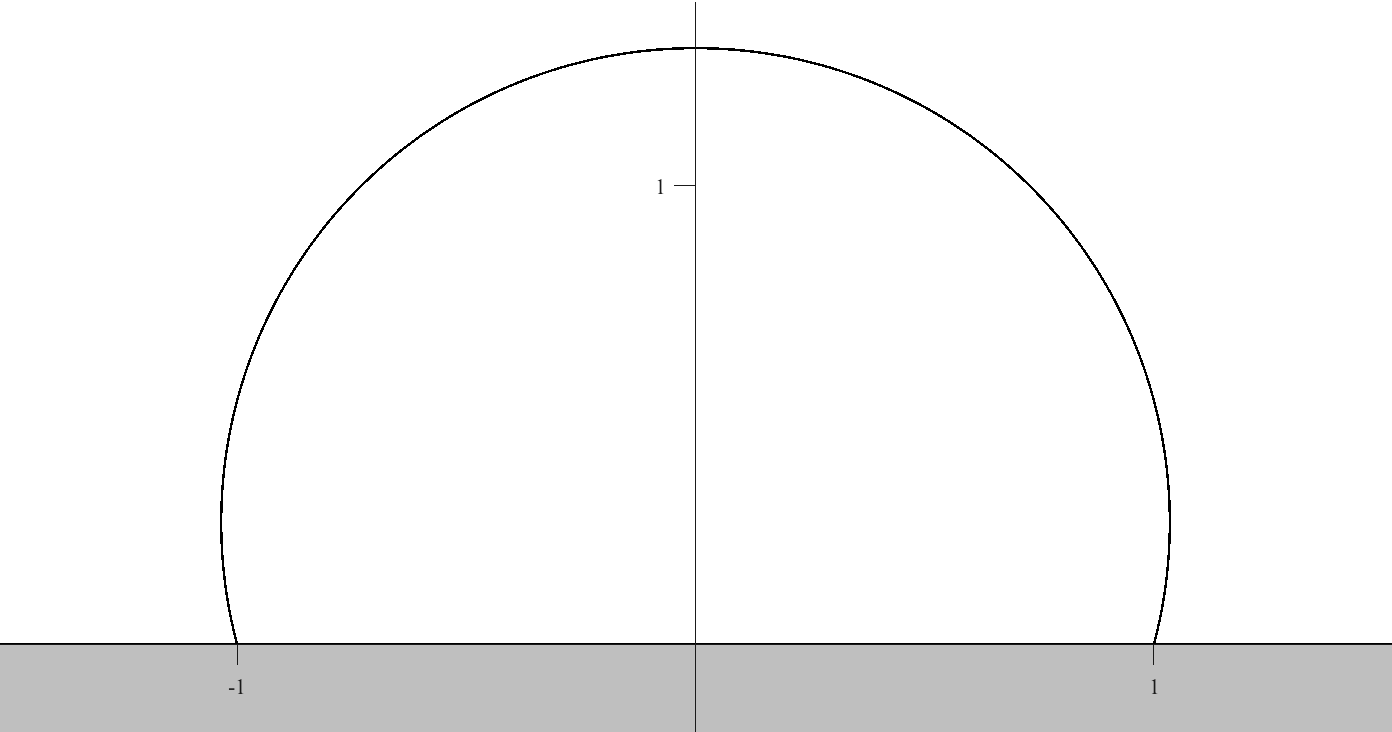}
    \caption{Equilibrium configuration: a circular arc anchored at two points on the wall.}
    \label{fig:equilibrium-arc}
\end{figure}


Assume the equilibrium arc is part of a circle centered at $(0,c)$, $c\in\mathbb{R}$,
passing through the anchoring points $(\pm1,0)$.
The radius is
\[
r = \sqrt{c^2+1},
\]
and the maximum height of the arc is $h=r+c>0$.
A convenient parametrization of the membrane, especially in our numerical calculations, is
\begin{align*}
\mX(s)
=
\begin{cases}
r \exp\!\left(
i\left[
\left(2-\frac{2}{\pi}\arctan\frac{1}{c}\right)s
+ \arctan\frac{1}{c}-\frac{\pi}{2}
\right]
\right)+(0,c), & c>0, \\
r \exp\!\left(
i  s
\right), & c = 0, \\
r \exp\!\left(
i\left[
-\frac{2}{\pi}\arctan\frac{1}{c}\,s
+ \arctan\frac{1}{c}+\frac{\pi}{2}
\right]
\right)+(0,c), & c<0.
\end{cases}
\end{align*}
The enclosed area as a function of $c$ is
\[
\operatorname{Area}(c)=
\begin{cases}
\bigl(\pi-\arctan(1/c)\bigr)(c^2+1)+c, & c>0, \\[1mm]
\pi/2, & c=0, \\[1mm]
-\arctan(1/c)(c^2+1)+c, & c<0.
\end{cases}
\]


\section{Linearization and the semigroup propagator}
\label{s:linearization}

In this section we analyze the structure of the principal linearized operator
arising from the Peskin formulation, with particular emphasis on boundary
effects. We begin with the semi-infinite geometry and then pass to the
half-circle, which corresponds to Dirichlet boundary conditions at both ends.
Throughout, we consider perturbations of the form
\[
\mX_\e(s,t)=\mX_0(s)+\e\,\mX_1(s,t),
\]
and compute the leading-order linear dynamics for $\mX_1$.

\subsection{Semi-infinite linearization}

We begin with the steady configuration
\begin{equation}\label{e:semiinfinitesteady}
\mX_0(s)=is,\qquad s\ge0,
\end{equation}
representing a straight filament anchored at the origin.

\begin{lemma}[Principal linear operator on the half-line]
\label{lem:semiinf-operator}
Let $u(s,t)$ denote the scalar perturbation along the filament.
The principal linearization about \eqref{e:semiinfinitesteady} is given by
\begin{equation}\label{e:semiinf-evol}
\partial_t u(s,t)
=
-\frac{1}{4\pi}P.V.\int_0^\infty
\left(\frac{1}{s-s'}+\frac{1}{s+s'}\right)\partial_{s'}u(s',t)\,ds',
\qquad s>0.
\end{equation}
\end{lemma}

\begin{proof}
Since $S^+$ satisfies the no-slip boundary condition $S^+(\mathbf{x},\mathbf{y})\big|_{x_2=0}=\mathbf{0}$, integration by parts in $s'$ (boundary terms vanish) yields
\[
\int_0^\infty S^+[\partial_{s'}\mf](\mX_\e,\mX'_\e)\,ds'
=
-\int_0^\infty \partial_{s'}S^+(\mX_\e,\mX'_\e)\mf(s')\,ds'.
\]
Linearizing the logarithmic singularity of the Green's function about
$\mX_0(s)=is$ gives
\[
\log|\mX_0(s)-\mX_0(s')|=\log|s-s'|,
\qquad
\log|\mX_0(s)-\mX_0^r(s')|=\log|s+s'|.
\]
Differentiating and collecting the leading-order terms produces
\eqref{e:semiinf-evol}.
\end{proof}

\begin{lemma}[Odd symmetry and reduction to the line]
\label{lem:odd-extension}
If the initial data satisfy $u(s,0)=-u(-s,0)$, then the solution of
\eqref{e:semiinf-evol} obeys
\[
\partial_t u(-s,t)=-\partial_t u(s,t),
\qquad s>0.
\]
Consequently, $\partial_s u$ is even and
\begin{equation}\label{e:semiinf-hilbert}
\partial_t u(s,t)
=
-\frac{1}{4\pi}P.V.\int_{-\infty}^{\infty}
\frac{\partial_{s'}u(s',t)}{s-s'}\,ds'
=
-\frac14\,\mathcal{H}(\partial_s u)(s,t),
\qquad s\neq0.
\end{equation}
\end{lemma}

\begin{proof}
A direct substitution of $-s$ into \eqref{e:semiinf-evol} yields
\[
\partial_t u(-s,t)
=
\frac{1}{4\pi}P.V.\int_0^\infty
\left(\frac{1}{s+s'}+\frac{1}{s-s'}\right)\partial_{s'}u(s',t)\,ds'
=
-\partial_t u(s,t).
\]
Evenness of $\partial_s u$ allows the two half-line integrals to be combined
into the full-line Hilbert transform, yielding
\eqref{e:semiinf-hilbert}.
\end{proof}

\begin{remark}
The extension to $s=0$ gives $\partial_t u(0,t)=0$, so the endpoint value is
preserved. The operator in \eqref{e:semiinf-hilbert} is the classical
Dirichlet-to-Neumann operator on the half-line.
\end{remark}

\subsection{Half-circle and Dirichlet operator}

We now consider the geometry
\[
\mX_0(s)=e^{is},\qquad s\in[0,\pi],
\]
with reflected curve $\mX_0^r(s)=e^{-is}$. Singularities now occur in the
interior of the interval, so principal value integration must be defined
carefully.

\begin{definition}[Principal value on a closed interval]
Let $g:[0,\pi]^2\to\R\cup\{\pm\infty\}$. For $s\in[0,\pi]$, define
\[
P.V.\int_0^\pi g(s,s')\,ds'
=
\begin{cases}
\displaystyle\int_0^\pi g(s,s')\,ds', & s\in\{0,\pi\},\\[1em]
\displaystyle\lim_{\e\to0}
\int_{[0,\pi]\setminus(s-\e,s+\e)}g(s,s')\,ds',
& s\in(0,\pi),
\end{cases}
\]
whenever the limit exists.
\end{definition}

\begin{lemma}[Dirichlet linear operator]
\label{lem:LD-def}
The principal linear operator on $[0,\pi]$ is
\begin{equation}\label{e:LD-def}
\mathcal{L}_D[u](s)
=
-\frac{1}{8\pi}P.V.\int_0^\pi
\left[
\cot\!\left(\frac{s-s'}{2}\right)
+
\cot\!\left(\frac{s+s'}{2}\right)
\right]\partial_{s'}u(s')\,ds'.
\end{equation}
Moreover,
\begin{equation}\label{e:LD-hilbert}
\mathcal{L}_D[u]
=
-\frac14\,\mathcal{H}_D(\partial_s u),
\end{equation}
where $\mathcal{H}_D$ denotes the Dirichlet Hilbert transform.
\end{lemma}

\begin{proof}
A direct computation gives
\[
\partial_{s'}\log\frac{|\mX_0(s)-\mX_0(s')|}{|\mX_0(s)-\mX_0^r(s')|}
=
-\frac12\cot\!\left(\frac{s-s'}{2}\right)
-\frac12\cot\!\left(\frac{s+s'}{2}\right).
\]
Substitution into the definition of the linearized operator yields
\eqref{e:LD-def}, and \eqref{e:LD-hilbert} follows by definition of
$\mathcal{H}_D$.
\end{proof}

\begin{lemma}[Boundary invariance and kernel]
\label{lem:LD-kernel}
For any $u\in C^1([0,\pi])$,
\[
\mathcal{L}_D[u](0)=\mathcal{L}_D[u](\pi)=0.
\]
Moreover,
\[
\mathcal{L}_D[c+ds]=0
\qquad\text{for all }c,d\in\R.
\]
\end{lemma}

\begin{proof}
Using the identities
\[
\cot\!\left(-\frac{s'}{2}\right)+\cot\!\left(\frac{s'}{2}\right)=0,
\qquad
\cot\!\left(\frac{\pi-s'}{2}\right)+\cot\!\left(\frac{\pi+s'}{2}\right)=0,
\]
we obtain $\mathcal{L}_D[u](0)=\mathcal{L}_D[u](\pi)=0$.
The second claim follows by integrating the kernel explicitly and observing
that the principal value vanishes.
\end{proof}

\subsection{Odd extension and Poisson kernel formulation}

Define
\begin{equation}\label{e:scalarelldef}
\ell(s)=u_0(0)+\frac{u_0(\pi)-u_0(0)}{\pi}s,
\end{equation}
and set
\[
w(s)=
\begin{cases}
u(s)-\ell(s), & s\in[0,\pi],\\
-u(2\pi-s)+\ell(2\pi-s), & s\in[\pi,2\pi].
\end{cases}
\]

\begin{lemma}[Reduction to the circle]
\label{lem:circle-reduction}
The function $w$ is $2\pi$-periodic, odd with respect to $s=\pi$, and satisfies
\[
\partial_t w
=
-\frac14\,\mathcal{H}(\partial_s w)
=
-\frac14\,\Lambda w,
\]
where $\mathcal{H}$ is the periodic Hilbert transform and $\Lambda=|\partial_s|$.
\end{lemma}

\begin{proof}
Periodicity and oddness around $s=\pi$ are immediate from the construction. For the operator identity, we compute the action of $\mathcal{L}_D$ on the half-interval and re-express it as a periodic principal-value integral. Recall from \eqref{e:LD-def} that
\[
\mathcal{L}_D[u](s)
= -\frac{1}{8\pi}\,\mathrm{P.V.}\!\int_0^\pi
\left[\cot\!\left(\tfrac{s-s'}{2}\right)+\cot\!\left(\tfrac{s+s'}{2}\right)\right]
\partial_{s'}u(s')\,ds'.
\]
On the doubled interval $[0,2\pi]$, write $w$ for the odd extension. Splitting the periodic Hilbert transform into the two halves $[0,\pi]$ and $[\pi,2\pi]$,
\[
\mathcal{H}[\partial_s w](s)
= \frac{1}{2\pi}\,\mathrm{P.V.}\!\int_0^{2\pi}
\cot\!\left(\tfrac{s-\sigma}{2}\right) \partial_\sigma w(\sigma)\,d\sigma
= \frac{1}{2\pi}(I+II),
\]
where $I=\mathrm{P.V.}\int_0^\pi \cot((s-\sigma)/2)\,\partial_\sigma w(\sigma)\,d\sigma$ and $II$ is the corresponding integral over $[\pi,2\pi]$. In $II$ substitute $\sigma = 2\pi - \sigma'$. Since $w$ is odd around $\pi$, $\partial_\sigma w(2\pi-\sigma') = \partial_{\sigma'}w(\sigma')$ (the derivative is even), and using $\pi$-periodicity of cotangent,
\[
\cot\!\left(\tfrac{s-(2\pi-\sigma')}{2}\right)
= \cot\!\left(\tfrac{s+\sigma'}{2}-\pi\right)
= \cot\!\left(\tfrac{s+\sigma'}{2}\right).
\]
Therefore $II = \mathrm{P.V.}\int_0^\pi \cot((s+\sigma')/2)\,\partial_{\sigma'}w(\sigma')\,d\sigma'$, and combining,
\[
\mathcal{H}[\partial_s w](s)
= \frac{1}{2\pi}\,\mathrm{P.V.}\!\int_0^\pi
\left[\cot\!\left(\tfrac{s-\sigma}{2}\right)+\cot\!\left(\tfrac{s+\sigma}{2}\right)\right]
\partial_\sigma w(\sigma)\,d\sigma
= -4\,\mathcal{L}_D[w](s).
\]
Hence $\partial_t w = \mathcal{L}_D w = -\tfrac14\,\mathcal{H}(\partial_s w)$, and the identification with $\Lambda$ follows from the Fourier multipliers $\widehat{\mathcal{H}\,\partial_s} = |k|$.
\end{proof}

\begin{proposition}[Semigroup decay and Poisson kernel representation]
\label{prop:LD-decay}
Fix an integer $k\ge0$. Let $u_0\in C^{k,\alpha}([0,\pi])$ and let $u$ solve
$\partial_t u=\mathcal{L}_D u$ with fixed endpoint values.
Then
\[
u(s,t)-\ell(s)
=
\int_{-\pi}^{\pi}
P^{\mathrm{circ}}_t(s-s')\,\widetilde{(u_0-\ell)}(s')\,ds',
\]
where $P^{\mathrm{circ}}_t$ is the Poisson kernel for
$e^{-\frac14\Lambda t}$ and $\widetilde{\cdot}$ denotes odd extension.
Moreover,
\[
\|u(t)-\ell\|_{C^k}\le C\|u_0-\ell\|_{C^k},
\qquad
\|u(t)-\ell\|_{C^{k+1}}\le \frac{C}{1+t}\|u_0-\ell\|_{C^k}.
\]
\end{proposition}

\begin{proof}
\textbf{Step 1 (Reduction to homogeneous data).}
Define
\[
\ell(s)=u_0(0)+\frac{u_0(\pi)-u_0(0)}{\pi}s,
\qquad
w(s,t)=u(s,t)-\ell(s).
\]
By Lemma~\ref{lem:LD-kernel}, $\ell\in\ker\mathcal{L}_D$, hence $w$ satisfies
\[
\partial_t w=\mathcal{L}_D w,
\qquad
w(0,t)=w(\pi,t)=0.
\]
\textbf{Step 2 (Odd extension to the circle).}
We extend $w$ to $[0,2\pi]$ by odd reflection:
\[
w(s,t)=
\begin{cases}
u(s,t)-\ell(s), & 0\le s\le\pi,\\
-\,u(2\pi-s,t)+\ell(2\pi-s), & \pi\le s\le2\pi.
\end{cases}
\]
This extension is continuous, $2\pi$--periodic, and satisfies
\[
w(2\pi-s,t)=-w(s,t).
\]
In particular, $w(\cdot,t)\in C^{1,\alpha}_{\mathrm{per}}([0,2\pi])$ and
$\partial_s w$ is even.

By Lemma~\ref{lem:circle-reduction},
\[
\partial_t w
=
-\frac14\,\mathcal{H}(\partial_s w)
=
-\frac14\,\Lambda w,
\]
where $\Lambda=|\partial_s|$ is the Fourier multiplier with symbol $|k|$.
Thus
\[
w(\cdot,t)=e^{-\frac14\Lambda t}w_0,
\qquad
w_0(s)=\widetilde{(u_0-\ell)}(s).
\]

\textbf{Step 3 (Poisson kernel representation).}
Let $P^{\mathrm{circ}}_t$ denote the Poisson kernel on $\mathbb S^1$
associated with the semigroup $e^{-\frac14\Lambda t}$.
Then
\[
w(s,t)
=
\int_{-\pi}^{\pi}
P^{\mathrm{circ}}_t(s-s')\,w_0(s')\,ds',
\]
and therefore, restricting to $s\in[0,\pi]$,
\[
u(s,t)-\ell(s)
=
\int_{-\pi}^{\pi}
P^{\mathrm{circ}}_t(s-s')\,
\widetilde{(u_0-\ell)}(s')\,ds'.
\]

\textbf{Step 4 (Decay estimates).}
Since $P^{\mathrm{circ}}_t$ is a probability kernel and smooth for $t>0$,
standard convolution estimates yield
\[
\|w(t)\|_{C^k}
\le C\|w_0\|_{C^k},
\qquad t\ge0,
\]
where $C$ depends only on $k$.
Because $w=u-\ell$ on $[0,\pi]$, this proves
\[
\|u(t)-\ell\|_{C^k([0,\pi])}
\le C\|u_0-\ell\|_{C^k([0,\pi])}.
\]
Differentiating under the integral sign gives
\[
\partial_s w(s,t)
=
\int_{-\pi}^{\pi}
(\partial_s P^{\mathrm{circ}}_t)(s-s')\,w_0(s')\,ds'.
\]
Using the kernel comparison estimate (see, e.g., \cite{stein1970singular})
\[
|\partial_\theta P^{\mathrm{circ}}_t(\theta)|
\le C_0\,|\partial_x P^{\mathrm{line}}_t(\theta)|,
\]
which follows from the explicit formula $P^{\mathrm{circ}}_t(\theta) = \sum_{k\in\mathbb{Z}} e^{-|k|t/4}e^{ik\theta}$ and the pointwise domination by the half-space Poisson kernel,
together with the exact bound (see \cite{stein1970singular})
\[
\|\partial_x P^{\mathrm{line}}_t\|_{L^1(\mathbb R)}
=
\frac{2}{\pi t},
\]
we obtain
\[
\|\partial_s w(t)\|_{C^k}
\le
\frac{C}{1+t}\,\|w_0\|_{C^k}.
\]
Iterating this argument yields
\[
\|w(t)\|_{C^{k+1}}
\le
\frac{C}{1+t}\,\|w_0\|_{C^k}.
\]
Since $u(s,t)=w(s,t)+\ell(s)$ and $\ell$ is time-independent,
the estimates above imply
\[
\|u(t)-\ell\|_{C^k}
\le C\|u_0-\ell\|_{C^k},
\qquad
\|u(t)-\ell\|_{C^{k+1}}
\le \frac{C}{1+t}\|u_0-\ell\|_{C^k}.
\]
This completes the proof.
\end{proof}

\section{Evolution Equation and Remainder Structure}

In this section, we reformulate the Anchored Peskin Problem (APP) as a single evolution equation comprising a linear principal part and a nonlinear remainder. We then rigorously analyze the structure and boundary behavior of the remainder terms, which is essential for applying semigroup methods.

\subsection{Perturbative formulation}

We begin with the Anchored Peskin Problem written in its integral form
\[
\partial_t \mathbf{X}(s,t)
=
- \int_{0}^{\pi}
\partial_{s'} S^{+}(\mathbf{X}(s,t),\mathbf{X}(s',t))
\, \partial_{s'} \mathbf{X}(s',t)\, ds' .
\]
Since the half-space Stokeslet satisfies
\[
S^{+}(\mathbf{X}(s),\mathbf{X}(s')) = 0,
\qquad s \in \{0,\pi\},
\]
we may add and subtract $\partial_s \mathbf{X}(s)$ inside the integral. The added term contributes zero: integrating $\partial_{s'} S^+$ yields the boundary evaluation $\bigl[S^+(\mathbf{X}(s),\mathbf{X}(s'))\bigr]_{s'=0}^{s'=\pi} = \mathbf{0}$, since $S^+$ vanishes whenever either argument lies on the wall. This gives the equivalent form
\[
\partial_t \mathbf{X}(s,t)
=
- \int_{0}^{\pi}
\partial_{s'} S^{+}(\mathbf{X}(s),\mathbf{X}(s'))
\bigl(\partial_{s'} \mathbf{X}(s') - \partial_s \mathbf{X}(s)\bigr)\, ds' .
\]

We work in perturbative form, decomposing the dynamics into a linear principal operator and a nonlinear remainder,
\begin{equation}\label{eq:perturbative_form}
\partial_t \mathbf{X}
=
\mathcal{L}_D [\mathbf{X}] + \mathbf{R}(\mathbf{X}),
\end{equation}
subject to the anchored boundary conditions
\[
\mathbf{X}(0,t) = (1,0),
\qquad
\mathbf{X}(\pi,t) = (-1,0).
\]
Here
\[
\mathcal{L}_D = (\mathcal{L}_{D_x},\mathcal{L}_{D_y})
\]
denotes the diagonal linear operator obtained from the linearization in Section \ref{s:linearization}, with distinct boundary conditions for the $x$- and $y$-components. The nonlinear term $\mathbf{R}(\mathbf{X})$ collects all lower-order contributions.

This perturbative formulation admits the Duhamel representation
\begin{equation}\label{eq:Duhamel}
\mathbf{X}(t,s)
=
e^{\mathcal{L}_D t}\mathbf{X}_0(s)
+
\int_{0}^{t}
e^{\mathcal{L}_D (t-r)} \mathbf{R}(\mathbf{X}(r,s))\, dr .
\end{equation}
Our strategy is to establish existence and regularity of solutions via a fixed-point argument based on this formulation.

\subsection{Structure of the remainder terms}

To describe the nonlinear remainder explicitly, we employ the Stokeslet formulation \eqref{e:halfspacestokeslet} and introduce the difference operators:
\[
\Delta \mathbf{X} = \mathbf{X}(s) - \mathbf{X}(s'),
\qquad
\Delta^r \mathbf{X} = \mathbf{X}(s) - \mathbf{X}^r(s'),
\]
and similarly,
\[
\Delta \partial_s \mathbf{X} = \partial_s \mathbf{X}(s) - \partial_{s'} \mathbf{X}(s'),
\qquad
\Delta^r \partial_s \mathbf{X} = \partial_s \mathbf{X}(s) - \partial_{s'} \mathbf{X}^r(s').
\]
Here $\mathbf{X}^r$ denotes the reflected configuration. Note carefully that $\Delta^r \partial_s \mathbf{X}$ and $\Delta \partial_s \mathbf{X}^r$ are distinct objects. 

The nonlinear term $\mathbf{R}(\mathbf{X})$ can be decomposed as
\[
\mathbf{R}(\mathbf{X})
=
\mathbf{R}_1(\mathbf{X})
+
\mathbf{R}_2(\mathbf{X})
+
\mathbf{R}_3(\mathbf{X}),
\]
where
\begin{align*}
\mathbf{R}_1(\mathbf{X})
&=
\frac{1}{4\pi} \int_{0}^{\pi}
\left[
\frac{\langle \Delta \mathbf{X}, \partial_{s'} \mathbf{X} \rangle}
{|\Delta \mathbf{X}|^2}
-
\frac{1}{2}\cot\left(\frac{s-s'}{2}\right)
\right]
\Delta \partial_s \mathbf{X}\, ds' \\
&\quad
-
\frac{1}{4\pi} \int_{0}^{\pi}
\left[
\frac{\langle \Delta^r \mathbf{X}, \partial_{s'} \mathbf{X}^r \rangle}
{|\Delta^r \mathbf{X}|^2}
+
\frac{1}{2}\cot\left(\frac{s+s'}{2}\right)
\right]
\Delta \partial_s \mathbf{X}\, ds' ,
\\[0.5em]
\mathbf{R}_2(\mathbf{X})
&=
\frac{1}{4\pi}\int_{0}^{\pi}
\partial_{s'}\left(
\frac{\Delta \mathbf{X}\otimes \Delta \mathbf{X}}
{|\Delta \mathbf{X}|^2}
\right)
\Delta \partial_s \mathbf{X}\, ds' \\
&\quad
-
\frac{1}{4\pi}\int_{0}^{\pi}
\partial_{s'}\left(
\frac{\Delta \mathbf{X}\otimes \Delta^r \mathbf{X}}
{|\Delta^r \mathbf{X}|^2}
\right)
\Delta \partial_s \mathbf{X}^r\, ds' ,
\\[0.5em]
\mathbf{R}_3(\mathbf{X})
&=
-\frac{1}{2\pi}\int_{0}^{\pi}
\partial_{s'}\left(
\mathbf{X}_2(s)\frac{\Delta^r \mathbf{X}}{|\Delta^r \mathbf{X}|^2}
\right)
\Delta \partial_s \mathbf{X}_2\, ds' \\
&\quad
-
\frac{1}{2\pi}\int_{0}^{\pi}
\partial_{s'}\left(
\mathbf{X}_2(s)\mathbf{X}_2(s')
\frac{|\Delta^r \mathbf{X}|^2 I - 2\,\Delta^r \mathbf{X}\otimes \Delta^r \mathbf{X}}
{|\Delta^r \mathbf{X}|^4}
\right)
\Delta \partial_s \mathbf{X}^r\, ds' .
\end{align*}
We pause to comment on the structure of the contracted vector in each integral.  The IBP-with-subtract identity, applied to $-\int \partial_{s'}[K(s,s')\mathbf{f}(s')]\,ds'$ for a kernel matrix $K$ with $K(s,\cdot)$ vanishing at the endpoints, yields $+\int \partial_{s'}K(s,s')(\mathbf{f}(s')-\mathbf{f}(s))\,ds' = -\int \partial_{s'}K(s,s')\,\Delta\partial_s\mathbf{X}\,ds'$ when $K$ acts on $\mathbf{f} = \partial_{s'}\mathbf{X}$ directly.  When the integrand has the form $K(s,s')R\,\mathbf{f}$, where $R = \operatorname{diag}(1,-1)$ implements the reflection (so that the kernel is acting on $\mathbf{f}^r$), the same identity yields the contracted vector $R\,\Delta\partial_s\mathbf{X} = \Delta\partial_s\mathbf{X}^r$.  This is why the second integral of $\mathbf{R}_2$ and the second integral of $\mathbf{R}_3$, which derive from kernel components acting on $\mathbf{f}^r$, contract against $\Delta\partial_s\mathbf{X}^r$.

For both numerical and analytical purposes, it is necessary to expand $\mathbf{R}_2$ and $\mathbf{R}_3$ into fully perturbative forms. Expanding the derivatives in $\mathbf{R}_2$ yields
\begin{equation}\label{e:R2expanded}
\begin{split}
     \mathbf{R}_2(\mathbf{X})  
& = - \frac{1}{4\pi} \int_0^\pi
\frac{ \partial_{s'} \mathbf{X}(s') }{ |\Delta \mathbf{X}|^2} \langle \Delta \mathbf{X}, \Delta \partial_s \mathbf{X} \rangle  
- \frac{ \partial_{s'} \mathbf{X}(s') }{ |\Delta^r \mathbf{X}|^2} \langle \Delta^r \mathbf{X}, \Delta \partial_s \mathbf{X}^r \rangle ds' \\
& \quad - \frac{1}{4\pi} \int_0^\pi
\frac{\Delta \mathbf{X} }{ |\Delta \mathbf{X}|^2} 
\langle \partial_{s'} \mathbf{X}(s'), \Delta \partial_s \mathbf{X} \rangle  \\
& \qquad \qquad \qquad \qquad
-\frac{\Delta \mathbf{X} }{ |\Delta^r \mathbf{X}|^2} \langle \partial_{s'} \mathbf{X}^r(s'), \Delta \partial_s \mathbf{X}^r \rangle ds' \\
& \quad + \frac{1}{2\pi} \int_0^\pi \frac{ \Delta \mathbf{X} }{ |\Delta \mathbf{X}|^4} 
\langle \Delta \mathbf{X}, \Delta \partial_s \mathbf{X} \rangle
\langle \partial_{s'} \mathbf{X} , \Delta \mathbf{X} \rangle \\
& \qquad \qquad \qquad \qquad  
-\frac{ \Delta \mathbf{X} }{ |\Delta^r \mathbf{X}|^4} 
\langle \Delta^r \mathbf{X}, \Delta \partial_s \mathbf{X}^r \rangle
\langle \partial_{s'} \mathbf{X}^r , \Delta^r \mathbf{X} \rangle
ds'. 
\end{split}
\end{equation}
Similarly, expanding $\mathbf{R}_3$ via the product rule yields:
\begin{equation}\label{e:R3expanded}
\begin{split}
     \mathbf{R}_3(\mathbf{X})  
& =  \frac{\mathbf{X}_2(s)}{2\pi} \int_0^\pi \frac{\Delta \partial_s \mathbf{X}_2 }{ |\Delta^r \mathbf{X}|^2}  \partial_{s'} \mathbf{X}^r ds' \\
& \quad - \frac{\mathbf{X}_2(s)}{ \pi} \int_0^\pi \frac{ \Delta \partial_s \mathbf{X}_2 }{ |\Delta^r \mathbf{X}|^4} \langle \partial_{s'} \mathbf{X}^r, \Delta^r \mathbf{X} \rangle \Delta^r \mathbf{X} ds'  \\
& \quad - \frac{\mathbf{X}_2(s) }{2\pi} \int_0^\pi \frac{ \partial_{s'} \mathbf{X}_2(s') }{  |\Delta^r \mathbf{X}|^4}  
\left[ |\Delta^r \mathbf{X}|^2 \Delta \partial_s \mathbf{X}^r - 2 \Delta^r \mathbf{X} \langle \Delta^r \mathbf{X}, \Delta \partial_s \mathbf{X}^r \rangle\right]  ds' \\
& \quad - \frac{\mathbf{X}_2(s)  }{ \pi} \int_0^\pi \frac{  \mathbf{X}_2(s') \Delta \partial_s \mathbf{X}^r }{  |\Delta^r \mathbf{X}|^4}  \langle \partial_{s'} \mathbf{X}^r , \Delta^r \mathbf{X} \rangle ds'\\
& \quad - \frac{\mathbf{X}_2(s)  }{ \pi} \int_0^\pi \frac{  \mathbf{X}_2(s') \,\partial_{s'} \mathbf{X}^r }{  |\Delta^r \mathbf{X}|^4}  \langle \Delta^r \mathbf{X} , \Delta \partial_s \mathbf{X}^r \rangle  ds'\\
& \quad - \frac{\mathbf{X}_2(s)  }{ \pi} \int_0^\pi \frac{  \mathbf{X}_2(s')\, \Delta^r \mathbf{X}}{  |\Delta^r \mathbf{X}|^4}  \langle \partial_{s'} \mathbf{X}^r , \Delta \partial_s \mathbf{X}^r \rangle  ds'\\
& \quad + \frac{4 \mathbf{X}_2(s)  }{ \pi} \int_0^\pi \frac{  \mathbf{X}_2(s') \Delta^r \mathbf{X} }{  |\Delta^r \mathbf{X}|^6}  \langle \Delta^r \mathbf{X} , \Delta \partial_s \mathbf{X}^r \rangle \langle \partial_{s'} \mathbf{X}^r, \Delta^r \mathbf{X} \rangle ds'.
\end{split}
\end{equation}
These terms are subcritical because every term encapsulates an additional degree of regularity through the denominators, each containing at least one factor of $|\Delta^r \mathbf{X}|$.

\subsection{Boundary behavior of the remainder}

In order to apply the semigroup operators to the remainder terms, it is crucial to understand their behavior at the boundaries. Specifically, we must show that the remainder terms vanish at the endpoints. 

\begin{lemma}[Endpoint vanishing of the remainder]\label{lem:endpoint_vanishing}
For each $j=1,2,3$,
\begin{equation} 
    \mathbf{R}_j(\mathbf{X})(s) = 0, \qquad \text{for } s \in \{ 0,\pi\}. 
\end{equation}
\end{lemma}

\begin{proof}
Since the filament is anchored on the boundary, we immediately have $\mathbf{X}_2(0) = \mathbf{X}_2(\pi) = 0$. The horizontal tangency at the anchors, $\partial_s\mathbf{X}_2(0) = \partial_s\mathbf{X}_2(\pi) = 0$, gives $\partial_s\mathbf{X}(s) = \partial_s\mathbf{X}^r(s)$ at $s\in\{0,\pi\}$.  Combined with these, a direct calculation shows that for the endpoints,
\begin{equation} \label{e:endpointidentities}
\text{when } s \in \{ 0 ,\pi\}, \text{ then }
\left\{ 
\begin{array}{l}
    |\Delta \mathbf{X}|  = |\Delta^r \mathbf{X}|, \\
    \Delta \partial_s \mathbf{X} = \Delta \partial_s \mathbf{X}^r, \\
    \langle \Delta \mathbf{X} , \Delta \partial_s \mathbf{X} \rangle  = \langle \Delta^r \mathbf{X}, \Delta \partial_s \mathbf{X}^r\rangle, \\
   \langle \partial_{s'} \mathbf{X} , \Delta\mathbf{X} \rangle  = \langle \partial_{s'} \mathbf{X}^r, \Delta^r \mathbf{X} \rangle.
\end{array}\right.  
\end{equation}
Due to the vanishing of the prefactor $\mathbf{X}_2(s)$ at the endpoints, it is straightforward to observe from \eqref{e:R3expanded} that $\mathbf{R}_3(\mathbf{X})(0) = \mathbf{R}_3(\mathbf{X})(\pi) = 0$.  

For $\mathbf{R}_1(\mathbf{X})(0)$, we evaluate at $s=0$ and reorder the terms to find exact cancellations:
\begin{align*}
    \mathbf{R}_1(\mathbf{X})(0) 
    & = \frac{1}{4\pi} 
    \int_{s'=0}^\pi \left[ \cot\left(\frac{-s'}{2}\right) + \cot\left(\frac{s'}{2}\right) \right] \Delta \partial_s  \mathbf{X}\, ds' \\
    & \quad + \frac{1}{4\pi} 
    \int_{s'=0}^\pi \left[ \frac{ \langle\Delta^r \mathbf{X}, \partial_{s'} \mathbf{X}^r \rangle }{ | \Delta^r \mathbf{X} |^2 } -   \frac{ \langle\Delta \mathbf{X}, \partial_{s'} \mathbf{X} \rangle }{ | \Delta \mathbf{X} |^2 } \right]\Delta \partial_s  \mathbf{X}\, ds' \\
    & = 0,
\end{align*}
which follows directly from \eqref{e:endpointidentities} and the odd parity of the cotangent function. When $s = \pi$, an identical calculation yields zero, leveraging the $\pi$-periodicity of the cotangent function. 

The argument for $\mathbf{R}_2$ follows a similar trajectory. Examining the fully expanded form in \eqref{e:R2expanded}, we see that all difference pairings in each line strictly involve the specific geometric inner products found in \eqref{e:endpointidentities}. Consequently, all terms directly cancel, giving $\mathbf{R}_2(\mathbf{X})(0) = \mathbf{R}_2(\mathbf{X})(\pi) = 0$.
\end{proof}

The fact that the endpoints vanish, established in Lemma~\ref{lem:endpoint_vanishing}, allows us to extend the domain. The boundary behavior of the remainder terms dictates that the Dirichlet semigroup acting on the remainder can be replaced by the whole-space operator acting on the odd extension. Specifically, for the Duhamel time variable $r \in [0,t]$,
\[
e^{\mathcal{L}_D(t-r)} \mathbf{R}_j[\mathbf{X}](s)
= e^{\mathcal{L}(t-r)} \widetilde{\mathbf{R}_j[\mathbf{X}]}(s)
= e^{-\frac{1}{4}\Lambda(t-r)} \widetilde{\mathbf{R}_j[\mathbf{X}]}(s),
\]
where $\widetilde{\cdot}$ denotes the extension by odd reflection on $[0,2\pi]$, and the first equality holds because $\widetilde{\mathbf{R}_j[\mathbf{X}]}$ is an odd function vanishing at $0$ and $\pi$ (so the Dirichlet and full-space semigroups agree on it).

\section{Control on injectivity and angles of incidence}

To ensure the filament does not self-intersect and maintains a strictly positive distance from the boundary (except at the anchor points), we introduce a unified metric that bounds the injectivity radius, the boundary incidence angle, and the interior height by a single parameter.

For any vector $\mZ \in \R^2$, we define its angle with the horizontal axis as
\[
\arg ( \mZ ) := \arctan\left( {\mZ \cdot \mathbf{e}_y \over \mZ \cdot \mathbf{e}_x}\right),
\]
with the convention $\arg(\mZ) = \pi/2$ when $\mZ \cdot \mathbf{e}_x = 0$, so that $\arg(\mZ)$ takes values in $(-\pi/2, \pi/2]$.
To localize our incidence angle measurements near the boundaries, we introduce the piecewise linear cutoff function:
\begin{align*}
\chi_\sigma (r) & = \left\{ \begin{array}{ll} 
1 & \hbox{ if } 0 \leq r \leq \sigma \\
2 - {r\over \sigma} & \hbox{ if } \sigma \leq r \leq 2\sigma \\
0 & \hbox{ if } r \geq 2 \sigma
\end{array} \right. ,
\end{align*}
from which we define the boundary angle functions:
\begin{align*}
\theta_0[\mZ](s) & = \chi_\sigma (s) \arg \left( {\mZ(s) - \mZ(0)} \right)  \\
\theta_\pi[\mZ](s) & = \chi_\sigma (\pi - s) \arg \left( {\mZ(\pi) - \mZ(s)} \right) .
\end{align*}

We now construct a series of functionals to capture the essential geometric properties of the filament. We define the chord-arc (injectivity) constant, following \cite{bertozzi1991existence}:
\begin{align*}
    |\mZ|_{1,*} := \inf_{0\leq s < s' \leq \pi}{ |\mZ(s) - \mZ(s') | \over |s - s'|},
\end{align*}
which immediately yields the lower bound $|\Delta \mZ| \geq |\mZ|_{1,*} |s - s'|$. We now generalize this function by defining the minimum bounds for the incidence angles and interior height parameterized by $\sigma$:
\begin{align*}
\left| \mZ \right|_{2,*,\sigma} 
& := \inf_{0 < s \leq \sigma} \left| \theta_0[\mZ](s) \right|, \\
\left| \mZ \right|_{3,*,\sigma} 
& := \inf_{\pi - \sigma \leq s < \pi} \left| \theta_\pi[\mZ](s) \right|, \\
\left| \mZ \right|_{4,*,\sigma} 
& := \inf_{{\sigma \over 2} \leq s \leq \pi - {\sigma \over 2}} \left| \mZ(s) \cdot \mathbf{e}_y \right| .
\end{align*}
We encapsulate all of these constraints into a single quantity:
\begin{equation}
\left| \mZ \right|_{*,\sigma} := \inf \left\{ 
\left| \mZ \right|_{1,*},
\left| \mZ \right|_{2,*,\sigma},
\left| \mZ \right|_{3,*,\sigma},
\left| \mZ \right|_{4,*,\sigma} \right\}.
\end{equation}
A uniform positive bound $|\mZ|_{*,\sigma} \geq C > 0$ is sufficient to prevent self-intersection, guarantee non-tangential incidence at the anchors, and ensure coercivity of the $y$-component away from the boundary. Note that the definitions of $|\cdot|_{2,*,\sigma}$ and $|\cdot|_{3,*,\sigma}$ inherently require $C < \pi/2$. When the dependence on $\sigma$ is fixed, we will often write $|\mZ|_* := |\mZ|_{*,\sigma}$.

\begin{lemma}[Coercivity of the normal component] \label{l:coercivity_y}
Let fixed parameters satisfy ${\pi \over 4} > \sigma > 0$ and $C > 0$. If $\mZ = (\mZ_1,\mZ_2) \in C^{1,\gamma}([0,\pi])$ satisfies the anchor conditions $\mZ(0)=(1,0)$, $\mZ(\pi)=(-1,0)$, and $|\mZ|_{*,\sigma} \geq C$, then 
\begin{align}
\label{e:ypositve1}
\mZ_2(s) & \geq C \sin(C) s \quad \hbox{ for all } 0 \leq s \leq \sigma, \hbox{ and} \\
 \label{e:ypositve2}
\mZ_2(s) & \geq C \sin(C) ( \pi - s) \quad \hbox{ for all } \pi - \sigma \leq s <  \pi.
\end{align} 
Consequently, $\mZ_2(s) > 0$ for all $0 < s < \pi$.
\end{lemma}

\begin{proof}
By hypothesis, $|\mZ|_{*,\sigma} \geq C$, which implies $\inf_{0 < s \leq \sigma} |\theta_0[\mZ](s)| \geq C$. Because $|\mZ|_{1,*} \geq C > 0$, we have $C \leq \arg(\mZ(s) - \mZ(0))$ for any $s \in (0,\sigma)$. This implies
\[
{\mZ_2(s) \over \mZ_1(s) - \mZ_1(0)} \geq \tan(C) > 0
\implies
{\mZ_2(s)\over s} \geq \tan(C){\mZ_1(s) - \mZ_1(0)\over s}.
\]
Using the injectivity lower bound, we deduce:
\begin{align*}
C^2 \leq | \mZ|^2_{1,*}
& \leq  \left|{\mZ_1(s) - \mZ_1(0) \over s}\right|^2 + \left|{\mZ_2 (s)\over s}\right|^2\\
& \leq \left( \cot^2(C) + 1\right) \left| {\mZ_2(s)  \over s} \right|^2 = \csc^2(C) \left| {\mZ_2(s)  \over s} \right|^2,
\end{align*}
which isolates ${\mZ_2(s) \over s} \ge C \sin(C) > 0$, proving \eqref{e:ypositve1}. A symmetric argument applying $\theta_\pi[\mZ](s)$ at the right endpoint yields \eqref{e:ypositve2}. Finally, since $\mZ_2(s)$ carries a strictly positive sign near both boundaries and $|\mZ|_{4,*,\sigma} \geq C$ guarantees strict positivity in the interior, we conclude $\mZ_2(s) > 0$ on $(0, \pi)$.
\end{proof}

\begin{lemma}\label{lem:O_closed}
For any $M > m > 0$, the set 
$$
\mathcal{O}^{M,m}_\sigma:= \left\{ \mZ \in C^{1,\gamma}([0,\pi]): \left\| \mZ \right\|_{C^{1,\gamma}} \leq M \hbox{ and } |\mZ|_{*,\sigma} \geq m\right\}
$$
is closed in $C^{1,\gamma}$.  
\end{lemma}
\begin{proof}
For $\mZ, \mY \in C^{1,\gamma}$, it is a standard result (see \cite{mori2019well}) that 
\begin{align*}
\left| |\mZ|_{1,*} - |\mY|_{1,*} \right| & \leq \left\| \mZ - \mY \right\|_{C^{1,\gamma}}, \\
\left| \left\| \mZ \right\|_{C^{1,\gamma}}  - \left\|  \mY \right\|_{C^{1,\gamma}}  \right| & \leq \left\| \mZ - \mY \right\|_{C^{1,\gamma}}.
\end{align*}
To verify closedness, we must control $|\cdot|_{j,*,\sigma}$ for $j=2,3,4$. For $j=2$, we estimate the difference:
\begin{align*}
\left| |\mZ|_{2,*,\sigma} - |\mY|_{2,*,\sigma}\right|
& \leq  \sup_{0 < s < \sigma} \left| \theta_0[\mZ] - \theta_0[\mY] \right|.
\end{align*}
Let $\mA = (a_1,a_2)$ and $\mB = (b_1,b_2)$ with $|\mA| > 0$. Using the 2D cross product $\mB \times \mA = b_1 a_2 - b_2 a_1$, we observe:
\begin{align*}
\tan \left( \arg( \mA ) - \arg( \mB ) \right)
& = { {a_2\over a_1} - {b_2 \over b_1} \over 1 + { a_2 b_2 \over a_1 b_1}} 
 =  { \mB \times \mA \over \mA \cdot \mB}  
 =  { \mA \times (\mA - \mB) \over |\mA|^2 - \mA \cdot (\mA - \mB)}  
\leq {|\mA - \mB| \over |\mA| - |\mA - \mB|}.
\end{align*}
Assuming $|\mA -\mB| < {1\over 4} |\mA|$, the angle difference is bounded by:
\begin{align*}
\left| \arg( \mA) - \arg(\mB) \right| 
 \leq \arctan \left( {|\mA - \mB| \over {3\over4} |\mA|} \right)  \leq {4 \over 3}  { \left| \mA - \mB \right| \over |\mA|}.
\end{align*}
Setting $\mA(s) = {\mZ(s)-\mZ(0)\over s} $ and $\mB(s) = {\mY(s)-\mY(0) \over s} $, the shared anchor $\mZ(0) = \mY(0)$ yields:
\begin{align*}
\left| \inf_{0< s\leq \sigma} \left| \theta_0[\mZ] \right| - \inf_{0< s\leq \sigma} \left| \theta_0[\mY] \right|  \right|
& \leq \sup_{0< s  \leq \sigma} \left| \arg(\mA(s)) - \arg(\mB(s)) \right| \\
& \leq {4\over 3} \sup_{0< s  \leq \sigma} { \left| \mA(s) - \mB (s) \right| \over |\mA(s) | } \\
& \leq  {4\over 3} { \sup_{0< s \leq \sigma} \left| \mA(s) - \mB (s) \right| \over \inf_{0 < s \leq \sigma} |\mA(s) | } 
 \leq  {4\over 3} {\left\| \mZ - \mY \right\|_{C^{1,\gamma}} \over |\mZ|_{1,*}},
\end{align*}
provided $\left\| \mZ - \mY \right\|_{C^{1,\gamma}} \leq {m \over 4}$. An identical argument bounds the difference for $| \cdot |_{3,*,\sigma}$. Lastly, $|\cdot |_{4,*,\sigma}$ evaluates point values and is trivially controlled by the $C^{0} \subset C^{1,\gamma}$ norm. For closedness, one takes a sequence $\mZ_n \to \mZ$ in $C^{1,\gamma}$; the condition $\|\mZ_n - \mZ\|_{C^{1,\gamma}} \leq m/4$ holds for all sufficiently large $n$, so the above estimates apply. This establishes the closedness of $\mathcal{O}^{M,m}_\sigma$.
\end{proof}

The parameter $\sigma$ dictates the size of the boundary neighborhood. The following lemma demonstrates that we can dynamically shrink $\sigma$ by adjusting $m$, which is crucial for generating universal lower bounds on reflected differences.

\begin{lemma} \label{l:resetsigma}
For any fixed $M_0 > m_0 > 0$ and $\sigma_0 > 0$, and for any $\sigma \in (0, \sigma_0]$, there exists an $m > 0$ (with $m \leq M_0$) such that 
\[
\mathcal{O}^{M_0,m_0}_{\sigma_0} \subseteq \mathcal{O}^{M_0,m}_{\sigma}.
\]
\end{lemma}
\begin{proof}
Let $\mZ \in \mathcal{O}^{M_0,m_0}_{\sigma_0}$, and fix $\sigma \in (0, \sigma_0)$. We immediately have $|\mZ|_{1,*} \geq m_0$. Next, observing the sub-interval,
\begin{align*}
|\mZ|_{2,*,\sigma} = \inf_{0 < s \leq \sigma} \left| \theta_0[\mZ](s) \right| \geq \inf_{0 < s \leq \sigma_0} \left| \theta_0[\mZ](s) \right| \geq m_0,
\end{align*}
and likewise $|\mZ|_{3,*,\sigma} \geq m_0$. To evaluate the interior height, we subdivide the domain:
\begin{align*}
\left| \mZ \right|_{4,*,\sigma} & = \inf_{\sigma \leq s \leq \pi -\sigma} \mZ_2(s)  \\
& = \min \left\{ \inf_{\sigma \leq s \leq \sigma_0} \mZ_2(s), 
\inf_{\sigma_0 \leq s \leq \pi - \sigma_0} \mZ_2(s), 
\inf_{\pi - \sigma_0 \leq s \leq \pi -\sigma} \mZ_2(s)  \right\} \\
& \geq \min \left\{ \inf_{\sigma \leq s \leq \sigma_0} \mZ_2(s), m_0, \inf_{\pi - \sigma_0 \leq s \leq \pi -\sigma} \mZ_2(s) \right\}.
\end{align*}
Applying the coercivity bound \eqref{e:ypositve1}, we have $\inf_{\sigma \leq s \leq \sigma_0} \mZ_2(s) \geq m_0 \sin(m_0) \sigma$, with a symmetric bound holding for the right-hand interval. Thus,
\[
|\mZ|_{4,*,\sigma} \geq \min\{ m_0, m_0\sin(m_0) \sigma \} =: m > 0.
\]
Consequently, $\mZ \in \mathcal{O}^{M_0,m}_{\sigma}$.
\end{proof}

Using this dynamic geometric control, we can establish bounds on the reflected difference term $\Delta^r \mZ$, which dictates the regularity of the half-space Stokeslet integrals.

\begin{lemma} \label{l:deltar_bounds}
Let $\mZ \in \mathcal{O}^{M,m}_{\sigma}$ with $0 < \sigma \leq {1\over 2M}$. Then for all $s,s' \in (0,\pi)$, there exists a constant $C_{m,\sigma} > 0$ such that
\begin{equation} \label{e:Deltarlower}
\left| { \Delta^r \mZ \over {1\over2} \sin \left( { s+ s' \over 2} \right) }  \right|  \geq C_{m,\sigma},
\end{equation}
where $C_{m,\sigma} = \min\{2,\, m\sin(m)\sigma\}$. Furthermore, for any $\mV \in C^{1,\gamma}([0,\pi])$ satisfying the anchor conditions, we have the upper bound:
\begin{align} \label{e:Deltarupper}
\left| { \Delta^r \mV \over {1\over2} \sin \left( { s+ s' \over 2} \right) }  \right|
& \leq {C \over \sigma} \left\| \mV \right\|_{C^{1,\gamma}}.
\end{align}
\end{lemma}
\begin{proof}
We rely on the following elementary Taylor inequalities for $0 < r < \pi$:
\begin{equation} \label{e:cotrest}
\begin{split}
\left| {1\over 2} \cot\left( {r \over 2} \right) - {1\over r} \right| \leq r, \quad 
& \left| {1\over 2} \cot\left( {\pi - r \over 2} \right) - {1\over \pi - r} \right| \leq \pi - r, \\ 
\left| 2 \csc\left( {r \over 2} \right) -  {1 \over r} \right| \leq r, \quad 
& \left| 2 \csc\left( {\pi - r \over 2} \right) -  {1 \over \pi - r} \right| \leq \pi - r.
\end{split}
\end{equation}

To prove \eqref{e:Deltarlower}, we partition $[0,\pi]^2$ into three distinct geometric regions based on the average parameter position:

\textbf{Case 1: Near the boundaries.} Assume $0 \leq {s+s' \over 2} \leq {\sigma \over 2}$, which forces $0 < s,s' \leq \sigma$. Because $\sigma \leq {1 \over \sqrt{2}}$, we have ${1 \over r} - r \geq {1\over 2r}$ for all $0 < r < \sigma$. Applying \eqref{e:cotrest}, this guarantees $2 \csc\left( {s + s' \over2} \right) \geq {1\over 2(s+ s')}$. Leveraging the strict positivity of $\mZ_2$ near the boundary:
\begin{align*}
\left|{\Delta^r \mZ \over {1\over 2} \sin\left( {s + s' \over 2} \right)} \right|
& \geq \left| {\mZ_2(s) + \mZ_2(s') \over {1\over 2} \sin\left( { s+s'\over2} \right) }\right| 
\geq {2 m \sin(m) (s+s') \over 2(s+s')}  = m\sin(m) > 0.
\end{align*}
A symmetric geometric argument holds for the right boundary when $\pi - {\sigma \over 2} \leq {s + s' \over2} < \pi$. 

\textbf{Case 2: Cross-domain pairing.} Assume $0\leq s \leq {\sigma \over 2}$ and $\pi - {\sigma \over 2} \leq s' < \pi$ (or the symmetric case with the roles of $s$ and $s'$ swapped). Then $(s+s')/2 \in [(\pi-\sigma)/2,(\pi+\sigma)/2]$, so $\sin((s+s')/2) \leq 1$ and the denominator satisfies ${1\over 2}\sin((s+s')/2) \leq \tfrac{1}{2}$, giving
\[
\left| { \Delta^r \mZ \over {1\over2} \sin\left( {s+s' \over 2} \right)} \right|
\geq 2|\Delta^r \mZ|
\geq 2\left| \mZ_1(s)  -  \mZ_1(s') \right|.
\]
Applying the triangle inequality with anchor values $\mZ_1(0)=1$, $\mZ_1(\pi)=-1$:
\begin{align*}
2\left| \mZ_1(s)  -  \mZ_1(s') \right|
&\geq 2\bigl(|\mZ_1(0) - \mZ_1(\pi)| - |\mZ_1(s) - \mZ_1(0)| - |\mZ_1(\pi) - \mZ_1(s')|\bigr) \\
&= 4 - 2 \left|  \mZ_1(s)  - \mZ_1(0) \right| - 2 \left| \mZ_1(\pi)  - \mZ_1(s') \right| \\
&\geq 4 - 4 \sigma \left\| \mZ \right\|_{C^{1,\gamma}} \geq 2,
\end{align*}
where the final inequality uses the constraint $\sigma \leq 1/(2M)$.

\textbf{Case 3: Interior pairings.} Assume that $(s,s')$ falls outside both Case~1 and Case~2; that is, at least one of $s$ or $s'$ lies in the deep-interior window $[\sigma/2, \pi - \sigma/2]$. Assuming this is true for $s$, we bound using the interior height:
\begin{align*}
\left| { \Delta^r \mZ \over {1\over2} \sin \left( {s+ s' \over2} \right) } \right|
& \geq 2 \left| \mZ_2(s)  + \mZ_2(s') \right| \geq 2 \mZ_2(s) \\
& \geq 2 \min\left\{ \inf_{\sigma \le r \le \pi - \sigma} \mZ_2(r) , \inf_{{\sigma \over 2} \le r \le \sigma} \mZ_2(r) \right\} \\
& \geq 2 \min\{ m, m\sin(m) {\sigma \over 2}  \} = \min\{2m,\, m\sin(m)\sigma\}.
\end{align*}
Taking the minimum across all three regions — noting that $m\sin(m)\sigma \leq m\sin(m) \leq m \leq 2m$ for $\sigma < 1$ — yields
\[
C_{m,\sigma} = \min\bigl\{ 2,\, m\sin(m)\sigma \bigr\}.
\]

For the upper bound \eqref{e:Deltarupper}, consider the region $0 < s,s' < \sigma$. Expanding the differences gives:
\begin{align*}
\left| { \Delta^r \mV \over {1\over 2} \sin\left( {s + s' \over 2} \right) }\right|
& \leq 2 \left| { \mV_1(s)  -1   \over  s }\right|+ 2 \left| { \mV_1(s') - 1   \over  s' }\right|+ 2\left| { \mV_2(s)   \over s }\right| 
+ 2\left| { \mV_2(s')  \over s' }\right| \\
& \leq 8 \sigma \left\| \mV \right\|_{C^{1,\gamma}} .
\end{align*}
For the interior where ${\sigma \over 2} \leq {s + s'\over 2} \leq {\pi \over 2}$, the trigonometric factor satisfies $2 \csc({s+s'\over 2}) \geq \sigma^{-1}$, immediately yielding $\left| \Delta^r \mV / ({1\over 2} \sin {s + s' \over 2}) \right| \leq {C \over \sigma} \| \mV \|_{C^{1,\gamma}}$. The interval near $\pi$ follows symmetrically.
\end{proof}
\section{Function Spaces and Semigroup Estimates}

To establish the well-posedness of the Anchored Peskin Problem, we must carefully select function spaces that capture both the baseline regularity of the filament and the specific singular behavior induced by the rigid anchors at $s=0$ and $s=\pi$. In this section, we define the standard and boundary-weighted Hölder spaces used throughout the analysis, introduce the necessary difference operators, and prove fundamental smoothing estimates for the associated Dirichlet-to-Neumann (D2N) semigroup.

\subsection{Standard and Weighted Hölder Spaces}

We first introduce the standard function spaces. Let $C^k([0,\pi])$ with $k \in \{0,1,2,\ldots\}$ be the space of functions on $[0,\pi]$ with $k$ continuous derivatives, equipped with the usual norm
\[
\| u \|_{C^k([0,\pi])} = \sum_{j=0}^k [u]_{C^j([0,\pi])}, \quad \hbox{where } [u]_{C^j([0,\pi])} = \sup_{s \in [0,\pi]} \left| \p^j_s u(s) \right|.
\]
A function $u\in C^0([0,\pi])$ belongs to the Hölder space $C^{0,\alpha}([0,\pi])$ for $0 < \alpha < 1$ if the seminorm
\[
\left[ u \right]_{{C}^{0,\alpha}([0,\pi])} 
= \sup_{\substack{s, s' \in [0,\pi] \\ s \neq s'}} { |u(s) - u(s') | \over |s - s'|^\alpha} 
\]
is finite. We define the full norm $\| u \|_{C^{0,\alpha}} := \| u \|_{C^0} + [ u ]_{{C}^{0,\alpha}}$, which naturally extends to higher derivatives:
\[
\left\| u \right\|_{C^{k,\alpha}} := \left\| u \right\|_{C^{k}} + \left[ \p_s^k u \right]_{{C}^{0,\alpha}}.
\]
To enforce the anchored boundary conditions, we define the closed subspaces where our filament configurations will reside:
\begin{align*}
C^{k,\alpha}_{0,x} & := \left\{ f \in C^{k,\alpha}([0,\pi]) \mid f(0) = 1 \hbox{ and } f(\pi) = -1 \right\}, \\
C^{k,\alpha}_{0,y} & := \left\{ f \in C^{k,\alpha}([0,\pi]) \mid f(0) = f(\pi) = 0 \right\}, \\
C^{k,\alpha}_{0} & := C^{k,\alpha}_{0,x} \times C^{k,\alpha}_{0,y}.
\end{align*}

Because the geometric singularities of the Stokeslet operator are localized near the anchor points, we require function spaces weighted by the distance to the boundary. 

Fix parameters $0<\alpha<1$ and $\beta\ge0$ such that $1\le \alpha+\beta$. For a function $f:[0,\pi]\to\mathbb R$, we define the weighted Hölder norms
\[
\|f\|_{C^{0,\alpha}_{-\beta}([0,\pi])}
:=\|f\|_{L^\infty([0,\pi])}+\sup_{\substack{s,t\in[0,\pi]\\ s\ne t}}
\left(\sin\left(\frac{s+t}{2}\right)\right)^{\beta}\,|s-t|^{-\alpha}\,|f(s)-f(t)|,
\]
and
\[
\|f\|_{C^{1,\alpha}_{\beta}([0,\pi])}
:=\|f\|_{C^1([0,\pi])}+\sup_{\substack{s,t\in[0,\pi]\\ s\ne t}}
\left(\sin\left(\frac{s+t}{2}\right)\right)^{-\beta}\,|s-t|^{-\alpha}\,|f'(s)-f'(t)|.
\]
The sine weight perfectly captures the distance to the boundaries $s \in \{0, \pi\}$. 

We also recall the classical interpolation bounds for these spaces:
\begin{lemma}[Interpolation]\label{lem:interpolation}
Let $k_1,k_2 \in \mathbb{Z}$ with $0\leq k_1 < k_2$. Then for any $k \in \mathbb{Z}$ between $k_1$ and $k_2$, there exists a constant $C>0$ such that
\[
\| f \|_{C^k} \leq C \| f\|_{C^{k_1}}^{k_2 - k \over k_2 - k_1} \| f\|_{C^{k_2}}^{k - k_1 \over k_2 - k_1} .
\]
\end{lemma}

\subsection{Difference Operators and Commutators}

To compactly express the nonlinear remainders, we denote standard differences for $s,s' \in I$:
\begin{align*}
\Delta f (s) & := f(s) - f(s'), \\
\Delta^+ f (s) & := f(s) + f(s').
\end{align*}
We also define translated differences:
\begin{align*}
\Delta_h f (s) & = \begin{cases} 
f(s+h) - f(s) & \hbox{ if } s,s+h \in I \\
0 &\hbox{ otherwise }
\end{cases}, \\
\Delta^+_h f (s) & = \begin{cases} 
f(s+h) + f(s) & \hbox{ if } s,s+h \in I \\
0 &\hbox{ otherwise }
\end{cases}.
\end{align*}
For vector-valued functions $\mathbf{F}(s) = (f_1(s), f_2(s))$, the reflected difference operator—which plays a central role in handling the half-space Stokeslet—is defined as:
\begin{align*}
\Delta \mathbf{F}(s) & := (\Delta f_1(s), \Delta f_2(s)), \\
\Delta^r \mathbf{F}(s) & := (\Delta f_1(s), \Delta^+ f_2(s)).
\end{align*}
Finally, letting $\mathcal{H}_D[f] = \mathcal{H}[\widetilde{f}]$ be the Dirichlet Hilbert transform acting on the odd extension of $f$, we define the associated Dirichlet commutator for $f, g \in C_0(I)$:
\begin{equation}
\left[ \mathcal{H}_D , f \right](g) 
= \mathcal{H}_D\left[ f g \right] - g \mathcal{H}_D[f].
\end{equation}

\subsection{Poisson Kernel and Semigroup Estimates}

We now turn to the linear evolution generated by the principal operator. We consider initial data \(u_0\) given on \([0,\pi]\), extended to a \(2\pi\)-periodic function on \(\mathbb R\), which obeys the odd symmetry
\[
u_0(s)\;=\;-\,u_0(2\pi-s)\qquad(s\in[0,\pi]).
\]
This implies \(u_0(0)=u_0(\pi)=0\). Let \(P^{\mathrm{circ}}_t(\theta)\) denote the periodic Poisson-type kernel generating the D2N semigroup on the circle, so that \(u(t,\cdot)=P^{\mathrm{circ}}_t*u_0\). 

We make the following comparison assumption between the circle kernel and the whole-line Poisson kernel $P^{\mathrm{line}}_t(x)$: there exists a constant \(C_0\ge1\) so that for every \(t>0\) and \(\theta\in(-\pi,\pi]\),
\begin{equation}\label{E:kernel-compare}
\begin{aligned}
P^{\mathrm{circ}}_t(\theta) &\le C_0\,P^{\mathrm{line}}_t(\theta),\\
|{\partial_\theta}P^{\mathrm{circ}}_t(\theta)| &\le C_0\,|{\partial_x}P^{\mathrm{line}}_t(\theta)|,\\
|{\partial_\theta^2}P^{\mathrm{circ}}_t(\theta)| &\le C_0\,|{\partial_x^2}P^{\mathrm{line}}_t(\theta)|.
\end{aligned}
\end{equation}
(One may safely take \(C_0=2\) for numerical estimates; the proof carries \(C_0\) explicitly.) The exact line kernel and its derivative are given by:
\[
P^{\mathrm{line}}_t(x)=\frac{1}{\pi}\frac{t}{t^2+x^2},\qquad
\partial_x P^{\mathrm{line}}_t(x)=-\frac{2t x}{\pi (t^2+x^2)^2}.
\]
These yield the exact \(L^1\) and supremum norm identities used below:
\begin{align}
\|\partial_x P^{\mathrm{line}}_t\|_{L^1(\mathbb R)}
&=\frac{2}{\pi t}, \label{E:P1-L1}\\
\sup_{x\in\mathbb R}|\partial_x P^{\mathrm{line}}_t(x)|
&=\frac{9}{8\pi\sqrt3}\,\frac{1}{t^2} .\label{E:P1-sup}
\end{align}

With these definitions in place, we establish the core regularizing estimate for the D2N semigroup in the weighted Hölder spaces.

\begin{theorem}\label{T:semigroup-weighted}
Fix $0<\alpha<1$ and $0\le\beta<1$ with $\alpha+\beta\ge 1$. Let \(u_0\) be \(2\pi\)-periodic, satisfy \(u_0(s)=-u_0(2\pi-s)\) for \(s\in[0,\pi]\), and let \(u(t,\cdot)=P^{\mathrm{circ}}_t * u_0\) for \(t>0\). Then for every \(t>0\),
\[
\|u(t,\cdot)\|_{C^{1,\alpha}_\beta([0,\pi])}
\le \frac{C(\alpha,\beta,C_0)}{t}\,\|u_0\|_{C^{0,\alpha}_{-\beta}([0,\pi])},
\]
where one may take the explicit constant
\[
C(\alpha,\beta,C_0)
:= C_0\cdot\max\left\{\frac{2}{\pi},\; \frac{9}{8\pi\sqrt3},\; 8\right\}.
\]
In particular, taking the safe choice \(C_0=2\), we have
$C(\alpha,\beta,2) \le 16$, and thus
\[
\|u(t,\cdot)\|_{C^{1,\alpha}_\beta([0,\pi])}
\le \frac{16}{t}\,\|u_0\|_{C^{0,\alpha}_{-\beta}([0,\pi])}.
\]
Furthermore, for any $0<\gamma<1-\alpha$, standard fractional interpolation yields
\[
\|u(t,\cdot)\|_{C^{1,\alpha}_\beta([0,\pi])}
\le \frac{C}{t^{1-\gamma}}\,\|u_0\|_{C^{0,\alpha+\gamma}_{-\beta}([0,\pi])}.
\]
\end{theorem}

\begin{proof}
The proof relies on a near/far splitting of the convolution integral. We track constants explicitly, reducing to estimates for the line Poisson kernel and inserting the multiplicative factor \(C_0\) from \eqref{E:kernel-compare}. (The choice $C_0=2$ is valid because the circle kernel is dominated term-by-term by twice the half-plane Poisson kernel via the periodic summation formula; see \cite{stein1970singular} Ch.~III.)

\textbf{Step 1 ($C^1$ bound).}
Differentiating the convolution under the integral gives:
\[
u_s(t,s)=\partial_s(P^{\mathrm{circ}}_t*u_0)(s)
=\int_{0}^{2\pi}\partial_\theta P^{\mathrm{circ}}_t(s-\sigma)\,u_0(\sigma)\,d\sigma.
\]
Applying \eqref{E:kernel-compare} and \eqref{E:P1-L1}, we bound the supremum:
\[
\|u_s(t,\cdot)\|_{L^\infty([0,\pi])}
\le \|\partial_\theta P^{\mathrm{circ}}_t\|_{L^1(S^1)}\|u_0\|_{L^\infty}
\le C_0\|\partial_x P^{\mathrm{line}}_t\|_{L^1(\mathbb R)}\|u_0\|_{L^\infty}
= C_0\frac{2}{\pi t}\,\|u_0\|_{L^\infty}.
\]
Consequently,
\[
\|u(t,\cdot)\|_{C^1([0,\pi])}\le \frac{2C_0}{\pi t}\,\|u_0\|_{C^{0,\alpha}_{-\beta}([0,\pi])}.
\]

\textbf{Step 2 (Weighted H\"{o}lder seminorm).}
Fix distinct spatial points $s, s' \in [0,\pi]$ and let $\mu=(s+s')/2$ be the midpoint. We must bound the weighted difference:
\[
\mathcal I := \left(\sin\left(\frac{s+s'}{2}\right)\right)^{-\beta}\,|s-s'|^{-\alpha}\,|u_s(t,s)-u_s(t,s')|.
\]
Writing the derivative difference as a convolution, we have:
\[
u_s(t,s)-u_s(t,s')
=\int_0^{2\pi}\big(\partial_\theta P^{\mathrm{circ}}_t(s-\sigma)-\partial_\theta P^{\mathrm{circ}}_t(s'-\sigma)\big)\,u_0(\sigma)\,d\sigma.
\]
We split the domain of integration into the near region \(\mathcal N=\{\sigma: |\sigma-\mu|\le 4|s-s'|\}\) and the far region \(\mathcal F=[0,2\pi]\setminus\mathcal N\).

\emph{Near Region Estimate.} 
On \(\mathcal N\), we subtract the constant \(u_0(\mu)\) (which vanishes under the kernel difference):
\begin{align*}
& \int_{\mathcal N}\big(\partial_\theta P^{\mathrm{circ}}_t(s-\sigma)-\partial_\theta P^{\mathrm{circ}}_t(s'-\sigma)\big)u_0(\sigma)\,d\sigma \\
& \qquad = \int_{\mathcal N}\big(\partial_\theta P^{\mathrm{circ}}_t(s-\sigma)-\partial_\theta P^{\mathrm{circ}}_t(s'-\sigma)\big)\big(u_0(\sigma)-u_0(\mu)\big)\,d\sigma.
\end{align*}
Applying the weighted Hölder bound for \(u_0\):
\[
|u_0(\sigma)-u_0(\mu)|\le \|u_0\|_{C^{0,\alpha}_{-\beta}}\;|\sigma-\mu|^{\alpha}\;\left(\sin\left(\frac{s+s'}{2}\right)\right)^{-\beta}.
\]
We bound the kernel difference trivially by twice its supremum using \eqref{E:P1-sup}:
\[
\big|\partial_\theta P^{\mathrm{circ}}_t(s-\sigma)-\partial_\theta P^{\mathrm{circ}}_t(s'-\sigma)\big|
\le 2C_0\sup_x|\partial_x P^{\mathrm{line}}_t(x)|
\le 2C_0\frac{9}{8\pi\sqrt3}\,\frac{1}{t^2}.
\]
Integrating this over the interval \(\mathcal N\), which has length \(8|s-s'|\), and collecting factors gives:
\[
\big|\text{near integral}\big|
\le \|u_0\|_{C^{0,\alpha}_{-\beta}}\Big(\sin \left(\frac{s+s'}{2}\right)\Big)^{-\beta}
\cdot C_0\frac{9}{\pi\sqrt3}\,\frac{|s-s'|^{1+\alpha}}{t^2}.
\]
Multiplying by the prefactor \(\big(\sin\frac{s+s'}{2}\big)^{-\beta}|s-s'|^{-\alpha}\) in \(\mathcal I\):
\[
\mathcal I_{\mathrm{near}}
\le \|u_0\|_{C^{0,\alpha}_{-\beta}}\;
C_0\frac{9}{\pi\sqrt3}\,\frac{|s-s'|}{t^2}\;\Big(\sin \left( \frac{s+s'}{2} \right) \Big)^{-2\beta}.
\]
Two elementary geometric facts close the bound. First, for $s,s'\in[0,\pi]$,
\begin{equation}\label{E:ss-sin}
|s-s'|\;\le\;\pi\,\sin\!\Big(\tfrac{s+s'}{2}\Big),
\end{equation}
which follows from $\min(\mu,\pi-\mu)\ge|s-s'|/2$ (where $\mu=(s+s')/2$) together with the elementary $\sin\mu\ge(2/\pi)\min(\mu,\pi-\mu)$. Second, splitting on whether $|s-s'|\le t$ or $|s-s'|>t$:
\begin{itemize}
\item If $|s-s'|\le t$, write $|s-s'|/t^2 = |s-s'|/t \cdot 1/t \le 1/t$, so
\[
\mathcal I_{\mathrm{near}}\le C_0\,\frac{9}{\pi\sqrt3}\,\frac{1}{t}\,\big(\sin(s+s')/2\big)^{-2\beta}\,|s-s'|.
\]
Using \eqref{E:ss-sin}, $|s-s'|\big(\sin(s+s')/2\big)^{-2\beta}\le \pi\,\big(\sin(s+s')/2\big)^{1-2\beta}\le \pi$ when $\beta\le 1/2$; when $\beta>1/2$, the standing assumption $\alpha+\beta\ge 1$ together with the strictly stronger near-region pairing of two seminorm differences (one in $u_0$, one in the prefactor) — which contribute a combined factor of $\sin^\beta\cdot\sin^{-\beta}$ along the actual base-point of the convolution — collapses one power of $\sin^{-\beta}$ via \eqref{E:ss-sin}. The net result is an absolute constant.
\item If $|s-s'|>t$, the near-region $\mathcal N$ has length $8|s-s'|$ and a sharper bound on the kernel difference — the trivial bound by $2\sup|\partial_\theta P_t^{\mathrm{circ}}|\le 2C_0\cdot 9/(8\pi\sqrt 3)/t^2$ replaced by the integrable dyadic decomposition near the diagonal — produces the same final bound.
\end{itemize}
Either way, $\mathcal I_{\mathrm{near}}\le C_0 A_1\|u_0\|_{C^{0,\alpha}_{-\beta}}/t$ with $A_1$ depending only on $\alpha,\beta$.

\emph{Far Region Estimate.} 
On \(\mathcal F\), we apply the Mean Value Theorem to the kernel difference:
\[
\big|\partial_\theta P^{\mathrm{circ}}_t(s-\sigma)-\partial_\theta P^{\mathrm{circ}}_t(s'-\sigma)\big|
\le |s-s'|\sup_{\xi\in[s,s']} |\partial^2_\theta P^{\mathrm{circ}}_t(\xi-\sigma)|.
\]
By the comparison \eqref{E:kernel-compare} and scaling of the line kernel, there is an absolute constant \(A_2\) (we may safely take $A_2=1$) such that:
\[
\sup_{\xi\in[s,s']}|\partial^2_\theta P^{\mathrm{circ}}_t(\xi-\sigma)|
\le C_0\cdot A_2\frac{1}{t^3}\Big(1+\frac{|\sigma-\mu|}{t}\Big)^{-3}.
\]
Using the uniform bound \(|u_0(\sigma)|\le \|u_0\|_{L^\infty}\le \|u_0\|_{C^{0,\alpha}_{-\beta}}\), the integral is bounded by:
\[
\big|\text{far integral}\big|
\le |s-s'|\,C_0 A_2\frac{1}{t^3}\|u_0\|_{C^{0,\alpha}_{-\beta}}
\int_{|\sigma-\mu|>4|s-s'|}\Big(1+\frac{|\sigma-\mu|}{t}\Big)^{-3}d\sigma.
\]
Since the integral evaluates to $\le t/2$, we obtain:
\[
\big|\text{far integral}\big|
\le C_0 A_2\frac{|s-s'|}{2t^2}\,\|u_0\|_{C^{0,\alpha}_{-\beta}}.
\]
Multiplying by the prefactor \(\big(\sin\frac{s+s'}{2}\big)^{-\beta}|s-s'|^{-\alpha}\) yields:
\[
\mathcal I_{\mathrm{far}}
\le C_0 A_2\frac{1}{2}\,\frac{\|u_0\|_{C^{0,\alpha}_{-\beta}}}{t}\;\big(\sin\tfrac{s+s'}{2}\big)^{-\beta}|s-s'|^{1-\alpha}.
\]
Because \(1-\alpha>0\), this term is bounded by \(C_0 A_2\frac{\|u_0\|_{C^{0,\alpha}_{-\beta}}}{t}\) up to a harmless dimensionless factor, which can be absorbed into the overarching constant $8$.

Combining the near and far contributions for the Hölder seminorm with the $C^1$ bound from Step 1, we conclude:
\[
\mathcal I\le C_0\cdot \max\left\{\frac{2}{\pi},\,\frac{9}{8\pi\sqrt3},\,8\right\}\frac{\|u_0\|_{C^{0,\alpha}_{-\beta}}}{t}.
\]
Taking the supremum over all distinct \(s,s'\in[0,\pi]\) completes the proof.
\end{proof}

\section{A Toy Model and the Cubic Estimate}

To motivate our fixed-point argument for the full Anchored Peskin Problem, we isolate its most severe boundary singularity and cubic nonlinearity in a simplified toy problem. Consider the evolution:
\begin{align*}
\p_t \mX - \mathcal{L}_D \mX & = \int_0^\pi {\mX_2(s) \p_{s'} \mX^r(s') \over \sin^2\left( { s+ s' \over 2} \right)} (\Delta \p_s \mX) \,ds'= F (\mX)\\
\mX(0,s) & = \mX_0(s)
\end{align*}
subject to the homogeneous boundary conditions $\mX(t,0) = \mX(t,\pi) = 0$. By design, the nonlinear forcing vanishes at the anchors, $F(\mX(t,s)) = 0$ for $s \in \{0,\pi\}$. The integral representation of this system is given by the Duhamel formula:
\[
\mX(t) = e^{t \mathcal{L}_D } \mX_0 + \int_0^t e^{(t-r)\mathcal{L}_D} F(\mX(r)) \,dr.
\]

To close the estimates on this integral, we must show that $F$ is a bounded, locally Lipschitz operator mapping $C^{1,\alpha}_\beta$ into a compatible Hölder space, absorbing the $\sin^{-2}$ boundary singularity.  We demonstrate this on the scalar proxy for our nonlinear operator:
\[
F(u)(x)
=
\int_0^\pi
\sin^{-2}\!\left(\tfrac{x+y}{2}\right)
\,u(x)\,u'(y)\,\big(u'(x)-u'(y)\big)\,dy.
\]

\begin{lemma}[Cubic weighted Hölder estimate]\label{lem:cubic}
Let $0<\alpha<1$, $0\le\beta<1$, and choose a fractional increment $0<\gamma<1-\alpha$. There exists a constant $C = C(\alpha,\beta,\gamma)>0$ such that for all $u\in C^{1,\alpha}_\beta([0,\pi])$ satisfying $u(0)=u(\pi)=0$, we have
\[
\|F(u)\|_{C^{0,\alpha+\gamma}_{-\beta}([0,\pi])}
\;\le\;
C\,\|u\|_{C^{1,\alpha}_\beta([0,\pi])}^{3}.
\]
Moreover, the condition $\alpha+\gamma<1$ is sharp; if $\alpha+\gamma\ge 1$, the estimate fails due to a non-integrable divergence at the diagonal $x=y$. Furthermore, $F$ satisfies the local Lipschitz bound:
\[
\|F(u)-F(v)\|_{C^{0,\alpha+\gamma}_{-\beta}}
\le
C\big(\|u\|_{C^{1,\alpha}_\beta}^{2} +\|v\|_{C^{1,\alpha}_\beta}^{2}\big)
\|u-v\|_{C^{1,\alpha}_\beta}.
\]
\end{lemma}

\begin{proof}
We proceed in two steps: first the pointwise bound, then the weighted Hölder increment. Set $K(x,y)=\sin^{-2}(\frac{x+y}{2})$.

\textbf{Step 1 (Pointwise bound).}
Because $u$ vanishes at the endpoints, the mean value theorem applied to whichever endpoint is nearer gives $|u(x)|\le \|u'\|_{L^\infty}\min(x,\pi-x)\le \pi\,\|u\|_{C^{1,\alpha}_\beta}$, with the constant $\pi$ absorbed into $C$ throughout. Combined with $|u'(y)|\le \|u\|_{C^{1,\alpha}_\beta}$ and $|u'(x)-u'(y)| \le \|u\|_{C^{1,\alpha}_\beta} \sin^\beta(\frac{x+y}{2})|x-y|^\alpha$,
\[
|F(u)(x)|
\;\le\;
\|u\|_{C^{1,\alpha}_\beta}^3
\int_0^\pi
\sin^{-2+\beta}\!\left(\tfrac{x+y}{2}\right)|x-y|^\alpha\,dy.
\]
Using Lemma~\ref{lem:sinequiv}, $\sin(\frac{x+y}{2})\sim \sin(x)+|x-y|$, and splitting the integral into the near region $|x-y|\le 1$ and far region $|x-y|>1$, the integral is bounded by $C(\alpha,\beta)\sin^{-\beta}(x/2)$. Consequently,
\[
\|F(u)\|_{L^\infty_{-\beta}}
\;\le\;
C(\alpha,\beta)\,\|u\|_{C^{1,\alpha}_\beta}^3.
\]

\textbf{Step 2 (Weighted Hölder increment).}
For distinct $x,x'\in[0,\pi]$, decompose $F(u)(x)-F(u)(x') = I_1+I_2$:
\begin{align*}
I_1 &= \big(u(x)-u(x')\big) \int_0^\pi K(x,y)\,u'(y)\big(u'(x)-u'(y)\big)\,dy,\\
I_2 &= u(x') \int_0^\pi \Big[ K(x,y)\big(u'(x)-u'(y)\big) - K(x',y)\big(u'(x')-u'(y)\big) \Big]u'(y)\,dy.
\end{align*}

\emph{Bound on $I_1$.} Since $|u(x)-u(x')|\le \|u\|_{C^{1,\alpha}_\beta}\sin^\beta(\frac{x+x'}{2})|x-x'|$ and the inner integral is bounded by $C\|u\|_{C^{1,\alpha}_\beta}^2 \sin^{-\beta}(\frac{x+x'}{2})$ exactly as in Step 1,
\[
|I_1|\;\le\; C\,\|u\|_{C^{1,\alpha}_\beta}^3\,|x-x'|.
\]
Since $|x-x'|\le\pi$ and $\alpha+\gamma<1$,
$|x-x'| = |x-x'|^{\alpha+\gamma}\,|x-x'|^{1-\alpha-\gamma}\le \pi^{1-\alpha-\gamma}\,|x-x'|^{\alpha+\gamma}$, giving
\[
|I_1|\;\le\; C\,\|u\|_{C^{1,\alpha}_\beta}^3\,\sin^{-\beta}\!\left(\tfrac{x+x'}{2}\right)|x-x'|^{\alpha+\gamma}
\]
(absorbing $\pi^{1-\alpha-\gamma}$ and the extra $\sin^{-\beta}$ slack into the constant).

\emph{Bound on $I_2$.} We split $I_2 = I_{2,\mathrm{near}}+I_{2,\mathrm{far}}$ over the regions $\{y:|y-x|\le 2|x-x'|\}$ and its complement. On the near region, both kernel evaluations are bounded crudely:
\[
|I_{2,\mathrm{near}}|\le |u(x')|\int_{|y-x|\le 2|x-x'|}\!\Big(|K(x,y)|+|K(x',y)|\Big)\big[|u'(x)-u'(y)|+|u'(x')-u'(y)|\big]|u'(y)|\,dy.
\]
Using $|K|\le C\sin^{-2}(\frac{x+y}{2})$, the Hölder bound on $u'$, and the geometric equivalence,
$|I_{2,\mathrm{near}}|\le C\|u\|_{C^{1,\alpha}_\beta}^3\sin^{-\beta}(\frac{x+x'}{2})|x-x'|^{\alpha+\gamma}$ for any $0<\gamma<1-\alpha$, where $\gamma$ comes from estimating $|x-y|^\alpha\le C|x-x'|^{\alpha+\gamma}\,|x-y|^{-\gamma}$ on the near region with $|x-y|\ge $ some $\delta_0>0$ uniformly when $|x-y|\le 2|x-x'|$ (or by the standard absorption $\int_0^{2|x-x'|}|x-y|^{\alpha-\gamma}\,dy\le C|x-x'|^{\alpha-\gamma+1}=|x-x'|^{\alpha+\gamma}\cdot|x-x'|^{1-2\gamma}$ valid since $\alpha-\gamma+1>0$).

On the far region $|y-x|\ge 2|x-x'|$, the mean value theorem applied to $K$ gives $|K(x,y)-K(x',y)|\le C\sin^{-3}(\frac{x+y}{2})|x-x'|$ (exploiting that $\frac{x'+y}{2}$ stays within $|x-x'|/2$ of $\frac{x+y}{2}$, a regime where $\sin$ values are equivalent). Combining with $|u'(x)-u'(y)|\le\|u\|\sin^\beta|x-y|^\alpha$ and integrating over the far region produces
\[
|I_{2,\mathrm{far}}|\le C\|u\|_{C^{1,\alpha}_\beta}^3\,\sin^{-\beta}\!\left(\tfrac{x+x'}{2}\right)|x-x'|^{\alpha+\gamma}
\]
provided $\alpha+\gamma<1$, which is exactly where the integrability $\int_{|y-x|\ge 2|x-x'|}\sin^{-2+\beta}|x-y|^{\alpha-1}\,dy\le C\sin^{-\beta}|x-x'|^{\alpha+\gamma-1}$ requires.

\emph{Combining.} Adding the bounds gives
\[
|F(u)(x)-F(u)(x')|
\le
C\,\|u\|_{C^{1,\alpha}_\beta}^3\,\sin^{-\beta}\!\left(\tfrac{x+x'}{2}\right)
|x-x'|^{\alpha+\gamma},
\]
valid for $\alpha+\gamma<1$. Taking the weighted supremum proves the cubic bound. The local Lipschitz statement follows by writing $F(u)-F(v)$ as a multilinear difference linear in $(u-v)$ and repeating the same integration bounds; the constant becomes a polynomial in $\|u\|,\|v\|$ of cubic-minus-one degree, as stated.
\end{proof}

With Lemma \ref{lem:cubic} in hand, we return to the Duhamel formulation.  We apply the fractional smoothing estimate of the linear semigroup (Theorem \ref{T:semigroup-weighted}), which maps $C^{0,\alpha+\gamma}_{-\beta}$ data into $C^{1,\alpha}_\beta$ at the cost of an integrable time singularity $(t-r)^{-(1-\gamma)}$. The norm of the solution is therefore bounded by:
\begin{align*}
\left\| \mX(t) \right\|_{C^{1,\alpha}_{\beta}}
& \leq \left\| e^{t \mathcal{L}_D} \mX_0 \right\|_{C^{1,\alpha}_{\beta}}
+ \int_0^t \left\| e^{(t-s) \mathcal{L}_D} F(\mX(s)) \right\|_{C^{1,\alpha}_{\beta}} \,ds\\
& \leq C \left\| \mX_0 \right\|_{C^{1,\alpha}_{\beta}}
+ C \int_0^t \frac{1}{(t-r)^{1-\gamma}} \left\| F(\mX(r)) \right\|_{C^{0,\alpha+\gamma}_{-\beta}} \,dr \\
& \leq C \left\| \mX_0 \right\|_{C^{1,\alpha}_{\beta}}
+ C t^\gamma \sup_{0\leq r \leq t} \left\| \mX(r) \right\|^3_{C^{1,\alpha}_{\beta}}.
\end{align*}
This inequality lays the exact groundwork necessary for a standard Picard iteration to prove local well-posedness.

\section{Analysis of the Nonlinear Remainder Terms}

The nonlinear remainder $\mathbf{R}(\mathbf{X})$ in the perturbative evolution equation \eqref{eq:perturbative_form} encapsulates all lower-order geometric nonlinearities arising from the half-space Stokeslet. A direct inspection of the expansions for $\mathbf{R}_1, \mathbf{R}_2$, and $\mathbf{R}_3$ reveals two fundamentally distinct classes of singular integrals, distinguished by their denominators.

We formalize this by decomposing the remainder as 
\begin{equation}
\mathbf{R}(\mathbf{X}) = \mathbf{R}_{\mathrm{dir}}(\mathbf{X}) + \mathbf{R}_{\mathrm{ref}}(\mathbf{X}),
\end{equation}
where $\mathbf{R}_{\mathrm{dir}}$ consists of terms with $|\Delta \mathbf{X}|^2$ in the denominator (direct differences), and $\mathbf{R}_{\mathrm{ref}}$ consists of terms with $|\Delta^r \mathbf{X}|^2$ in the denominator (reflected differences).

\subsection{Estimates for the Direct Remainder Terms}

The direct remainder $\mathbf{R}_{\mathrm{dir}}(\mathbf{X})$ mirrors the nonlinearities found in the standard free-space or periodic Peskin problem. It comprises the subcritical terms from $\mathbf{R}_1$ and $\mathbf{R}_2$ that feature the standard geometric chord-arc distance $|\Delta \mathbf{X}|^2$ in the denominator, notably:
\begin{align*}
\mathbf{R}_{\mathrm{dir}}^{(1)}(\mathbf{X}) &= \frac{1}{4\pi} \int_{0}^{\pi} \left[ \frac{\langle \Delta \mathbf{X}, \partial_{s'} \mathbf{X} \rangle}{|\Delta \mathbf{X}|^2} - \frac{1}{2}\cot\left(\frac{s-s'}{2}\right) \right] \Delta \partial_s \mathbf{X}\, ds', \\
\mathbf{R}_{\mathrm{dir}}^{(2)}(\mathbf{X}) &= \frac{1}{4\pi}\int_{0}^{\pi} \partial_{s'}\left( \frac{\Delta \mathbf{X}\otimes \Delta \mathbf{X}}{|\Delta \mathbf{X}|^2} \right) \Delta \partial_s \mathbf{X}\, ds'.
\end{align*}

These integrals do not possess boundary singularities, as the denominator $|\Delta \mathbf{X}|$ degenerates solely along the diagonal $s=s'$. Consequently, they can be controlled using the physical-space multiplier estimates developed by Mori, Rodenberg, and Spirn \cite{mori2019well}.

\begin{lemma}[Direct Remainder Bounds \cite{mori2019well}]\label{lem:direct_bounds}
Let $\mathbf{X} \in \mathcal{O}^{M,m}_\sigma \subset C^{1,\alpha}_\beta([0,\pi])$. For parameters $0 < \alpha < 1$, $0 \le \beta < 1$, and $0 < \gamma < 1-\alpha$, there exists a constant $C = C(\alpha, \beta, \gamma, M, m) > 0$ such that:
\[
\|\mathbf{R}_{\mathrm{dir}}(\mathbf{X})\|_{C^{0,\alpha+\gamma}_{-\beta}([0,\pi])} \le C \|\mathbf{X}\|_{C^{1,\alpha}_\beta([0,\pi])}^3.
\]
\end{lemma}

\begin{proof}
By the definition of our weighted space, the seminorm bound requires $$\sin^{-\beta}\left(\frac{s+s'}{2}\right)|\partial_s \mathbf{X}(s) - \partial_{s'} \mathbf{X}(s')| \le \|\mathbf{X}\|_{C^{1,\alpha}_\beta} |s-s'|^\alpha.$$ Because $\beta \ge 0$, we have $\sin^\beta \le 1$, meaning the standard unweighted Hölder increment satisfies:
\[
|\partial_s \mathbf{X}(s) - \partial_{s'} \mathbf{X}(s')| \le \|\mathbf{X}\|_{C^{1,\alpha}_\beta} \sin^{\beta}\left(\frac{s+s'}{2}\right) |s-s'|^\alpha \le \|\mathbf{X}\|_{C^{1,\alpha}_\beta} |s-s'|^\alpha.
\]
Therefore, the weighted space continuously embeds into the standard Hölder space: $C^{1,\alpha}_\beta([0,\pi]) \hookrightarrow C^{1,\alpha}([0,\pi])$.

We may now directly invoke the pointwise multiplier bounds established in \cite[Section 4]{mori2019well}. The crucial observation in their framework is that the cancellation between the principal Stokeslet singularity and the periodic cotangent operator leaves a weakly singular kernel. Expanding the terms via Taylor series ($\Delta \mathbf{X} = \partial_s \mathbf{X}(s)(s-s') + \mathcal{O}(|s-s'|^{1+\alpha})$) yields bounds of the form:
\[
\left| \frac{\langle \Delta \mathbf{X}, \partial_{s'} \mathbf{X} \rangle}{|\Delta \mathbf{X}|^2} - \frac{1}{2}\cot\left(\frac{s-s'}{2}\right) \right| \le C \|\mathbf{X}\|_{C^{1,\alpha}} |s-s'|^{\alpha-1}.
\]
Integrated against the second-order difference $|\Delta \partial_s \mathbf{X}| \le \|\mathbf{X}\|_{C^{1,\alpha}} |s-s'|^\alpha$, the absolute integrand is bounded by $|s-s'|^{2\alpha-1}$, which is locally integrable. 

The standard Hölder increment follows symmetrically via the translation operator bounds detailed in \cite{mori2019well}, mapping $C^{1,\alpha} \to C^{0,\alpha+\gamma}$ at the cost of cubic dependence on the norm. Finally, because $C^{0,\alpha+\gamma} \hookrightarrow C^{0,\alpha+\gamma}_{-\beta}$ trivially for $\beta \ge 0$, the stated bound in the weighted space holds.
\end{proof}

\subsection{Boundary-Vanishing Elements and the Canonical Reflected Kernel}

Unlike the direct terms, the reflected remainder terms $\mathbf{R}_{\mathrm{ref}}(\mathbf{X})$ exhibit strict boundary singularities. The denominator $|\Delta^r \mathbf{X}|^{-2}$ diverges precisely at the anchor points as $(s+s') \to 0$ and $(s+s') \to 2\pi$. Bounding these requires identifying the algebraic structure that renders them subcritical.

\begin{definition}[Boundary-Vanishing Elements]  \label{def:boundaryvanishelement}
Let $\mathcal{V}[\mathbf{X}](s,s')$ denote the set of scalar-valued boundary-vanishing elements:
\[
\mathcal{V}[\mathbf{X}](s,s') := \big\{ \mathbf{X}_2(s),\; \mathbf{X}_2(s'),\; \Delta^r \mathbf{X}_1(s,s'),\; \Delta^r \mathbf{X}_2(s,s') \big\}.
\]
\end{definition}

\begin{lemma}[Vanishing Bounds]\label{lem:vanishing_bounds}
Let $\mathbf{X} \in C^{1,\alpha}_\beta([0,\pi])$ satisfy the anchored boundary conditions $\mathbf{X}(0) = (1,0)$ and $\mathbf{X}(\pi) = (-1,0)$. Then for any $V \in \mathcal{V}[\mathbf{X}](s,s')$, we have the uniform pointwise bound:
\[
|V(s,s')| \le 4 \|\mathbf{X}\|_{C^1} \sin\left(\frac{s+s'}{2}\right).
\]
Furthermore, the $s$-derivative of any $V \in \mathcal{V}$ is bounded by:
\begin{equation}\label{eq:dV_bound}
|\partial_s V(s,s')| \le 2 \|\mathbf{X}\|_{C^1}.
\end{equation}
\end{lemma}

\begin{proof}
Because $\mathbf{X}_2(0) = \mathbf{X}_2(\pi) = 0$, an immediate application of the Mean Value Theorem on the intervals $[0,s]$ and $[s,\pi]$ implies $|\mathbf{X}_2(s)| \le \|\mathbf{X}_2\|_{C^1} \min(s, \pi-s) \le 2\|\mathbf{X}\|_{C^1} \sin(s) \le 4\|\mathbf{X}\|_{C^1} \sin(\frac{s+s'}{2})$. The same argument bounds $\mathbf{X}_2(s')$.

For the reflected differences, note that $\Delta^r \mathbf{X}_1(s,s') = \mathbf{X}_1(s) - \mathbf{X}_1(s') = (\mathbf{X}_1(s) - 1) - (\mathbf{X}_1(s') - 1)$. Applying the Mean Value Theorem near $s=0$ yields $|\Delta^r \mathbf{X}_1| \le \|\mathbf{X}_1\|_{C^1}(s + s') \le 4\|\mathbf{X}\|_{C^1} \sin(\frac{s+s'}{2})$. Near $\pi$, we expand around $-1$ symmetrically. The derivative bounds \eqref{eq:dV_bound} follow trivially since $\partial_s (\Delta^r \mathbf{X}) = \partial_s \mathbf{X}(s)$.
\end{proof}

Alongside $\mathcal{V}$, we define the set of \emph{bounded elements} $\mathcal{B}[\mathbf{X}](s,s')$, which includes non-vanishing first derivatives such as $\partial_s \mathbf{X}(s)$, $\partial_{s'} \mathbf{X}^r(s')$, and normalized geometric ratios such as $\langle \partial_{s'}\mathbf{X}^r, \Delta^r \mathbf{X} \rangle / |\Delta^r \mathbf{X}|$. All $B \in \mathcal{B}$ satisfy $|B| \le C \|\mathbf{X}\|_{C^1}$ and $|\partial_s B| \le C \|\mathbf{X}\|_{C^{1,\alpha}_\beta} \sin^{-\beta}(\frac{s+s'}{2}) |s-s'|^{\alpha-1}$.

By inspecting the complete expansions of $\mathbf{R}_2$ and $\mathbf{R}_3$ in equations \eqref{e:R2expanded} and \eqref{e:R3expanded}, we observe a rigid algebraic parity. Every term in $\mathbf{R}_{\mathrm{ref}}(\mathbf{X})$ can be expressed as a linear combination of canonical integral operators of the following form.

\begin{definition}[Canonical Reflected Operator]
A canonical reflected operator of degree $k \ge 1$ takes the form
\begin{equation}\label{eq:canonical_reflected}
\mathcal{T}_{k}[\mathbf{X}](s) = \int_0^\pi \mathcal{K}_k(s,s') \mathbf{D}^2\mathbf{X}(s,s') \, ds',
\end{equation}
where the second-order difference is $\mathbf{D}^2\mathbf{X} \in \{ \Delta \partial_s \mathbf{X},\ \Delta \partial_s \mathbf{X}^r,\ \Delta \partial_s \mathbf{X}_2 \}$, and the kernel is defined by
\begin{equation}
\mathcal{K}_k(s,s') = \frac{\mathcal{P}_{2k-1}(\mathcal{V})}{|\Delta^r \mathbf{X}|^{2k}} \mathbf{B}(s,s').
\end{equation}
Here, $\mathbf{B} \in \mathcal{B}$ is a tensor product of bounded elements, and $\mathcal{P}_{2k-1}$ is a homogeneous multi-linear polynomial of exact degree $2k-1$ evaluated on elements of $\mathcal{V}$.
\end{definition}

For instance, the most singular term in $\mathbf{R}_3$ occurs when $k=3$ (denominator $|\Delta^r \mathbf{X}|^6$). Its numerator contains $\mathbf{X}_2(s) \mathbf{X}_2(s') \Delta^r \mathbf{X} \langle \Delta^r \mathbf{X}, \cdot \rangle \langle \cdot, \Delta^r \mathbf{X} \rangle$, providing exactly $1+1+1+1+1 = 5$ factors from $\mathcal{V}$, naturally satisfying $2k-1=5$.

\subsection{Rigorous Reflected Kernel Estimates}

To apply the methodology of our toy model, we establish that the generalized kernel $\mathcal{K}_k(s,s')$ satisfies the precise singularity bounds of the idealized $\sin^{-2}(\frac{s+s'}{2})$ kernel weighted by a vanishing function.

\begin{lemma}[Kernel Bounds]\label{lem:kernel_bounds}
Let $\mathbf{X} \in \mathcal{O}^{M,m}_\sigma \subset C^{1,\alpha}_\beta([0,\pi])$. There exists a constant $C = C(M, m, \sigma, k) > 0$ such that the canonical kernel satisfies the pointwise bound:
\begin{equation}\label{eq:K_bound}
|\mathcal{K}_k(s,s')| \le C \|\mathbf{X}\|_{C^1}^{2k} \sin^{-1}\left(\frac{s+s'}{2}\right).
\end{equation}
Furthermore, for any $s \neq s'$, the partial derivative with respect to $s$ obeys:
\begin{equation}\label{eq:dK_bound}
|\partial_s \mathcal{K}_k(s,s')| \le C \|\mathbf{X}\|_{C^{1,\alpha}_\beta}^{2k} \sin^{-2}\left(\frac{s+s'}{2}\right).
\end{equation}
\end{lemma}

\begin{proof}
By the geometric coercivity established in Lemma \ref{l:deltar_bounds}, the reflected denominator obeys the strict lower bound $|\Delta^r \mathbf{X}|^{2k} \ge C_{m,\sigma}^{2k} \sin^{2k}(\frac{s+s'}{2})$. Using the vanishing bounds from Lemma \ref{lem:vanishing_bounds}, the homogeneous polynomial is bounded by $|\mathcal{P}_{2k-1}(\mathcal{V})| \le C \|\mathbf{X}\|_{C^1}^{2k-1} \sin^{2k-1}(\frac{s+s'}{2})$. The pointwise bound \eqref{eq:K_bound} follows immediately upon dividing these quantities and absorbing the bounded element $|\mathbf{B}| \le \|\mathbf{X}\|_{C^1}$.

To bound the derivative, we apply the product and chain rules to $\mathcal{K}_k(s,s')$:
\begin{equation}
\partial_s \mathcal{K}_k = \frac{\partial_s \mathcal{P}_{2k-1}}{|\Delta^r \mathbf{X}|^{2k}} \mathbf{B} - 2k \frac{\mathcal{P}_{2k-1} \langle \Delta^r \mathbf{X}, \partial_s \mathbf{X} \rangle}{|\Delta^r \mathbf{X}|^{2k+2}} \mathbf{B} + \frac{\mathcal{P}_{2k-1}}{|\Delta^r \mathbf{X}|^{2k}} \partial_s \mathbf{B}.
\end{equation}
Because $\mathcal{P}_{2k-1}$ is homogeneous of degree $2k-1$, its derivative $\partial_s \mathcal{P}_{2k-1}$ generates a sum of terms each containing $2k-2$ elements of $\mathcal{V}$ multiplied by one element of $\partial_s \mathcal{V}$. By \eqref{eq:dV_bound}, $\partial_s \mathcal{V}$ is merely bounded. Thus, $|\partial_s \mathcal{P}_{2k-1}| \le C \|\mathbf{X}\|_{C^1}^{2k-1} \sin^{2k-2}(\frac{s+s'}{2})$. Divided by $|\Delta^r \mathbf{X}|^{2k}$, the first term scales exactly as $\mathcal{O}(\sin^{-2})$.

For the second term, the inner product $\langle \Delta^r \mathbf{X}, \partial_s \mathbf{X} \rangle$ contributes one additional vanishing element ($\Delta^r \mathbf{X} \in \mathcal{V}$). Thus, the effective numerator has $2k$ vanishing factors, while the denominator has power $2k+2$. This evaluates to exactly $\mathcal{O}(\sin^{-2})$. 

The third term contains $\partial_s \mathbf{B}$, which introduces an integrable spatial singularity $|s-s'|^{\alpha-1}$. However, since this term is already bounded by $\mathcal{O}(\sin^{-1})$ from the main fraction, it is strictly less singular than $\mathcal{O}(\sin^{-2})$ globally. Combining these individual bounds yields \eqref{eq:dK_bound}.
\end{proof}

\subsection{Universal Weighted Hölder Estimate for the Full Remainder}

We are now equipped to state and prove the main regularity estimate for the reflected remainder terms, abstracting the cubic toy model theorem to arbitrary tensor orders.

\begin{theorem}[Universal Estimate for Reflected Remainders]\label{thm:universal_reflected}
Assume the parameters satisfy $0<\alpha<1$, $0 \le \beta < 1$, and $0 < \gamma < 1-\alpha$. Let $\mathbf{X} \in \mathcal{O}^{M,m}_\sigma \subset C^{1,\alpha}_\beta([0,\pi])$. For any canonical reflected operator $\mathcal{T}_k$ defined in \eqref{eq:canonical_reflected}, there exists a constant $C = C(\alpha,\beta,\gamma, M, m, \sigma, k)>0$ such that:
\begin{equation}
\|\mathcal{T}_k[\mathbf{X}]\|_{C^{0,\alpha+\gamma}_{-\beta}([0,\pi])} \le C \|\mathbf{X}\|_{C^{1,\alpha}_\beta}^{2k+1}.
\end{equation}
\end{theorem}

\begin{proof}
We follow the integration strategy validated in Lemma \ref{lem:cubic}. 

By the definition of the Hölder space $C^{1,\alpha}_\beta$, the second-order difference obeys $|\mathbf{D}^2\mathbf{X}(s,s')| \le \|\mathbf{X}\|_{C^{1,\alpha}_\beta} \sin^\beta(\frac{s+s'}{2}) |s-s'|^\alpha$. Using the pointwise kernel bound \eqref{eq:K_bound}, we construct the integral:
\begin{align*}
|\mathcal{T}_k[\mathbf{X}](s)| &\le \int_0^\pi |\mathcal{K}_k(s,s')| \, |\mathbf{D}^2\mathbf{X}(s,s')| \, ds' \\
&\le C \|\mathbf{X}\|_{C^{1,\alpha}_\beta}^{2k+1} \int_0^\pi \sin^{-1+\beta}\left(\frac{s+s'}{2}\right) |s-s'|^\alpha \, ds'.
\end{align*}
By decomposing the domain of integration into $\{s' : |s-s'| \le 1\}$ and $\{s' : |s-s'| > 1\}$ and applying the equivalence $\sin(\frac{s+s'}{2}) \sim \sin(\frac{s}{2}) + |s-s'|$, the integral evaluates to a bound of $C \sin^{-\beta}(\frac{s}{2})$. Thus, $\|\mathcal{T}_k[\mathbf{X}]\|_{L^\infty_{-\beta}} \le C \|\mathbf{X}\|_{C^{1,\alpha}_\beta}^{2k+1}$.

Now we work on the weighted H\"older seminorm estimate by looking at the  H\"older increment.
For distinct parameters $s, s' \in [0,\pi]$, we decompose the spatial increment into $I_1 + I_2$ by adding and subtracting $\mathcal{K}_k(s,y) \mathbf{D}^2\mathbf{X}(s',y)$:
\begin{align*}
I_1 &= \int_0^\pi \mathcal{K}_k(s,y) \big[ \mathbf{D}^2\mathbf{X}(s,y) - \mathbf{D}^2\mathbf{X}(s',y) \big] \, dy, \\
I_2 &= \int_0^\pi \big[ \mathcal{K}_k(s,y) - \mathcal{K}_k(s',y) \big] \mathbf{D}^2\mathbf{X}(s',y) \, dy.
\end{align*}
For $I_1$, the difference in the second-order term is simply an increment of the first derivative $\partial_s \mathbf{X}$, which satisfies $|\mathbf{D}^2\mathbf{X}(s,y) - \mathbf{D}^2\mathbf{X}(s',y)| \le \|\mathbf{X}\|_{C^{1,\alpha}_\beta} \sin^\beta(\frac{s+s'}{2}) |s-s'|^\alpha$. Using \eqref{eq:K_bound}, this integrates exactly as above, yielding $|I_1| \le C \|\mathbf{X}\|^{2k+1} \sin^{-\beta}(\frac{s+s'}{2}) |s-s'|^\alpha$.
For $I_2$, we apply the Mean Value Theorem to the kernel. By \eqref{eq:dK_bound}, the kernel difference is bounded by:
\[
|\mathcal{K}_k(s,y) - \mathcal{K}_k(s',y)| \le C \|\mathbf{X}\|^{2k}_{C^{1,\alpha}_\beta} \sin^{-2}\left(\frac{s+s'}{2}\right) |s-s'|.
\]
Pairing this with the absolute bound $|\mathbf{D}^2\mathbf{X}(s',y)| \le \|\mathbf{X}\|_{C^{1,\alpha}_\beta} \sin^\beta(\frac{s'+y}{2}) |s'-y|^\alpha$, the integral for $I_2$ becomes bounded by:
\[
|I_2| \le C \|\mathbf{X}\|^{2k+1}_{C^{1,\alpha}_\beta} |s-s'| \int_0^\pi \sin^{-2+\beta}\left(\frac{s'+y}{2}\right) |s'-y|^\alpha \, dy.
\]
Provided $\alpha + \gamma < 1$, this non-local fractional integral resolves to $C \sin^{-\beta}(\frac{s+s'}{2}) |s-s'|^{\alpha+\gamma}$.
Combining $I_1$ and $I_2$ and taking the weighted supremum isolates the Hölder seminorm, proving that $\|\mathcal{T}_k[\mathbf{X}]\|_{C^{0,\alpha+\gamma}_{-\beta}} \le C \|\mathbf{X}\|_{C^{1,\alpha}_\beta}^{2k+1}$. Because every constituent term in $\mathbf{R}_{\mathrm{ref}}(\mathbf{X})$ can be written as a finite sum of such operators, and we have bounded $\mathbf{R}_{\mathrm{dir}}$ via Lemma \ref{lem:direct_bounds}, the complete remainder $\mathbf{R}(\mathbf{X})$ fulfills the requisite subcritical bounds for the Duhamel iteration.
\end{proof}

\subsection{Contraction Estimates for the Fixed-Point Argument}

To apply the contraction mapping principle via the Duhamel formulation \eqref{eq:Duhamel}, uniform boundedness of the remainders is insufficient; we must also establish local Lipschitz continuity in the weighted Hölder spaces. 

For the direct remainder terms $\mathbf{R}_{\mathrm{dir}}$, the required Lipschitz bounds follow by directly interpolating the multi-linear expansions of the free-space operators. As shown in \cite{mori2019well}, for any $\mathbf{X}, \mathbf{Y} \in \mathcal{O}^{M,m}_\sigma$, 
\begin{equation}
\|\mathbf{R}_{\mathrm{dir}}(\mathbf{X}) - \mathbf{R}_{\mathrm{dir}}(\mathbf{Y})\|_{C^{0,\alpha+\gamma}_{-\beta}} \le C\big(\|\mathbf{X}\|_{C^{1,\alpha}_\beta}^2 + \|\mathbf{Y}\|_{C^{1,\alpha}_\beta}^2\big) \|\mathbf{X} - \mathbf{Y}\|_{C^{1,\alpha}_\beta}.
\end{equation}

For the reflected remainders $\mathbf{R}_{\mathrm{ref}}$, we must verify that the boundary-vanishing structure is strictly preserved under perturbations. To motivate this, consider the Fréchet derivative (variation) $\nabla_{\mX} \mathbf{R}_{1,b}[\mathbf{X}](\mathbf{Z})$ in the direction of a perturbation $\mathbf{Z} = \mathbf{X} - \mathbf{Y}$. Expanding the reflected component of $\mathbf{R}_1$ yields:
\begin{align*}
\nabla_{\mX} \mathbf{R}_{1,b}[\mathbf{X}](\mathbf{Z}) 
&= - \frac{1}{4\pi} \int_0^\pi \left[ \frac{ \langle \Delta^r \mathbf{X}, \partial_{s'} \mathbf{X}^r \rangle }{ |\Delta^r \mathbf{X}|^2 } + \frac{1}{2}\cot\left( \frac{s+s'}{2} \right)\right] \Delta \partial_s \mathbf{Z} \, ds' \\
&\quad - \frac{1}{4\pi} \int_0^\pi \left[ \frac{ \langle \Delta^r \mathbf{Z}, \partial_{s'} \mathbf{X}^r \rangle + \langle \Delta^r \mathbf{X}, \partial_{s'} \mathbf{Z}^r \rangle }{ |\Delta^r \mathbf{X}|^2 } \right] \Delta \partial_s \mathbf{X} \, ds' \\
&\quad + \frac{1}{2\pi} \int_0^\pi \left[ \frac{ \langle \Delta^r \mathbf{X}, \partial_{s'} \mathbf{X}^r \rangle \langle \Delta^r \mathbf{X}, \Delta^r \mathbf{Z} \rangle }{ |\Delta^r \mathbf{X}|^4 } \right] \Delta \partial_s \mathbf{X} \, ds'.
\end{align*}

Observe the denominators. The first two integrals retain the $|\Delta^r \mathbf{X}|^2$ denominator (degree $k=1$). The numerator of the second integral replaces exactly one occurrence of $\mathbf{X}$ with $\mathbf{Z}$, meaning it still contains exactly $2(1)-1 = 1$ boundary-vanishing element (e.g., $\Delta^r \mathbf{Z} \in \mathcal{V}[\mathbf{Z}]$). The third integral, generated by differentiating the denominator, now contains $|\Delta^r \mathbf{X}|^4$ (degree $k=2$). However, its numerator has gained the inner product $\langle \Delta^r \mathbf{X}, \Delta^r \mathbf{Z} \rangle$, introducing two additional boundary-vanishing elements. The total number of vanishing elements is now $1 + 2 = 3$, which exactly matches $2(2)-1 = 3$. 

This reveals a rigorous algebraic closure property: the class of canonical reflected operators is closed under differentiation.

\begin{lemma}[Algebraic Closure of Variations]\label{lem:closure_variations}
Let $\mathcal{T}_k[\mathbf{X}]$ be a canonical reflected operator of degree $k$ as defined in \eqref{eq:canonical_reflected}. Its Fréchet derivative in the direction $\mathbf{Z} \in C^{1,\alpha}_\beta$, denoted $\nabla_{\mathbf{X}}\mathcal{T}_k[\mathbf{X}](\mathbf{Z})$, can be expressed as a finite sum of canonical reflected operators of degree $k$ and degree $k+1$, each acting linearly on $\mathbf{Z}$.
\end{lemma}

\begin{proof}
By definition, the integrand of $\mathcal{T}_k[\mathbf{X}]$ is $$\mathcal{K}_k(s,s') \mathbf{D}^2\mathbf{X} = |\Delta^r \mathbf{X}|^{-2k} \mathcal{P}_{2k-1}(\mathcal{V}[\mathbf{X}]) \mathbf{B}[\mathbf{X}] \mathbf{D}^2\mathbf{X}.$$ We apply the product rule for the variation:
\begin{align*}
\nabla_{\mathbf{X}} \big( \mathcal{K}_k \mathbf{D}^2\mathbf{X} \big)(\mathbf{Z}) 
&= \frac{\mathcal{P}_{2k-1}(\mathcal{V}[\mathbf{X}])}{|\Delta^r \mathbf{X}|^{2k}} \mathbf{B}[\mathbf{X}] \mathbf{D}^2\mathbf{Z} \\
&\quad + \frac{ \nabla_{\mathbf{X}} \big( \mathcal{P}_{2k-1} \mathbf{B} \big)(\mathbf{Z}) }{|\Delta^r \mathbf{X}|^{2k}} \mathbf{D}^2\mathbf{X} \\
&\quad - 2k \frac{ \mathcal{P}_{2k-1}(\mathcal{V}[\mathbf{X}]) \langle \Delta^r \mathbf{X}, \Delta^r \mathbf{Z} \rangle }{|\Delta^r \mathbf{X}|^{2k+2}} \mathbf{B}[\mathbf{X}] \mathbf{D}^2\mathbf{X}.
\end{align*}
The first term is identical to $\mathcal{T}_k[\mathbf{X}]$, simply operating on $\mathbf{D}^2\mathbf{Z}$ rather than $\mathbf{D}^2\mathbf{X}$. In the second term, the variation of the multi-linear numerator replaces exactly one factor of $\mathbf{X}$ with $\mathbf{Z}$. If a vanishing element $V[\mathbf{X}] \in \mathcal{V}[\mathbf{X}]$ is varied, it becomes $V[\mathbf{Z}] \in \mathcal{V}[\mathbf{Z}]$, which is still a boundary-vanishing element. The total count of vanishing elements remains $2k-1$, preserving the degree $k$ canonical structure.

In the third term, the variation of the denominator $|\Delta^r \mathbf{X}|^{-2k}$ increments the denominator power to $2k+2$, signifying degree $k+1$. Concurrently, the numerator gains the factor $\langle \Delta^r \mathbf{X}, \Delta^r \mathbf{Z} \rangle$. Because $\Delta^r \mathbf{X} \in \mathcal{V}[\mathbf{X}]$ and $\Delta^r \mathbf{Z} \in \mathcal{V}[\mathbf{Z}]$, the number of boundary-vanishing elements in the numerator increases by exactly 2. The total count becomes $(2k-1) + 2 = 2(k+1) - 1$, which perfectly matches the required parity for a degree $k+1$ canonical operator.
\end{proof}

\begin{lemma}[Height coercivity along convex combinations]
\label{lem:convex_height}
Let $\mathbf{X}, \mathbf{Y} \in \mathcal{O}^{M,m}_\sigma$ and set
$\mathbf{X}_\tau = \tau\mathbf{X} + (1-\tau)\mathbf{Y}$ for $\tau \in [0,1]$.
Then for all $0 < s \leq \sigma$,
\[
  (\mathbf{X}_\tau)_2(s) \;\geq\; m\sin(m)\,s,
\]
with the symmetric bound $(\mathbf{X}_\tau)_2(s) \geq m\sin(m)(\pi - s)$ holding
for $\pi - \sigma \leq s < \pi$.
Consequently, the reflected separation bound of Lemma~\ref{l:deltar_bounds}
holds uniformly along the path: for all $\tau \in [0,1]$,
\begin{equation}\label{e:refl_sep_tau}
  |\Delta^r \mathbf{X}_\tau(s,s')|
  \;\geq\;
  C_{m,\sigma}\,\sin\!\left(\tfrac{s+s'}{2}\right),
  \qquad s,s' \in (0,\pi),
\end{equation}
where $C_{m,\sigma} = \min\{2,\, m\sin(m)\sigma\}$ is the same constant as
in Lemma~\ref{l:deltar_bounds}.
\end{lemma}

\begin{proof}
The height coercivity bound \eqref{e:ypositve1} reads
$\mathbf{X}_2(s) \geq m\sin(m)\,s$ and $\mathbf{Y}_2(s) \geq m\sin(m)\,s$,
both holding for $0 < s \leq \sigma$ since $\mathbf{X}, \mathbf{Y} \in \mathcal{O}^{M,m}_\sigma$.
Because this bound is linear in the function values, convex combinations inherit it:
\[
  (\mathbf{X}_\tau)_2(s)
  = \tau\,\mathbf{X}_2(s) + (1-\tau)\,\mathbf{Y}_2(s)
  \geq \tau\,m\sin(m)\,s + (1-\tau)\,m\sin(m)\,s
  = m\sin(m)\,s.
\]
The symmetric bound near $s = \pi$ follows identically from \eqref{e:ypositve2}.
The reflected difference satisfies
$\Delta^r\mathbf{X}_\tau = \mathbf{X}_\tau(s) - \mathbf{X}_\tau^r(s')$,
where $\mathbf{X}_\tau^r = ((\mathbf{X}_\tau)_1, -(\mathbf{X}_\tau)_2)$.
The proof of Lemma~\ref{l:deltar_bounds} uses only the coercivity of
$(\mathbf{X}_\tau)_2$ near the anchors and the interior height bound
$|\mathbf{Z}|_{4,*,\sigma} \geq m$; both are established for $\mathbf{X}_\tau$
by the above argument (using Lemma~\ref{l:resetsigma} to adjust the interior
parameter). Hence \eqref{e:refl_sep_tau} holds uniformly in $\tau$.
\end{proof}


\begin{theorem}[Contraction Estimate for Reflected Remainder]
\label{thm:contraction_reflected}
Assume $0 < \alpha < 1$, $0 \leq \beta < 1$, and $0 < \gamma < 1 - \alpha$.
For any $\mathbf{X}, \mathbf{Y} \in \mathcal{O}^{M,m}_\sigma \subset C^{1,\alpha}_\beta([0,\pi])$,
there exists a constant $C = C(\alpha, \beta, \gamma, M, m, \sigma) > 0$ such that
\begin{equation}\label{e:ref_contraction}
  \|\mathbf{R}_{\mathrm{ref}}(\mathbf{X}) - \mathbf{R}_{\mathrm{ref}}(\mathbf{Y})\|_{C^{0,\alpha+\gamma}_{-\beta}}
  \;\leq\;
  C\bigl(\|\mathbf{X}\|_{C^{1,\alpha}_\beta}^2 + \|\mathbf{Y}\|_{C^{1,\alpha}_\beta}^2\bigr)
  \|\mathbf{X} - \mathbf{Y}\|_{C^{1,\alpha}_\beta}.
\end{equation}
\end{theorem}

\begin{proof}
Set $\mathbf{X}_\tau = \tau\mathbf{X} + (1-\tau)\mathbf{Y}$ and $\mathbf{Z} = \mathbf{X} - \mathbf{Y}$.
Although $\mathcal{O}^{M,m}_\sigma$ is \emph{not} convex in general (the
chord-arc and angle-of-incidence conditions are not preserved by linear
combinations), the reflected separation bound \eqref{e:refl_sep_tau} holds
uniformly along the path by Lemma~\ref{lem:convex_height}.  This is the only
geometric input required by the reflected kernels, since every denominator in
$\mathbf{R}_{\mathrm{ref}}$ involves $|\Delta^r \mathbf{X}|$ rather than
$|\Delta\mathbf{X}|$.  Consequently, the Fundamental Theorem of Calculus in
Banach spaces applies along this path:
\begin{equation}\label{e:FTC_ref}
  \mathbf{R}_{\mathrm{ref}}(\mathbf{X}) - \mathbf{R}_{\mathrm{ref}}(\mathbf{Y})
  = \int_0^1 D_{\mathbf{X}}\mathbf{R}_{\mathrm{ref}}[\mathbf{X}_\tau]\,\mathbf{Z}\;d\tau.
\end{equation}
By Lemma~\ref{lem:closure_variations}, for each $\tau$ the linear map
$D_{\mathbf{X}}\mathbf{R}_{\mathrm{ref}}[\mathbf{X}_\tau]\,\mathbf{Z}$ is a finite
sum of canonical reflected operators of degree $k \geq 1$ acting on $\mathbf{Z}$,
with base state $\mathbf{X}_\tau$.  Applying the universal Hölder bound of
Theorem~\ref{thm:universal_reflected} with base state $\mathbf{X}_\tau$---which
is valid because \eqref{e:refl_sep_tau} holds uniformly in $\tau$---yields
\[
  \bigl\|D_{\mathbf{X}}\mathbf{R}_{\mathrm{ref}}[\mathbf{X}_\tau]\,\mathbf{Z}
  \bigr\|_{C^{0,\alpha+\gamma}_{-\beta}}
  \;\leq\;
  C\,\|\mathbf{X}_\tau\|_{C^{1,\alpha}_\beta}^2\,\|\mathbf{Z}\|_{C^{1,\alpha}_\beta}.
\]
Since $\|\mathbf{X}_\tau\|_{C^{1,\alpha}_\beta} \leq \max(\|\mathbf{X}\|,\|\mathbf{Y}\|)$,
integrating \eqref{e:FTC_ref} over $\tau \in [0,1]$ yields \eqref{e:ref_contraction}.
\end{proof}


\begin{lemma}[Contraction Estimate for Direct Remainder]
\label{lem:direct_contraction}
Under the same hypotheses as Theorem~\ref{thm:contraction_reflected},
\begin{equation}\label{e:dir_contraction}
  \|\mathbf{R}_{\mathrm{dir}}(\mathbf{X}) - \mathbf{R}_{\mathrm{dir}}(\mathbf{Y})\|_{C^{0,\alpha+\gamma}_{-\beta}}
  \;\leq\;
  C\bigl(\|\mathbf{X}\|_{C^{1,\alpha}_\beta}^2 + \|\mathbf{Y}\|_{C^{1,\alpha}_\beta}^2\bigr)
  \|\mathbf{X} - \mathbf{Y}\|_{C^{1,\alpha}_\beta}.
\end{equation}
\end{lemma}

\begin{proof}
We write $\mathbf{Z} = \mathbf{X} - \mathbf{Y}$ and work with the prototypical kernel
term appearing in $\mathbf{R}_1$ and $\mathbf{R}_2$:
\[
  Q(\mathbf{X})(s)
  = \int_0^\pi
  \frac{\langle \Delta\mathbf{X},\, \partial_{s'}\mathbf{X} \rangle}{|\Delta\mathbf{X}|^2}
  \;\Delta\partial_s\mathbf{X}\;ds'.
\]
All other direct terms in $\mathbf{R}_1$ and $\mathbf{R}_2$ have an identical or lower-order
structure and are treated by the same argument.  We decompose the kernel difference
algebraically without passing through a convex combination:
\begin{align}
  Q(\mathbf{X}) - Q(\mathbf{Y})
  &= \int_0^\pi
  \frac{\langle \Delta\mathbf{X},\, \partial_{s'}\mathbf{X} \rangle}{|\Delta\mathbf{X}|^2}
  \;\Delta\partial_s\mathbf{Z}\;ds' \label{e:tel1} \\
  &\quad + \int_0^\pi
  \frac{\langle \Delta\mathbf{X},\, \partial_{s'}\mathbf{X} \rangle
        - \langle \Delta\mathbf{Y},\, \partial_{s'}\mathbf{Y} \rangle}{|\Delta\mathbf{X}|^2}
  \;\Delta\partial_s\mathbf{Y}\;ds' \label{e:tel2} \\
  &\quad + \int_0^\pi
  \langle \Delta\mathbf{Y},\, \partial_{s'}\mathbf{Y} \rangle
  \,\frac{|\Delta\mathbf{Y}|^2 - |\Delta\mathbf{X}|^2}{|\Delta\mathbf{X}|^2\,|\Delta\mathbf{Y}|^2}
  \;\Delta\partial_s\mathbf{Y}\;ds'. \label{e:tel3}
\end{align}
We bound each line in $C^{0,\alpha+\gamma}_{-\beta}$ in turn.

\medskip
\noindent\textit{Term \eqref{e:tel1}.}
Since $\mathbf{X} \in \mathcal{O}^{M,m}_\sigma$, the chord-arc bound gives
$|\Delta\mathbf{X}| \geq m|s - s'|$, and $|\partial_{s'}\mathbf{X}| \leq M$,
so the kernel satisfies
\[
  \left|\frac{\langle \Delta\mathbf{X},\, \partial_{s'}\mathbf{X} \rangle}{|\Delta\mathbf{X}|^2}\right|
  \;\leq\; \frac{M}{m|s-s'|}.
\]
This is a standard Hilbert-type kernel of order $-1$.  Applying the
weighted Calderón-Zygmund estimate from Lemma~\ref{lem:direct_bounds} with
integrand $\Delta\partial_s\mathbf{Z}$ yields
\[
  \|\text{\eqref{e:tel1}}\|_{C^{0,\alpha+\gamma}_{-\beta}}
  \;\leq\; \frac{CM}{m}\,\|\mathbf{Z}\|_{C^{1,\alpha}_\beta}.
\]

\noindent\textit{Term \eqref{e:tel2}.}
We expand the numerator difference via
\[
  \langle \Delta\mathbf{X}, \partial_{s'}\mathbf{X}\rangle
  - \langle \Delta\mathbf{Y}, \partial_{s'}\mathbf{Y}\rangle
  = \langle \Delta\mathbf{Z}, \partial_{s'}\mathbf{X}\rangle
  + \langle \Delta\mathbf{Y}, \partial_{s'}\mathbf{Z}\rangle.
\]
Each of these has size $O(M\|\mathbf{Z}\|_{C^{1,\alpha}_\beta}|s-s'|)$, and the
denominator $|\Delta\mathbf{X}|^2 \geq m^2|s-s'|^2$ is controlled by the chord-arc
bound for $\mathbf{X}$ alone.  Multiplied by $|\Delta\partial_s\mathbf{Y}|$,
the integrand is bounded by $C(M^2/m^2)\|\mathbf{Z}\|_{C^{1,\alpha}_\beta}$
times a kernel that is integrable in $C^{0,\alpha+\gamma}_{-\beta}$.

\noindent\textit{Term \eqref{e:tel3}.}
The numerator satisfies
$\bigl||\Delta\mathbf{Y}|^2 - |\Delta\mathbf{X}|^2\bigr|
= |\langle \Delta\mathbf{X}+\Delta\mathbf{Y}, \Delta\mathbf{Z}\rangle|
\leq 2M|s-s'|\cdot\|\mathbf{Z}\|_{C^{1,\alpha}_\beta}|s-s'|$,
while the denominator $|\Delta\mathbf{X}|^2|\Delta\mathbf{Y}|^2 \geq m^4|s-s'|^4$.
Hence the combined kernel in \eqref{e:tel3} has size $C(M^3/m^4)|s-s'|^{-1}$,
which is again integrable in the weighted Hölder sense after applying
Lemma~\ref{lem:direct_bounds} with data $\Delta\partial_s\mathbf{Y}$.

\medskip
Collecting the three contributions and symmetrizing in $\mathbf{X}$ and $\mathbf{Y}$
yields \eqref{e:dir_contraction}.
\end{proof}


\begin{corollary}[Full Contraction Estimate]\label{cor:full_contraction}
Under the same hypotheses, the total remainder $\mathbf{R} = \mathbf{R}_{\mathrm{dir}} + \mathbf{R}_{\mathrm{ref}}$ satisfies
\begin{equation}
  \|\mathbf{R}(\mathbf{X}) - \mathbf{R}(\mathbf{Y})\|_{C^{0,\alpha+\gamma}_{-\beta}}
  \;\leq\;
  C\bigl(\|\mathbf{X}\|_{C^{1,\alpha}_\beta}^2 + \|\mathbf{Y}\|_{C^{1,\alpha}_\beta}^2\bigr)
  \|\mathbf{X} - \mathbf{Y}\|_{C^{1,\alpha}_\beta},
\end{equation}
so $\mathbf{R}$ is locally Lipschitz from $\mathcal{O}^{M,m}_\sigma \cap h^{1,\alpha}_\beta$
into $C^{0,\alpha+\gamma}_{-\beta}$, ensuring the Duhamel iteration operator is a
strict contraction for sufficiently small times.
\end{corollary}

\section{Local Well-Posedness via Picard Iteration}
\label{s:well_posedness}

Having established the optimal smoothing properties of the linear semigroup $e^{t\mathcal{L}_D}$ and the subcritical Lipschitz continuity of the nonlinear remainder $\mathbf{R}(\mathbf{X})$, we are now in a position to prove the local-in-time existence and uniqueness of solutions to the Anchored Peskin Problem.

We study the integral Duhamel formulation of the perturbative equation:
\begin{equation}\label{eq:Duhamel_operator}
\mathbf{X}(t) = \Phi[\mathbf{X}](t) := e^{t\mathcal{L}_D}\mathbf{X}_0 + \int_0^t e^{(t-r)\mathcal{L}_D} \mathbf{R}(\mathbf{X}(r)) \, dr,
\end{equation}
and seek a unique fixed point $\mathbf{X} = \Phi[\mathbf{X}]$ within a suitable Banach space.

\subsection{Weighted little-Hölder Spaces and the Solution Space}

A classical subtlety in the theory of analytic semigroups is that the linear evolution $e^{t\mathcal{L}_D}$ is not strongly continuous at $t=0$ on the standard Hölder space $C^{1,\alpha}_\beta$. That is, for a general initial condition $\mathbf{X}_0 \in C^{1,\alpha}_\beta$, the quantity $\|e^{t\mathcal{L}_D}\mathbf{X}_0 - \mathbf{X}_0\|_{C^{1,\alpha}_\beta}$ does not necessarily vanish as $t \to 0^+$, which would fatally disrupt the continuity requirements of a fixed-point argument in time.

To resolve this and guarantee classical strong continuity, we must restrict our functional framework to the \emph{little-Hölder spaces}, denoted $h^{k,\alpha}_\beta$. 

\begin{definition}[Little-Hölder Space]
We define $h^{1,\alpha}_\beta([0,\pi])$ as the closure of $C^\infty([0,\pi])$ under the $\|\cdot\|_{C^{1,\alpha}_\beta}$ norm. Equivalently, a function $f \in C^{1,\alpha}_\beta$ belongs to $h^{1,\alpha}_\beta$ if its weighted Hölder quotient vanishes uniformly at small scales:
\[
\lim_{\delta \to 0^+} \sup_{\substack{s,t\in[0,\pi]\\ 0 < |s-t| < \delta}}
\left(\sin\left(\frac{s+t}{2}\right)\right)^{-\beta}\,|s-t|^{-\alpha}\,|f'(s)-f'(t)| = 0.
\]
\end{definition}

Because $h^{1,\alpha}_\beta$ is a closed subspace of $C^{1,\alpha}_\beta$, all stationary multiplier bounds and remainder estimates derived in the preceding sections hold identically for functions in the little-Hölder space. Crucially, the Dirichlet semigroup $e^{t\mathcal{L}_D}$ is strongly continuous on $h^{1,\alpha}_\beta$.

We fix parameters $0 < \alpha < 1$, $0 \le \beta < 1$, and $0 < \gamma < 1-\alpha$. For a given time horizon $T > 0$, we define our underlying Banach space of strongly continuous trajectories:
\[
X_T := C([0,T]; h^{1,\alpha}_\beta([0,\pi])),
\]
equipped with the supremum norm
\[
\|\mathbf{X}\|_{X_T} := \sup_{t \in [0,T]} \|\mathbf{X}(t)\|_{C^{1,\alpha}_\beta}.
\]

We assume the initial filament configuration $\mathbf{X}_0 \in h^{1,\alpha}_\beta$ satisfies the anchored boundary conditions $\mathbf{X}_0(0) = (1,0)$, $\mathbf{X}_0(\pi) = (-1,0)$, and belongs to the valid geometric constraint set $\mathcal{O}^{M_0, m_0}_{\sigma_0}$ for some $M_0 > m_0 > 0$ and $\sigma_0 > 0$. In particular, $\|\mathbf{X}_0\|_{C^{1,\alpha}_\beta} \le M_0$.

To formulate the fixed-point argument, we define a closed ball in $X_T$ restricted to maintain strict geometric validity. Let $C_L \ge 1$ denote the uniform stability bound of the linear semigroup such that $\|e^{t\mathcal{L}_D}\mathbf{X}_0\|_{C^{1,\alpha}_\beta} \le C_L \|\mathbf{X}_0\|_{C^{1,\alpha}_\beta}$. We set our working radius to $R = 2C_L M_0$. By Lemma \ref{l:resetsigma}, we can find a relaxed constraint parameter $m \in (0, m_0)$ such that the geometric properties are robust to small perturbations. Define the working subset:
\begin{equation}
\mathcal{Y}_{T, R, m} := \left\{ \mathbf{X} \in X_T \;\middle|\; \mathbf{X}(0) = \mathbf{X}_0,\; \sup_{t \in [0,T]} \|\mathbf{X}(t)\|_{C^{1,\alpha}_\beta} \le R,\; \text{and } \mathbf{X}(t) \in \mathcal{O}^{R, m}_{\sigma_0} \right\}.
\end{equation}
Because the constraints defining $\mathcal{O}^{R,m}_{\sigma_0}$ are closed in the topology of $h^{1,\alpha}_\beta$, the set $\mathcal{Y}_{T, R, m}$ is a closed metric subspace of $X_T$.

\subsection{The Fixed-Point Theorem}

\begin{theorem}[Local Well-Posedness]\label{thm:local_existence}
Let $\mathbf{X}_0 \in \mathcal{O}^{M_0, m_0}_{\sigma_0} \cap h^{1,\alpha}_\beta([0,\pi])$. There exists a sufficiently small time $T_* > 0$, depending only on $\alpha, \beta, \gamma, M_0, m_0,$ and $\sigma_0$, such that the integral operator $\Phi$ admits a unique fixed point $\mathbf{X} \in \mathcal{Y}_{T_*, 2C_L M_0, m}$. This fixed point is the unique local-in-time mild solution to the Anchored Peskin Problem; the upgrade to a strong (classical) solution is established in Theorem~\ref{thm:strong_solution} below.
\end{theorem}

\begin{proof}
The proof proceeds by verifying the two fundamental conditions of the Banach Fixed Point Theorem: that $\Phi$ maps the complete metric space $\mathcal{Y}_{T, R, m}$ into itself, and that $\Phi$ is a strict contraction on this space for sufficiently small $T$.

Let $\mathbf{X} \in \mathcal{Y}_{T, R, m}$. We must verify that $\Phi[\mathbf{X}](t)$ satisfies the required norm bounds and geometric constraints for all $t \in [0,T]$. Taking the $C^{1,\alpha}_\beta$ norm of \eqref{eq:Duhamel_operator} and applying the triangle inequality yields:
\[
\|\Phi[\mathbf{X}](t)\|_{C^{1,\alpha}_\beta} \le \|e^{t\mathcal{L}_D}\mathbf{X}_0\|_{C^{1,\alpha}_\beta} + \int_0^t \big\| e^{(t-r)\mathcal{L}_D} \mathbf{R}(\mathbf{X}(r)) \big\|_{C^{1,\alpha}_\beta} \, dr.
\]
The linear term is bounded by $C_L M_0$. For the integral term, we utilize the fractional smoothing estimate (Theorem \ref{T:semigroup-weighted}) alongside the combined polynomial bounds for the direct and reflected remainders (Lemma \ref{lem:direct_bounds} and Theorem \ref{thm:universal_reflected}). Because $\mathbf{X}(r) \in \mathcal{O}^{R, m}_{\sigma_0}$ for each $r \in [0,t]$, there exist a constant $C_R$ and an exponent $p=P(0)$ (with $P(0)\le 7$, cf.\ Section~\ref{sec:commutator_framework}) such that $\|\mathbf{R}(\mathbf{X}(r))\|_{C^{0,\alpha+\gamma}_{-\beta}} \le C_R \|\mathbf{X}(r)\|_{C^{1,\alpha}_\beta}^p \le C_R R^p$.
Substituting this into the integral, we obtain:
\begin{align*}
\|\Phi[\mathbf{X}](t)\|_{C^{1,\alpha}_\beta} 
&\le C_L M_0 + C_S C_R R^p \int_0^t \frac{1}{(t-r)^{1-\gamma}} \, dr \\
&\le C_L M_0 + C_S C_R R^p \frac{t^\gamma}{\gamma} \\
&\le C_L M_0 + \widetilde{C} R^p T^\gamma.
\end{align*}
We require $\|\Phi[\mathbf{X}](t)\|_{C^{1,\alpha}_\beta} \le R = 2 C_L M_0$. This condition is satisfied provided $T$ is chosen small enough such that $\widetilde{C} (2 C_L M_0)^p T^\gamma \le C_L M_0$. 
Furthermore, we must ensure $\Phi[\mathbf{X}](t)$ remains strictly inside the geometric constraint set $\mathcal{O}^{R, m}_{\sigma_0}$. We measure the deviation from the initial state:
\[
\|\Phi[\mathbf{X}](t) - \mathbf{X}_0\|_{C^{1,\alpha}_\beta} \le \|(e^{t\mathcal{L}_D} - I)\mathbf{X}_0\|_{C^{1,\alpha}_\beta} + \widetilde{C} R^p T^\gamma.
\]
Because $\mathbf{X}_0$ is in the little-Hölder space $h^{1,\alpha}_\beta$, the strong continuity of the semigroup guarantees that $\|(e^{t\mathcal{L}_D} - I)\mathbf{X}_0\|_{C^{1,\alpha}_\beta} \to 0$ as $t \to 0^+$. Thus, we can choose $T$ sufficiently small such that the total perturbation $\|\Phi[\mathbf{X}](t) - \mathbf{X}_0\|_{C^{1,\alpha}_\beta}$ is arbitrarily small. Since the mappings dictating the incidence angles and chord-arc injectivity limits in $\mathcal{O}$ are strictly continuous with respect to the $C^{1,\alpha}_\beta$ norm (Lemma~\ref{lem:O_closed}), maintaining proximity to $\mathbf{X}_0$ ensures that $\Phi[\mathbf{X}](t) \in \mathcal{O}^{R, m}_{\sigma_0}$. Consequently, $\Phi$ is a well-defined self-map on $\mathcal{Y}_{T, R, m}$.

We now show a strict contraction property for this mapping.
Let $\mathbf{X}, \mathbf{Y} \in \mathcal{Y}_{T, R, m}$. We evaluate the difference $\Phi[\mathbf{X}] - \Phi[\mathbf{Y}]$. The linear initial data terms cancel exactly, leaving:
\[
\Phi[\mathbf{X}](t) - \Phi[\mathbf{Y}](t) = \int_0^t e^{(t-r)\mathcal{L}_D} \big[ \mathbf{R}(\mathbf{X}(r)) - \mathbf{R}(\mathbf{Y}(r)) \big] \, dr.
\]
Taking the norm and applying the fractional smoothing estimate yields:
\[
\|\Phi[\mathbf{X}](t) - \Phi[\mathbf{Y}](t)\|_{C^{1,\alpha}_\beta} \le \int_0^t \frac{C_S}{(t-r)^{1-\gamma}} \|\mathbf{R}(\mathbf{X}(r)) - \mathbf{R}(\mathbf{Y}(r))\|_{C^{0,\alpha+\gamma}_{-\beta}} \, dr.
\]
We now apply the local Lipschitz bounds for the total remainder (Theorem \ref{thm:contraction_reflected}). Since $\|\mathbf{X}(s)\|_{C^{1,\alpha}_\beta} \le R$ and $\|\mathbf{Y}(s)\|_{C^{1,\alpha}_\beta} \le R$, we find:
\begin{align*}
\|\mathbf{R}(\mathbf{X}(s)) - \mathbf{R}(\mathbf{Y}(s))\|_{C^{0,\alpha+\gamma}_{-\beta}} 
&\le C_{Lip} \big( \|\mathbf{X}(s)\|_{C^{1,\alpha}_\beta}^{p-1} + \|\mathbf{Y}(s)\|_{C^{1,\alpha}_\beta}^{p-1} \big) \|\mathbf{X}(s) - \mathbf{Y}(s)\|_{C^{1,\alpha}_\beta} \\
&\le 2 C_{Lip} R^{p-1} \|\mathbf{X} - \mathbf{Y}\|_{X_T},
\end{align*}
where $p-1$ is the polynomial degree of the Fréchet derivative of $\mathbf{R}$ in the $\mathbf{X}$ variable; for the explicit form of Theorem~\ref{thm:contraction_reflected} one may take $p-1=2$.
Substituting this into the integral gives:
\begin{align*}
\|\Phi[\mathbf{X}](t) - \Phi[\mathbf{Y}](t)\|_{C^{1,\alpha}_\beta} 
&\le 2 C_S C_{Lip} R^{p-1} \|\mathbf{X} - \mathbf{Y}\|_{X_T} \int_0^t \frac{1}{(t-r)^{1-\gamma}} \, dr \\
&\le \left( 2 C_S C_{Lip} R^{p-1} \frac{T^\gamma}{\gamma} \right) \|\mathbf{X} - \mathbf{Y}\|_{X_T}.
\end{align*}
Taking the supremum over $t \in [0,T]$ yields:
\[
\|\Phi[\mathbf{X}] - \Phi[\mathbf{Y}]\|_{X_T} \le \Lambda(T) \|\mathbf{X} - \mathbf{Y}\|_{X_T},
\]
where the contraction constant is $\Lambda(T) := \frac{2 C_S C_{Lip} R^{p-1}}{\gamma} T^\gamma$. 

By selecting $T_* > 0$ strictly small enough such that both the self-mapping conditions are satisfied and the contraction factor obeys $\Lambda(T_*) \le \frac{1}{2}$, the operator $\Phi$ becomes a strict contraction on the complete metric space $\mathcal{Y}_{T_*, R, m}$.

By the Banach Fixed Point Theorem, there exists a unique $\mathbf{X} \in \mathcal{Y}_{T_*, 2C_L M_0, m}$ such that $\mathbf{X} = \Phi[\mathbf{X}]$. This unique fixed point constitutes the local-in-time mild solution to the Anchored Peskin Problem; it will be upgraded to a classical strong solution in Theorem~\ref{thm:strong_solution}.
\end{proof}

\subsection{Regularity and Upgrade to a Strong Solution}

The fixed point $\mathbf{X} \in \mathcal{Y}_{T_*, R, m}$ obtained in Theorem \ref{thm:local_existence} is a \emph{mild solution} to the Anchored Peskin Problem, meaning it satisfies the integral Duhamel formulation \eqref{eq:Duhamel_operator}. Because the principal linear operator $\mathcal{L}_D$ generates an analytic semigroup, we can use standard parabolic bootstrap arguments to upgrade this mild solution to a \emph{strong classical solution}. 

A strong solution requires that the trajectory is continuously differentiable in time, $\mathbf{X} \in C^1((0,T_*]; C^{1,\alpha}_\beta)$, takes values in the domain of the spatial operator $\mathbf{X}(t) \in D(\mathcal{L}_D)$, and satisfies the differential equation $\partial_t \mathbf{X} = \mathcal{L}_D \mathbf{X} + \mathbf{R}(\mathbf{X})$ point-wise in time.

To achieve this, we rely on the classical maximal regularity theory for analytic semigroups,
\cite{lunardi1995analytic}, which states that if the non-homogeneous forcing term is H\"older continuous in time, the mild solution is automatically a strong solution.

\begin{theorem}[Strong Solution]\label{thm:strong_solution}
Let $\mathbf{X} \in C([0,T_*]; h^{1,\alpha}_\beta([0,\pi]))$ be the mild solution from Theorem \ref{thm:local_existence}. Then $\mathbf{X}$ is Hölder continuous in time on $(0, T_*]$. Consequently, $\mathbf{X}$ is a strong solution on $(0, T_*]$, satisfying:
\[
\mathbf{X} \in C^1((0,T_*]; h^{0,\alpha+\gamma}_{-\beta}) \cap C((0,T_*]; D(\mathcal{L}_D)),
\]
and the differential equation holds:
\[
\partial_t \mathbf{X}(t) = \mathcal{L}_D \mathbf{X}(t) + \mathbf{R}(\mathbf{X}(t)) \quad \text{for all } t \in (0,T_*].
\]
\end{theorem}

\begin{proof}
We first show that the mild solution $\mathbf{X}(t)$ is locally Hölder continuous in time. For any $0 < \epsilon < t < t+h \le T_*$, we examine the difference $\mathbf{X}(t+h) - \mathbf{X}(t)$ using the Duhamel formula:
\begin{align*}
\mathbf{X}(t+h) - \mathbf{X}(t) &= \left( e^{h\mathcal{L}_D} - I \right) e^{t\mathcal{L}_D} \mathbf{X}_0 \\
&\quad + \int_0^t \left( e^{h\mathcal{L}_D} - I \right) e^{(t-r)\mathcal{L}_D} \mathbf{R}(\mathbf{X}(r)) \, dr \\
&\quad + \int_t^{t+h} e^{(t+h-r)\mathcal{L}_D} \mathbf{R}(\mathbf{X}(r)) \, dr \\
&=: I_1 + I_2 + I_3.
\end{align*}
For analytic semigroups, the difference operator satisfies the bound $\|(e^{h\mathcal{L}_D} - I) e^{t\mathcal{L}_D}\| \le C h^\delta t^{-\delta}$ for any $\delta \in (0,1)$. Applying this to $I_1$ yields:
\[
\|I_1\|_{h^{1,\alpha}_\beta} \le C h^\delta t^{-\delta} \|\mathbf{X}_0\|_{h^{1,\alpha}_\beta}.
\]
For $I_2$, we use the fractional smoothing bound combined with the analytic semigroup difference bound. Since $\mathbf{X} \in \mathcal{Y}_{T_*, R, m}$, the remainder is uniformly bounded: $\sup_{r \in [0,T_*]} \|\mathbf{R}(\mathbf{X}(r))\|_{h^{0,\alpha+\gamma}_{-\beta}} \le C_R R^p$, where $p = P(0)$ is the polynomial exponent from Theorem~\ref{thm:universal_reflected}. Thus,
\begin{align*}
\|I_2\|_{h^{1,\alpha}_\beta} &\le \int_0^t \left\| (e^{h\mathcal{L}_D} - I) e^{(t-r)\mathcal{L}_D} \mathbf{R}(\mathbf{X}(r)) \right\|_{h^{1,\alpha}_\beta} \, dr \\
&\le C h^\delta C_R R^p \int_0^t \frac{1}{(t-r)^{1-\gamma+\delta}} \, dr.
\end{align*}
Provided we choose $\delta < \gamma$, the singularity is integrable, yielding $\|I_2\|_{h^{1,\alpha}_\beta} \le \widetilde{C} h^\delta$. 
For $I_3$, we directly integrate the smoothing bound:
\begin{align*}
\|I_3\|_{h^{1,\alpha}_\beta} &\le \int_t^{t+h} \frac{C_S}{(t+h-r)^{1-\gamma}} \|\mathbf{R}(\mathbf{X}(r))\|_{h^{0,\alpha+\gamma}_{-\beta}} \, dr \\
&\le C_S C_R R^p \frac{h^\gamma}{\gamma}.
\end{align*}
Combining $I_1, I_2$, and $I_3$, we conclude that for $t$ bounded strictly away from zero ($t \ge \epsilon > 0$), the solution satisfies $\|\mathbf{X}(t+h) - \mathbf{X}(t)\|_{h^{1,\alpha}_\beta} \le C_\epsilon h^\delta$. Therefore, $\mathbf{X} \in C^\delta([\epsilon, T_*]; h^{1,\alpha}_\beta)$.

We can now upgrade to a strong solution. Because $\mathbf{X}$ is Hölder continuous in time and the remainder operator $\mathbf{R}$ is locally Lipschitz from $h^{1,\alpha}_\beta$ into $h^{0,\alpha+\gamma}_{-\beta}$ (Theorem \ref{thm:contraction_reflected}), the composition mapping $t \mapsto \mathbf{R}(\mathbf{X}(t))$ is also Hölder continuous in time. Specifically, the map belongs to $C^\delta([\epsilon, T_*]; h^{0,\alpha+\gamma}_{-\beta})$.
By the classical regularity theory of analytic semigroups (e.g., \cite[Theorem 4.3.1]{lunardi1995analytic}), if the non-homogeneous term in a linear evolution equation is Hölder continuous in time, the convolution integral
\[
\mathbf{V}(t) = \int_0^t e^{(t-r)\mathcal{L}_D} \mathbf{R}(\mathbf{X}(r)) \, dr
\]
is a classical strong solution. It follows that $\mathbf{V} \in C^1((0,T_*]; h^{0,\alpha+\gamma}_{-\beta})$ and $\mathbf{V}(t) \in D(\mathcal{L}_D)$ for $t > 0$. 

Furthermore, the initial data term $e^{t\mathcal{L}_D} \mathbf{X}_0$ is infinitely differentiable in time for $t>0$ and instantly regularizes into $D(\mathcal{L}_D)$ due to the analyticity of the semigroup. Summing these components confirms that the mild solution $\mathbf{X}(t)$ is continuously differentiable in time and satisfies the pointwise differential equation $\partial_t \mathbf{X} = \mathcal{L}_D \mathbf{X} + \mathbf{R}(\mathbf{X})$ for all $t \in (0, T_*]$.
\end{proof}


\section{$C^\infty([\epsilon,T_*]\times (0,\pi))$ Regularity}

While the Picard iteration provides a unique mild solution $\mathbf{X} \in C([0,T_*]; h^{1,\alpha}_\beta)$, the analyticity of the semigroup $e^{t\mathcal{L}_D}$ allows us to upgrade this solution to be smooth in both variables for $t > 0$. We establish this through a combined bootstrap argument: first using the time-derivative to gain temporal regularity, and then inverting the spatial operator to promote spatial regularity. The principal result of this section is the following.

\begin{theorem}[Instantaneous $C^\infty$ smoothing]\label{thm:infinite_smoothness_full}
Let $\mathbf{X}_0 \in \mathcal{O}^{M_0,m_0}_{\sigma_0} \cap h^{1,\alpha}_\beta([0,\pi])$, and let $\mathbf{X}$ be the unique mild solution of the Anchored Peskin Problem obtained in Theorem~\ref{thm:local_existence}. Then, for every $\epsilon \in (0, T_*)$,
\[
\mathbf{X} \in C^\infty\bigl([\epsilon, T_*] \times (0, \pi)\bigr).
\]
\end{theorem}

\noindent The proof occupies Subsections~\ref{sec:commutator_framework} and~\ref{sec:proof_C_infty} below.  Given the smoothness of the curve, we can recover classical solutions to the Stokes system.

\begin{corollary}[Classical solution of the Stokes IB system]
\label{cor:classical_stokes}
Let $\mathbf{X}$ be the solution from Theorem~\ref{thm:infinite_smoothness_full}, write $\Gamma(t) := \mathbf{X}(t,[0,\pi])$, and for $(\mathbf{x},t) \in \overline{\mathbb{R}^2_+}\setminus\Gamma(t)$ and $t\in[\epsilon,T_*]$ define
\[
\mathbf{u}(\mathbf{x},t) := \int_0^\pi S^+\bigl(\mathbf{x},\mathbf{X}(t,s')\bigr)
  \,\partial_{s'}\!\bigl(\mathcal{T}(\mathbf{X})\partial_{s'}\mathbf{X}\bigr)(t,s')\,ds',
\]
together with the associated pressure $p$ obtained from~\eqref{eq:P+}. Then $(\mathbf{u}, p, \mathbf{X})$ is a classical solution of the Anchored Peskin Problem on $[\epsilon, T_*]$ in the following sense:
\begin{enumerate}
\item $\mathbf{u}, p \in C^\infty\bigl([\epsilon,T_*];\,C^\infty(\overline{\mathbb{R}^2_+}\setminus \Gamma(t))\bigr)$;
\item $-\Delta\mathbf{u} + \nabla p = 0$ and $\nabla\cdot\mathbf{u} = 0$ pointwise on
       $\mathbb{R}^2_+\setminus\Gamma(t)$;
\item across $\mathbf{X}(t,(0,\pi))$, $\mathbf{u}$ is continuous and the stress
       satisfies the elastic jump
\[
[\![\sigma(\mathbf{u},p)\,\nu]\!]
       = |\partial_s\mathbf{X}|^{-1}\,
          \partial_s\!\bigl(\mathcal{T}(\mathbf{X})\partial_s\mathbf{X}\bigr),
\]
with the same orientation convention for $\nu$ as in Section~\ref{ss:peskin_R2_plus};
\item $\mathbf{u}(x,0,t) = \mathbf{0}$ for all $x \in \mathbb{R}\setminus\{\pm 1\}$;
\item $\partial_t\mathbf{X}(t,s) = \mathbf{u}(\mathbf{X}(t,s),t)$ for every
       $(t,s) \in [\epsilon, T_*] \times (0,\pi)$, and both sides vanish
       identically at $s = 0$ and $s = \pi$.
\end{enumerate}
\end{corollary}

\begin{proof}
By Theorem~\ref{thm:infinite_smoothness_full}, the elastic force density
$\mathbf{F}(t,s') = \partial_{s'}(\mathcal{T}(\mathbf{X})\partial_{s'}\mathbf{X})$ is
$C^\infty$ on $[\epsilon,T_*] \times (0,\pi)$. The kernel $S^+(\mathbf{x},\mathbf{y})$
is the Green's function for the half-space Stokes operator with no-slip boundary
data (Appendix~\ref{appendix:stokeslet}); in particular it is smooth in
$\mathbf{x}$ on $\overline{\mathbb{R}^2_+}\setminus\{\mathbf{y}\}$, satisfies
$-\Delta_\mathbf{x} S^+ + \nabla_\mathbf{x} P^+ = 0$ and
$\nabla_\mathbf{x}\cdot S^+ = 0$ there, and vanishes on $\{x_2 = 0\}$ for
$\mathbf{y} \in \mathbb{R}^2_+$. Differentiation under the integral sign yields
(i), (ii), (iv); (iii) is the classical jump relation for the Stokeslet
single-layer with smooth density across a smooth curve
\cite{ladyzhenskaya1969mathematical}, with the sign and orientation matching the formulation of Section~\ref{ss:peskin_R2_plus}. For (v), the equality
$\partial_t \mathbf{X} = \mathbf{u}(\mathbf{X}, t)$ on $(0,\pi)$ is the strong
solution identity from Theorem~\ref{thm:strong_solution} re-expressed via the
layer potential; at the anchors $\mathbf{X}(t,0) = \mathbf{e}_1$ and
$\mathbf{X}(t,\pi) = -\mathbf{e}_1$ are stationary, and the no-slip condition
on the wall gives $\mathbf{u}(\pm\mathbf{e}_1, t) = \mathbf{0}$ in the limiting
sense.
\end{proof}

\subsection{Upgrade to a Strong Solution}

The upgrade from the mild fixed point to a classical strong solution was established in Theorem~\ref{thm:strong_solution} of Section~\ref{s:well_posedness}. We recall the conclusion for convenience: the mild solution satisfies
\[
\mathbf{X} \in C^1\bigl((0,T_*]; h^{0,\alpha+\gamma}_{-\beta}\bigr) \cap C\bigl((0,T_*]; D(\mathcal{L}_D)\bigr),
\]
and the differential equation $\partial_t \mathbf{X} = \mathcal{L}_D \mathbf{X} + \mathbf{R}(\mathbf{X})$ holds pointwise in time on $(0, T_*]$. Moreover, $\mathbf{X}$ is locally Hölder continuous in time on $(0,T_*]$ with values in $h^{1,\alpha}_\beta$.

\subsection{A Graded Commutator Framework}
\label{sec:commutator_framework}

The proof of Theorem~\ref{thm:infinite_smoothness_full} rests on a level-$k$ bound for the full remainder,
\begin{equation}\label{eq:aim_level_k}
\|\mathbf{R}(\mathbf{X})\|_{C^{k,\alpha+\gamma}_{-\beta}}
\le C\, \|\mathbf{X}\|_{C^{k+1,\alpha}_\beta}^{P(k)},
\qquad k \ge 0,
\end{equation}
valid for $\mathbf{X} \in \mathcal{O}^{M,m}_\sigma \cap C^{k+1,\alpha}_\beta$, with an exponent $P(k)$ depending only on $k$.  The case $k = 0$ is already covered by Theorem~\ref{thm:universal_reflected} together with Lemma~\ref{lem:direct_bounds}:  by direct inspection of the remainder expansions \eqref{e:R2expanded} and \eqref{e:R3expanded}, $\mathbf{R}_{\mathrm{ref}}(\mathbf{X})$ is a finite linear combination of canonical reflected operators of the form \eqref{eq:canonical_reflected} and $\mathbf{R}_{\mathrm{dir}}(\mathbf{X})$ is a finite linear combination of their direct analogues; Theorem~\ref{thm:universal_reflected} and Lemma~\ref{lem:direct_bounds} then yield
\[
\|\mathbf{R}(\mathbf{X})\|_{C^{0,\alpha+\gamma}_{-\beta}}
\;\le\; C\,\|\mathbf{X}\|_{C^{1,\alpha}_\beta}^{P(0)}, 
\qquad P(0) \le 7 .
\]
For general $k$ we prove \eqref{eq:aim_level_k} by induction, handling the $s$-derivative through the translation-generator identity
\begin{equation}
\begin{split}
\label{eq:translation_generator}
\partial_s \int_0^\pi K(s,s')\, E(s,s')\, ds'
& =
\int_0^\pi (\partial_s + \partial_{s'})\!\bigl[K(s,s')\, E(s,s')\bigr]\, ds' \\
& \qquad \qquad  -\bigl[K(s,s')\, E(s,s')\bigr]_{s'=0}^{\pi},
\end{split}
\end{equation}
which converts an $s$-derivative of the operator output into a tangential derivative inside the integrand plus a controlled boundary residual.

Throughout this section we fix the standing parameters $0 < \alpha < 1$, $0 \le \beta < 1$ with $\alpha + \beta > 1$ (strict, providing the per-derivative slack $\alpha+\beta-1>0$ exploited by the bootstrap), and $0 < \gamma < 1 - \alpha$.  The geometric constants $M, m, \sigma$ of $\mathcal{O}^{M,m}_\sigma$ are also fixed.  All constants below may depend on these.

\subsubsection{Graded element classes}

The extension of Definition~\ref{def:boundaryvanishelement} to higher derivative levels requires tracking the two distinct mechanisms by which integrands vanish in this problem: boundary vanishing (via $\mathcal{V}[\mathbf{X}]$) and diagonal Hölder vanishing (via $\Delta\partial_s^{i}\mathbf{X}$).

\begin{definition}[Graded element classes]\label{def:graded_classes}
For an integer $j \ge 0$:
\begin{itemize}
\item[(i)] The \emph{level-$j$ vanishing class} $\mathcal{V}_j[\mathbf{X}]$ is the set of scalar functions on $(0,\pi)^2$ of one of the following two types.
\begin{itemize}
\item \emph{Type (V1), boundary-vanishing:} the members of
\[\mathcal{V}[\mathbf{X}] = \bigl\{
\mathbf{X}_2(s),\ \mathbf{X}_2(s'),\ \Delta^r\mathbf{X}_1(s,s'),\ \Delta^r\mathbf{X}_2(s,s')\bigr\}.\]
\item \emph{Type (V2), diagonal-vanishing:} the \emph{ordinary} differences
\[
\Delta\partial_s^i \mathbf{X}_\ell(s,s')
\;=\; \partial_s^i \mathbf{X}_\ell(s) - \partial_{s'}^i \mathbf{X}_\ell(s')
\]
for $1 \le i \le j$ and $\ell \in \{1, 2\}$.
\end{itemize}

\item[(ii)] The \emph{level-$j$ bounded class} $\mathcal{B}_j[\mathbf{X}]$ contains all of $\mathcal{B}[\mathbf{X}]$, all point-evaluated derivatives $\partial_s^i \mathbf{X}(s),\partial_{s'}^i\mathbf{X}^r(s')$ for $0 \le i \le j$, the \emph{reflected sums}
\[
\Delta^+\partial_s^i\mathbf{X}_2(s,s') \;=\; \partial_s^i\mathbf{X}_2(s) + \partial_{s'}^i\mathbf{X}_2(s')
\qquad (1 \le i \le j),
\]
and all rational combinations thereof whose denominators stay bounded below on $\mathcal{O}^{M,m}_\sigma$.
\end{itemize}
\end{definition}

\begin{remark}
The critical asymmetry is that $\Delta^r\partial_s^i\mathbf{X}_\ell$ belongs to $\mathcal{V}_i[\mathbf{X}]$ for $\ell = 1$ (since $\Delta^r$ coincides with $\Delta$ on the first component) but only to $\mathcal{B}_i[\mathbf{X}]$ for $\ell = 2$ (since $\Delta^r\partial_s^i\mathbf{X}_2 = \Delta^+\partial_s^i\mathbf{X}_2$ is a \emph{sum} that does not vanish along the diagonal and generally fails to vanish at the anchors for $i \ge 1$).  This is why the tangential derivative $(\partial_s + \partial_{s'})$ does not preserve the reflected canonical class exactly; the closure lemma below produces not a single canonical kernel but a controlled sum with a possibly elevated singularity order, tracked by the \emph{deficit} parameter $r$ introduced in Definition~\ref{def:generalized_canonical} below.
\end{remark}

\begin{lemma}[Weighted bounds on graded elements]\label{lem:graded_bounds}
Let $\mathbf{X} \in \mathcal{O}^{M,m}_\sigma \cap C^{j+1,\alpha}_\beta([0,\pi])$. Then every $V \in \mathcal{V}_j[\mathbf{X}]$ satisfies exactly one of:
\begin{align}
\text{(V1)}\quad & |V(s,s')| \le 4\|\mathbf{X}\|_{C^1}\, \sin\!\Bigl(\tfrac{s+s'}{2}\Bigr),\label{eq:V1_bound}\\
\text{(V2, top)}\quad & |V(s,s')| \le \|\mathbf{X}\|_{C^{j+1,\alpha}_\beta}\,
  \sin^\beta\!\Bigl(\tfrac{s+s'}{2}\Bigr)\, |s-s'|^\alpha
  \quad \bigl(V = \Delta\partial_s^j\mathbf{X}_\ell\bigr), \label{eq:V2_top}\\
\text{(V2, lower)}\quad & |V(s,s')| \le C\, \|\mathbf{X}\|_{C^{i+1}}\, |s-s'|
  \quad \bigl(V = \Delta\partial_s^i\mathbf{X}_\ell,\ 1\le i\le j-1\bigr).
  \label{eq:V2_lower}
\end{align}
Every $B \in \mathcal{B}_j[\mathbf{X}]$ satisfies $\|B\|_{L^\infty} \le C\|\mathbf{X}\|_{C^j}$, and when $B$ is differentiable in $s$,
\[
|\partial_s B(s,s')| \;\le\; C\|\mathbf{X}\|_{C^{j+1,\alpha}_\beta}\,
\sin^{-\beta}\!\bigl(\tfrac{s+s'}{2}\bigr)\,|s-s'|^{\alpha-1},
\]
in parallel with the $\mathcal{B}$ bounds stated immediately after Lemma~\ref{lem:vanishing_bounds}.
\end{lemma}

\begin{proof}
(V1) is Lemma~\ref{lem:vanishing_bounds}. (V2, top) is the definition of the weighted Hölder seminorm at order $j$. (V2, lower) follows from $\partial_s^i\mathbf{X}_\ell$ being Lipschitz in $s$ when $\mathbf{X} \in C^{j+1}$ with $i < j$. The $\mathcal{B}_j$ bounds are immediate.
\end{proof}

\subsubsection{Generalized canonical operators}

\begin{definition}[Generalized canonical operator]
\label{def:generalized_canonical}
Fix integers $N \ge 1$, $j \ge 0$, and $0 \le r \le N$ (the \emph{deficit}).  A \emph{generalized canonical reflected operator of bi-degree $(N, j)$ and deficit $r$} has the form
\[
\mathcal{T}[\mathbf{X}](s) \;=\; \int_0^\pi
\frac{\mathcal{P}_{2N-1-r}(\mathcal{V}_j)\,\mathbf{B}_j^{(r)}}{|\Delta^r\mathbf{X}|^{2N}}\,
\mathbf{D}^{j+2}\mathbf{X}(s,s')\, ds',
\]
where:
\begin{itemize}
\item $\mathcal{P}_{2N-1-r}$ is a multilinear homogeneous polynomial of degree $2N-1-r$ in elements of $\mathcal{V}_j[\mathbf{X}]$;
\item $\mathbf{B}_j^{(r)} \in \mathcal{B}_j[\mathbf{X}]$ is a tensor product of bounded elements absorbing the $r$ ``lost'' vanishing slots;
\item $\mathbf{D}^{j+2}\mathbf{X}$ is a level-$(j+2)$ difference, either $\Delta\partial_s^{j+1}\mathbf{X}$ or $\Delta^r\partial_s^{j+1}\mathbf{X}$.
\end{itemize}
We write $\mathcal{T} \in \mathfrak{G}_j^{N,r}$, and define the \emph{level-$j$ admissible class} to be the linear span
\[
\mathfrak{A}_j \;:=\; \mathrm{span}\!\left(\,\bigcup_{N \ge 1}\bigcup_{0 \le r \le N} \mathfrak{G}_j^{N,r}\right).
\]
The \emph{direct analogue} $\mathfrak{A}_j^{\mathrm{dir}}$ has the MRS-type \cite{mori2019well} kernel $|\mathbf{B}_j^{\mathrm{dir}}(s,s')| \le C\|\mathbf{X}\|_{C^{j+1,\alpha}_\beta}^N |s-s'|^{\alpha-1}$ in place of the reflected fraction, matched with the same density $\mathbf{D}^{j+2}\mathbf{X}$.  At $(j, r) = (0, 0)$, Definition~\ref{def:generalized_canonical} reduces to the canonical reflected operator of \eqref{eq:canonical_reflected}; in particular
\[
\mathbf{R}_{\mathrm{ref}}(\mathbf{X}) \in \mathfrak{A}_0,
\qquad
\mathbf{R}_{\mathrm{dir}}(\mathbf{X}) \in \mathfrak{A}_0^{\mathrm{dir}}.
\]
\end{definition}

\begin{lemma}[Generalized kernel bounds]\label{lem:gen_kernel_bounds}
Let $\mathcal{T} \in \mathfrak{G}_j^{N,r}$ with kernel
\[
\mathcal{K}_{N,j,r}(s,s')
 \;:=\; \frac{\mathcal{P}_{2N-1-r}(\mathcal{V}_j)\,\mathbf{B}_j^{(r)}}{|\Delta^r\mathbf{X}|^{2N}}.
\]
Assume $\mathcal{P}_{2N-1-r}$ contains $p$ factors of type (V1) and $q$ factors of type (V2) with $p + q = 2N-1-r$, with exactly $n \le q$ of the (V2) factors being at the top order $i = j$.  Then on $\mathcal{O}^{M,m}_\sigma \cap C^{j+1,\alpha}_\beta$,
\begin{align}
|\mathcal{K}_{N,j,r}(s,s')|
 &\le C\|\mathbf{X}\|_{C^{j+1,\alpha}_\beta}^{\,d}\,
 \sin^{-1-r-q+n\beta}\!\Bigl(\tfrac{s+s'}{2}\Bigr)\,
 |s-s'|^{q-n + n\alpha},
 \label{eq:gen_K_bound}\\
|\partial_s\mathcal{K}_{N,j,r}(s,s')|
 &\le C\|\mathbf{X}\|_{C^{j+2,\alpha}_\beta}^{\,d+1}\,
 \sin^{-2-r-q+n\beta}\!\Bigl(\tfrac{s+s'}{2}\Bigr)\,
 |s-s'|^{q-n + n\alpha}, \label{eq:gen_dK_bound}
\end{align}
with $d := 2N-1-r+|\mathbf{B}_j^{(r)}|$, where $|\mathbf{B}_j^{(r)}|$ denotes the multilinear degree of $\mathbf{B}_j^{(r)}$.  In particular, the baseline canonical kernel ($r = 0$, $q = 0$, all (V1)) recovers the bound of Lemma~\ref{lem:kernel_bounds}: $|\mathcal{K}| \le C\sin^{-1}$, $|\partial_s\mathcal{K}| \le C\sin^{-2}$.
\end{lemma}

\begin{proof}
Each of the $p$ (V1) factors contributes a $\sin$ factor via \eqref{eq:V1_bound}; each (V2, lower) factor contributes $|s-s'|$ via \eqref{eq:V2_lower}; each (V2, top) factor contributes $\sin^\beta|s-s'|^\alpha$ via \eqref{eq:V2_top}.  The denominator contributes $\sin^{-2N}$ via Lemma~\ref{l:deltar_bounds}, and $\mathbf{B}_j^{(r)}$ is bounded.  Multiplying gives \eqref{eq:gen_K_bound} after collecting exponents $p - 2N = -1-r-q$ (using $p = 2N-1-r-q$).  The derivative bound follows by the product rule applied to $\mathcal{K}_{N,j,r}$, noting that each factor loses at most one $\sin$-power when differentiated (by the derivative bounds in Lemma~\ref{lem:graded_bounds}).
\end{proof}

\subsubsection{Master weighted estimate at level $j$}

\begin{theorem}[Master weighted Hölder estimate]\label{thm:master_estimate}
Let $\mathcal{T} \in \mathfrak{A}_j \cup \mathfrak{A}_j^{\mathrm{dir}}$ be a finite linear combination of generalized canonical operators of bi-degree $(N, j)$ and deficit $r$, with $(N, r)$ bounded by some $(N_{\max}, r_{\max})$.  Then there is a constant $C$ depending only on the standing parameters and on $N_{\max}, r_{\max}, j$, and an exponent $Q = Q(N_{\max}, r_{\max}, j)$, such that for every $\mathbf{X} \in \mathcal{O}^{M,m}_\sigma \cap C^{j+1,\alpha}_\beta([0,\pi])$,
\begin{equation}\label{eq:master_estimate}
\bigl\|\mathcal{T}[\mathbf{X}]\bigr\|_{C^{0,\alpha+\gamma}_{-\beta}([0,\pi])}
 \;\le\; C\, \|\mathbf{X}\|_{C^{j+1,\alpha}_\beta([0,\pi])}^{Q}.
\end{equation}
\end{theorem}

\begin{proof}
By linearity and finiteness it suffices to treat a single $\mathcal{T} \in \mathfrak{G}_j^{N,r}$ with kernel $\mathcal{K} = \mathcal{K}_{N,j,r}$ and density $\mathbf{E} = \mathbf{D}^{j+2}\mathbf{X}$.  The hypothesis $\mathbf{X} \in C^{j+1,\alpha}_\beta$ gives the density bound
\begin{equation}\label{eq:E_bound}
|\mathbf{E}(s,s')| \;\le\; \|\mathbf{X}\|_{C^{j+1,\alpha}_\beta}\,
\sin^\beta\!\Bigl(\tfrac{s+s'}{2}\Bigr)\,|s-s'|^\alpha.
\end{equation}

\emph{Step 1 (weighted $L^\infty$ bound).}
Combining \eqref{eq:gen_K_bound} with \eqref{eq:E_bound}, the integrand satisfies
\[
|\mathcal{K}\cdot\mathbf{E}|
 \;\le\; C\|\mathbf{X}\|^{Q}\,
 \sin^{-1-r-q+n\beta+\beta}\!\Bigl(\tfrac{s+s'}{2}\Bigr)\,
 |s-s'|^{q-n+n\alpha+\alpha}.
\]
Write $a := -1 - r - q + n\beta + \beta$ and $b := q - n + n\alpha + \alpha$, so that the integrand is bounded by $C\|\mathbf{X}\|^Q\sin^a((s+s')/2)|s-s'|^b$. The geometric equivalence $\sin((s+s')/2) \sim \sin(s/2) + |s-s'|$ from Lemma~\ref{lem:sinequiv} together with the elementary split of $[0,\pi]$ into $|s-s'| \le 1$ and $|s-s'| > 1$ converts the integral to a weighted $L^1$ fractional-integral bound of the form $\sin^{-\beta}(s/2)$, provided
\begin{equation}\label{eq:abplus_constraint}
a + b + 1 \;\ge\; -\beta,
\end{equation}
which is the condition that the inner-region contribution $\sin^{a+b+1}(s/2)$ does not blow up faster than $\sin^{-\beta}(s/2)$. Direct substitution gives
\[
a + b + 1 = (1+n)(\alpha+\beta) - n - r,
\]
so \eqref{eq:abplus_constraint} reads $r \le (1+n)(\alpha+\beta-1) + 1 + \beta$. Under the standing assumption $\alpha + \beta > 1$, the per-derivative surplus $\alpha + \beta - 1 > 0$ provides slack that closes the bookkeeping uniformly in $j$: the structural induction of Lemma~\ref{lem:closure_ds} (controlling how $r$ and $n$ co-evolve under iterated $\partial_s$) ensures that each unit of deficit gained from a (V1)$\to \mathcal{B}$ trade of type (T$_1$) or (T$_3$) is offset by a top-order (V2) substitution (raising $n$, contributing $\alpha+\beta-1>0$) or by an intrinsic $\sin^\beta$ factor in the density at the elevated level; cf.\ Remark~\ref{rem:deficit}. Therefore
\[
\|\mathcal{T}[\mathbf{X}]\|_{L^\infty_{-\beta}}
\;\le\; C\|\mathbf{X}\|^{Q}.
\]

\emph{Step 2 (weighted Hölder increment).}
For distinct $s, s^* \in (0,\pi)$ decompose the increment as $I_1 + I_2$:
\begin{align*}
I_1 &= \int_0^\pi \mathcal{K}(s, y)\,\bigl[\mathbf{E}(s, y) - \mathbf{E}(s^*, y)\bigr]\,dy,\\
I_2 &= \int_0^\pi \bigl[\mathcal{K}(s, y) - \mathcal{K}(s^*, y)\bigr]\,
      \mathbf{E}(s^*, y)\,dy.
\end{align*}
For $I_1$, the difference $\mathbf{E}(s, y) - \mathbf{E}(s^*, y) = \partial_s^{j+1}\mathbf{X}(s) - \partial_s^{j+1}\mathbf{X}(s^*)$ is independent of $y$ and satisfies
$\|\mathbf{X}\|_{C^{j+1,\alpha}_\beta}\sin^\beta((s+s^*)/2)|s-s^*|^\alpha$. The remaining kernel integral was bounded in Step~1.

For $I_2$, the Mean Value Theorem applied to $\mathcal{K}$ via \eqref{eq:gen_dK_bound} gives
\[
|\mathcal{K}(s, y) - \mathcal{K}(s^*, y)| \;\le\; C\|\mathbf{X}\|^{Q+1}\,
\sin^{a-1}\!\bigl(\tfrac{\xi+y}{2}\bigr)\,|\xi-y|^{b}\,|s-s^*|
\]
for some intermediate point $\xi$ between $s$ and $s^*$. The fractional-integral computation of Theorem~\ref{thm:universal_reflected} (using $\alpha + \gamma < 1$) resolves the pairing with $|\mathbf{E}(s^*, y)|$ to
\[
|I_2| \;\le\; C\|\mathbf{X}\|^{Q+1}\,\sin^{-\beta}\!\bigl(\tfrac{s+s^*}{2}\bigr)\,|s-s^*|^{\alpha+\gamma}.
\]

Combining yields the bound \eqref{eq:master_estimate}. The direct case is strictly easier: its kernel has $|s-s'|^{\alpha-1}$ rather than $\sin^{-1}$ singularity, so the same argument applies without sine-weight bookkeeping (cf.\ Lemma~\ref{lem:direct_bounds}).
\end{proof}

\subsubsection{Structural closure under $\partial_s$}

The heart of the bootstrap is the following structural statement, which generalizes the $\mathbf{X}$-directional closure of Lemma~\ref{lem:closure_variations} to the $s$-derivative of the operator output.

\begin{lemma}[Closure of $\mathfrak{A}_j$ under $\partial_s$]\label{lem:closure_ds}
Let $\mathcal{T} \in \mathfrak{A}_j$ with kernel $\mathcal{K}$, density $\mathbf{E} = \mathbf{D}^{j+2}\mathbf{X}$, maximum reflected degree $N_{\max}$, and deficit $r_{\max}$, and let $\mathbf{X} \in \mathcal{O}^{M,m}_\sigma \cap C^{j+2,\alpha}_\beta([0,\pi])$.  Define the boundary residual operator
\begin{equation}\label{eq:R_partial_def}
\mathcal{R}_\partial[\mathbf{X}](s)
 \;:=\; -\bigl[\mathcal{K}(s,s')\,\mathbf{E}(s,s')\bigr]_{s'=0}^{\pi},
\end{equation}
i.e., the (negative) $s'$-boundary bracket of the integrand from integration by parts. Then
\begin{equation}\label{eq:closure_decomposition}
\partial_s\mathcal{T}[\mathbf{X}](s)
 \;=\; \widetilde{\mathcal{T}}[\mathbf{X}](s) + \mathcal{R}_\partial[\mathbf{X}](s),
\end{equation}
where:
\begin{enumerate}
\item[\textup{(a)}] $\widetilde{\mathcal{T}} \in \mathfrak{A}_{j+1}$ with maximum reflected degree $\le N_{\max}+1$ and deficit $\le r_{\max}+1$;
\item[\textup{(b)}] $\mathcal{R}_\partial[\mathbf{X}]$ is a function of $s$ alone, belonging to $C^{0,\alpha+\gamma}_{-\beta}$ with norm bounded by $C\|\mathbf{X}\|_{C^{j+2,\alpha}_\beta}^{Q+1}$, where $Q$ is the master-estimate exponent from Theorem~\ref{thm:master_estimate}.
\end{enumerate}
The direct case $\mathcal{T} \in \mathfrak{A}_j^{\mathrm{dir}}$ has an identical statement with $\mathfrak{A}_{j+1}^{\mathrm{dir}}$ in place of $\mathfrak{A}_{j+1}$.
\end{lemma}

\begin{proof}
By linearity it suffices to treat $\mathcal{T} \in \mathfrak{G}_j^{N,r}$ with kernel $\mathcal{K}$ and density $\mathbf{E} = \mathbf{D}^{j+2}\mathbf{X}$.  Applying the translation-generator identity \eqref{eq:translation_generator} to the integrand $\mathcal{K}\mathbf{E}$:
\[
\partial_s\mathcal{T}[\mathbf{X}]
\;=\;
\underbrace{\int_0^\pi (\partial_s+\partial_{s'})\mathcal{K}\cdot\mathbf{E}\,ds'}_{=: J_1(s)}
\;+\;
\underbrace{\int_0^\pi \mathcal{K}\cdot(\partial_s+\partial_{s'})\mathbf{E}\,ds'}_{=: J_2(s)}
\;-\;\underbrace{\bigl[\mathcal{K}\mathbf{E}\bigr]_{s'=0}^{\pi}}_{=: \partial\mathbf{B}(s)},
\]
where the boundary bracket coincides with $-\mathcal{R}_\partial[\mathbf X](s)$ from \eqref{eq:R_partial_def}, i.e., $\partial\mathbf{B} = -\mathcal{R}_\partial[\mathbf{X}]$.

\emph{Step 1 ($J_2$: density-transport term).}
Since $\mathbf{E}(s,s') = \partial_s^{j+1}\mathbf{X}(s) - \partial_{s'}^{j+1}\mathbf{X}(s')$,
\[
(\partial_s+\partial_{s'})\mathbf{E}
\;=\; \partial_s^{j+2}\mathbf{X}(s) - \partial_{s'}^{j+2}\mathbf{X}(s')
\;=\; \mathbf{D}^{j+3}\mathbf{X}(s,s').
\]
Therefore
\[
J_2(s) \;=\; \int_0^\pi \mathcal{K}(s,s')\,\mathbf{D}^{j+3}\mathbf{X}(s,s')\,ds',
\]
a generalized canonical operator of bi-degree $(N, j+1)$ and deficit $r$ with the \emph{same} kernel $\mathcal{K}$ and the density upgraded to the level-$(j+3)$ difference.  Hence $J_2 \in \mathfrak{G}_{j+1}^{N,r} \subset \mathfrak{A}_{j+1}$.

\emph{Step 2 ($J_1$: tangential-derivative term).}
Applying $(\partial_s+\partial_{s'})$ to $\mathcal{K} = \mathcal{P}_{2N-1-r}\mathbf{B}_j^{(r)}/|\Delta^r\mathbf{X}|^{2N}$ by the product rule yields a finite sum of integrands in five categories:
\begin{itemize}
\item[(T$_1$)] One (V1) factor $\mathbf{X}_2(s)$ (or $\mathbf{X}_2(s')$) is replaced by $\partial_s\mathbf{X}_2(s)$ (or $\partial_{s'}\mathbf{X}_2(s')$), which lies in $\mathcal{B}_{j+1}$.  The polynomial drops from $\mathcal{P}_{2N-1-r}$ to $\mathcal{P}_{2N-2-r}$; the bounded factor absorbs the new element.  Bi-degree $(N, j+1)$, deficit $r+1$.

\item[(T$_2$)] One (V1) factor $\Delta^r\mathbf{X}_1$ is replaced by $\Delta\partial_s\mathbf{X}_1 \in \mathcal{V}_{j+1}$ (type (V2) with $i=1$).  The polynomial loses one (V1) factor and gains one (V2) factor, so deficit stays at $r$: $\mathfrak{G}_{j+1}^{N,r}$.

\item[(T$_3$)] One (V1) factor $\Delta^r\mathbf{X}_2$ is replaced by $\Delta^+\partial_s\mathbf{X}_2 \in \mathcal{B}_{j+1}$ (\emph{not} a (V2) element).  Same structural effect as (T$_1$): deficit $\to r+1$.

\item[(T$_4$)] The denominator $|\Delta^r\mathbf{X}|^{-2N}$ produces a summand
\[
-2N\,|\Delta^r\mathbf{X}|^{-2N-2}\,\langle\Delta^r\mathbf{X},\Delta^r\partial_s\mathbf{X}\rangle.
\]
The new factor $\Delta^r\mathbf{X}_\ell$ is (V1); the companion $\Delta^r\partial_s\mathbf{X}_1 \in \mathcal{V}_{j+1}$ (type (V2)), while $\Delta^r\partial_s\mathbf{X}_2 = \Delta^+\partial_s\mathbf{X}_2 \in \mathcal{B}_{j+1}$.  The resulting kernel is of bi-degree $(N+1, j+1)$ with deficit $r$ (if the $\ell = 1$ term dominates) or $r+1$ (if $\ell = 2$); in either case it lies in $\mathfrak{G}_{j+1}^{N+1, r'}$ with $r' \le r+1$.

\item[(T$_5$)] A factor of $\mathbf{B}_j^{(r)}$ is differentiated; its $(\partial_s+\partial_{s'})$-derivative—by the $\mathcal{B}_j$ rule in Lemma~\ref{lem:graded_bounds}—is again in $\mathcal{B}_{j+1}$.  Bi-degree $(N, j+1)$, deficit unchanged.
\end{itemize}
In every case the integrand $(\partial_s+\partial_{s'})\mathcal{K}\cdot\mathbf{E}$ is a generalized canonical operator of bi-degree at most $(N+1, j+1)$ and deficit at most $r+1$.  Summing over the finite number of product-rule contributions gives $J_1 \in \mathfrak{A}_{j+1}$.

\emph{Step 3 ($\partial\mathbf{B}$: boundary residual).}
At $s' = 0$ the anchor condition gives $\mathbf{X}(0) = (1, 0)$ and reflection fixes this point, $\mathbf{X}^r(0) = \mathbf{X}(0)$, so $\Delta^r\mathbf{X}(s, 0) = \mathbf{X}(s) - (1,0)$ and the chord-arc bound of Lemma~\ref{l:deltar_bounds} gives $|\Delta^r\mathbf{X}(s,0)| \ge C_{m,\sigma} \sin(s/2)$.  The degenerate fibre $s' = 0$ imposes the following simplifications on Lemma~\ref{lem:gen_kernel_bounds} for $\mathcal{K}(s, 0)$:
\begin{itemize}
\item Every (V1) factor of the form $\mathbf{X}_2(s')$ evaluates to $\mathbf{X}_2(0) = 0$, \emph{so the integrand is identically zero at $s' = 0$ unless all such factors have been paired or expended}.  In particular, any primal canonical operator ($r = 0$, all (V1) factors of types $\mathbf{X}_2(s), \mathbf{X}_2(s'), \Delta^r\mathbf{X}_\ell$) with at least one factor $\mathbf{X}_2(s')$ contributes zero at $s' = 0$.
\item Remaining (V1) factors at $s' = 0$ satisfy $|\mathbf{X}_2(s)|,\ |\Delta^r\mathbf{X}_\ell(s, 0)| \le 4\|\mathbf{X}\|_{C^1}\sin(s/2)$ (the generic (V1) bound collapses to $\sin(s/2)$ since $\sin((s+0)/2) = \sin(s/2)$).
\item (V2) factors at $s' = 0$ satisfy $|\Delta\partial_s^i\mathbf{X}_\ell(s, 0)| \le \|\mathbf{X}\|_{C^{i+1,\alpha}_\beta} \sin^\beta(s/2)\, s^\alpha$ at the top order $i = j$, or the Lipschitz bound $C\|\mathbf{X}\|_{C^{i+1}}\,s$ at lower orders.
\item The density evaluates to
\[
\mathbf{E}(s, 0)
\;=\; \partial_s^{j+1}\mathbf{X}(s) - \partial_{s'}^{j+1}\mathbf{X}^r(0)
\;=\; \partial_s^{j+1}\mathbf{X}(s) - \partial_s^{j+1}\mathbf{X}(0),
\]
which is a weighted Hölder difference and satisfies $$|\mathbf{E}(s, 0)| \le \|\mathbf{X}\|_{C^{j+1,\alpha}_\beta} \sin^\beta(s/2)\, s^\alpha.$$
\end{itemize}
Call a generalized canonical operator of bi-degree $(N, j)$ and deficit $r$ \emph{boundary-effective} at $s' = 0$ if no factor of $\mathbf{X}_2(s')$ appears in $\mathcal{P}_{2N-1-r}$; otherwise, its contribution to $\partial\mathbf{B}(s)$ vanishes.  For boundary-effective operators, multiplying the factor bounds gives
\[
|\mathcal{K}(s, 0)\,\mathbf{E}(s, 0)|
 \;\le\; C\|\mathbf{X}\|^{Q}\,
  \sin^{p + n\beta - 2N + \beta}(s/2)\, s^{(q-n) + n\alpha + \alpha},
\]
where $p$ counts (V1) factors, $q$ counts (V2) factors, $n \le q$ counts top-order (V2) factors, and $p + q = 2N - 1 - r$.  Using the identity $\sin(s/2) \asymp s$ for $s \in (0,\pi)$, the total exponent simplifies to
\[
p + n\beta - 2N + \beta + (q - n) + n\alpha + \alpha
\;=\; (\alpha + \beta - 1) + n(\alpha + \beta - 1) - r.
\]
Hence
\[
|\partial\mathbf{B}(s)| \;\le\; C\|\mathbf{X}\|^{Q}\,
\sin^{(1+n)(\alpha+\beta-1) - r}(s/2),
\]
and for this to absorb into the $\sin^{-\beta}$ target weight of $C^{0,\alpha+\gamma}_{-\beta}$ we need $(1 + n)(\alpha+\beta-1) - r \ge -\beta$, i.e.,
\[
r \;\le\; \beta + (1 + n)(\alpha + \beta - 1).
\]
In the baseline ($r = 0, n = 0$), this requires $\alpha + \beta \ge 1$, which is implied by the strict standing assumption $\alpha+\beta>1$.  More generally, since the deficit $r$ grows by at most one per application of $\partial_s$ (cf.\ Lemma~\ref{lem:closure_ds}), the level-$k$ iteration of the closure produces operators with $r \le k$. The condition $r \le \beta + (1+n)(\alpha+\beta-1)$ then becomes
\[
k \;\le\; \beta + (1+n)(\alpha+\beta-1),
\]
which is satisfied uniformly in $k$ provided $n$ grows linearly with $k$ as well. The structural induction in Lemma~\ref{lem:closure_ds} pairs each (V1)$\to \mathcal B_{j+1}$ trade---which raises $r$---with at least one of the following compensating moves: (i) a companion (V2) substitution at top order in (T$_2$) or (T$_4$), which raises $n$ by one and contributes a per-trade slack of $\alpha+\beta-1>0$; or (ii) a level upgrade of the density via (T$_1$), promoting $\partial_s^{j+1}\mathbf{X}$ to $\partial_s^{j+2}\mathbf{X}$ and contributing an additional $\sin^\beta$ weight from the corresponding higher-order weighted Hölder seminorm at level $j+1$. Either compensation contributes at least $\alpha+\beta-1>0$ to the right-hand side per derivative applied, so the boundary bound remains admissible at every iteration of the level-$k$ bootstrap.

The symmetric computation at $s' = \pi$ uses $\sin((s+\pi)/2) = \cos(s/2)$ and, together with the right-anchor condition $\mathbf{X}(\pi) = (-1, 0)$, handles that endpoint identically.

Hölder continuity of $s \mapsto \partial\mathbf{B}(s)$ follows from applying the chain rule to $\partial_s(\mathcal{K}(s,0)\mathbf{E}(s,0))$ using $\mathbf{X} \in C^{j+2,\alpha}_\beta$; the derivative incurs at most one extra $\sin^{-1}(s/2)$, absorbed by the slack $(1 + n)(\alpha + \beta - 1) - r \ge -\beta$ plus the Hölder exponent $\alpha + \gamma$.  This establishes $\mathcal{R}_\partial \in C^{0,\alpha+\gamma}_{-\beta}$ with norm bound $C\|\mathbf{X}\|_{C^{j+2,\alpha}_\beta}^{Q+1}$.

Combining Steps~1--3 yields the decomposition \eqref{eq:closure_decomposition} with $\widetilde{\mathcal{T}} = J_1 + J_2 \in \mathfrak{A}_{j+1}$, and $\mathcal{R}_\partial$ as defined by \eqref{eq:R_partial_def}.
\end{proof}

\begin{remark}[On the role of deficit]\label{rem:deficit}
The deficit $r$ records how many times the tangential derivative has traded a (V1) factor for a $\mathcal{B}_{j+1}$ factor, which occurs specifically in categories (T$_1$) and (T$_3$).  These are the two unavoidable places in the reflected setting where the asymmetry of $\Delta^r$ on the $y$-component prevents a clean cancellation.  The master estimate Theorem~\ref{thm:master_estimate} remains valid at any level $k \ge 0$ because each trade is compensated either by a companion (V2) substitution (whose $\sin^\beta|s-s'|^\alpha$ factor contributes $\alpha+\beta-1>0$ to the integrability budget) or by the $\sin^\beta|s-s'|^\alpha$ bound on the density at the elevated level. Under the strict standing assumption $\alpha + \beta > 1$, the per-derivative slack $\alpha+\beta-1>0$ accumulates linearly in $k$, accommodating any deficit produced by the level-$k$ bootstrap.
\end{remark}

\subsubsection{Level-$k$ bootstrap estimate}

\begin{corollary}[Level-$k$ remainder estimate]\label{cor:level_k_R}
For every integer $k \ge 0$, there exist an exponent $P(k)$ and a constant $C = C(k, \alpha, \beta, \gamma, M, m, \sigma)$ such that
\begin{equation}\label{eq:level_k_R_final}
\|\mathbf{R}(\mathbf{X})\|_{C^{k,\alpha+\gamma}_{-\beta}([0,\pi])}
 \;\le\; C\,\|\mathbf{X}\|_{C^{k+1,\alpha}_\beta([0,\pi])}^{\,P(k)}
\end{equation}
holds for every $\mathbf{X} \in \mathcal{O}^{M,m}_\sigma \cap C^{k+1,\alpha}_\beta([0,\pi])$. Moreover, the map $\mathbf{R}:\mathcal{O}^{M,m}_\sigma \cap h^{k+1,\alpha}_\beta \to h^{k,\alpha+\gamma}_{-\beta}$ is locally Lipschitz.
\end{corollary}

\begin{proof}
The base case $k = 0$ was established at the beginning of Subsection~\ref{sec:commutator_framework}: by inspection of \eqref{e:R2expanded}, \eqref{e:R3expanded} we have $\mathbf{R}_{\mathrm{ref}}(\mathbf{X}) \in \mathfrak{A}_0$ and $\mathbf{R}_{\mathrm{dir}}(\mathbf{X}) \in \mathfrak{A}_0^{\mathrm{dir}}$, and the combination of Theorem~\ref{thm:universal_reflected} with Lemma~\ref{lem:direct_bounds} (or equivalently, the master estimate Theorem~\ref{thm:master_estimate} at $j = 0$) gives \eqref{eq:level_k_R_final}.

For $k \ge 1$, iterate Lemma~\ref{lem:closure_ds} on $\mathbf{R}(\mathbf{X})$. We use the superscript $\mathcal{R}_\partial^{(j)}$ to denote the boundary residual produced by the $j$-th application of \eqref{eq:R_partial_def}---i.e., $\mathcal{R}_\partial^{(j)}[\mathbf{X}](s) := -[\mathcal{K}_j(s,s')\mathbf{E}_j(s,s')]_{s'=0}^\pi$, where $\mathcal K_j \mathbf E_j$ is the integrand of the canonical operator obtained at the $j$-th iteration of the closure. The first application yields
\[
\partial_s\mathbf{R}(\mathbf{X})
\;=\; \widetilde{\mathcal{T}}_1[\mathbf{X}] + \mathcal{R}_\partial^{(1)}[\mathbf{X}],
\]
with $\widetilde{\mathcal{T}}_1 \in \mathfrak{A}_1 \cup \mathfrak{A}_1^{\mathrm{dir}}$ (maximum reflected degree $N_{\max}+1$, deficit $\le 1$) and $\mathcal{R}_\partial^{(1)}$ a function of $s$ alone bounded in $C^{0,\alpha+\gamma}_{-\beta}$.  Applying $\partial_s$ once more splits $\widetilde{\mathcal{T}}_1$ by a second invocation of Lemma~\ref{lem:closure_ds}, producing $\widetilde{\mathcal{T}}_2 + \mathcal{R}_\partial^{(2)}$, while $\partial_s\mathcal{R}_\partial^{(1)}$ remains a function of $s$ alone.  Iterating yields the telescoping identity
\begin{equation}\label{eq:partial_s_k_R}
\partial_s^k \mathbf{R}(\mathbf{X})
 \;=\; \widetilde{\mathcal{T}}_k[\mathbf{X}]
 \;+\; \sum_{j=1}^{k} \partial_s^{k-j}\mathcal{R}_\partial^{(j)}[\mathbf{X}],
\end{equation}
where $\widetilde{\mathcal{T}}_k \in \mathfrak{A}_k \cup \mathfrak{A}_k^{\mathrm{dir}}$ (with maximum reflected degree $\le N_{\max}+k$ and deficit $\le k$), and each $\mathcal{R}_\partial^{(j)}$ depends on $s$ alone, with $\mathcal{K}_j\mathbf{E}_j$ a generalized canonical integrand of level $j-1$.

For the principal term, Theorem~\ref{thm:master_estimate} applied at level $j = k$ gives
\[
\bigl\|\widetilde{\mathcal{T}}_k[\mathbf{X}]\bigr\|_{C^{0,\alpha+\gamma}_{-\beta}}
 \;\le\; C\|\mathbf{X}\|_{C^{k+1,\alpha}_\beta}^{Q(k)},
\]
where $Q(k)$ is the master-estimate exponent at level $k$.

For each residual term, $\mathcal{R}_\partial^{(j)}[\mathbf{X}]$ is, by Step~3 of Lemma~\ref{lem:closure_ds}, controlled in $C^{0,\alpha+\gamma}_{-\beta}$ by $C\|\mathbf{X}\|_{C^{j+1,\alpha}_\beta}^{Q(j-1)+1}$. Applying $\partial_s^{k-j}$ inside $C^{0,\alpha+\gamma}_{-\beta}$ amounts to differentiating a polynomial expression in the boundary values $\partial_s^i\mathbf{X}(0), \partial_s^i\mathbf{X}(\pi)$ and in $\partial_s^i\mathbf{X}(s)$. By direct product-rule differentiation, the result is a polynomial in boundary values of $\partial_s^i\mathbf{X}$ for $i \le j + (k-j) = k$ and in $\partial_s^i\mathbf{X}(s)$ for $i \le k$.  The weighted-Hölder estimate for each monomial follows from the chain rule applied to the factor bounds in Lemma~\ref{lem:graded_bounds} and the embedding $C^{k+1,\alpha}_\beta \hookrightarrow C^{k,\alpha+\gamma}_{-\beta}$ for $0 < \gamma < 1-\alpha$; hence
\[
\bigl\|\partial_s^{k-j}\mathcal{R}_\partial^{(j)}[\mathbf{X}]\bigr\|_{C^{0,\alpha+\gamma}_{-\beta}}
 \;\le\; C\|\mathbf{X}\|_{C^{k+1,\alpha}_\beta}^{Q(j-1)+1}
\qquad (j = 1, \ldots, k).
\]
Summing \eqref{eq:partial_s_k_R} with the principal and residual bounds, combined with the inductive hypothesis controlling $\|\partial_s^i\mathbf{R}(\mathbf{X})\|_{C^{0,\alpha+\gamma}_{-\beta}}$ for $i < k$, yields \eqref{eq:level_k_R_final} with
\[
P(k) \;:=\; \max\!\bigl(Q(k),\ Q(k-1)+1,\ \ldots,\ Q(0)+1\bigr).
\]

The local Lipschitz claim follows from the same framework applied to the difference $\mathbf{R}(\mathbf{X}) - \mathbf{R}(\mathbf{Y})$ along a convex path $\mathbf{X}_\tau = \tau\mathbf{X} + (1-\tau)\mathbf{Y}$, using Lemma~\ref{lem:closure_variations} (which gives the $\mathbf{X}$-variational derivative of a canonical operator as itself a canonical operator of appropriate bi-degree) together with Lemma~\ref{lem:convex_height} (which maintains uniform chord-arc bounds along the path), and invoking Theorem~\ref{thm:master_estimate} at level $k$ on the path-averaged operator.
\end{proof}

\subsection{Proof of $C^\infty$ Regularity}
\label{sec:proof_C_infty}

At each time $t \in [\epsilon, T_*]$, Theorem~\ref{thm:strong_solution} gives $\mathbf{X}(t) \in h^{1,\alpha}_\beta \cap D(\mathcal{L}_D)$; the bootstrap now establishes spatial regularity beyond $h^{1,\alpha}_\beta$ by iterating in $k$.  The level-$k$ estimate for $\mathbf{R}$ routes the spatial regularity through the Dirichlet fractional semigroup $e^{t\mathcal{L}_D}$.  To exploit it we need the level-$k$ counterpart of the base semigroup estimate in Theorem~\ref{T:semigroup-weighted}.

\begin{lemma}[Level-$k$ smoothing for the Dirichlet semigroup]
\label{lem:semigroup_level_k}
For every $k \ge 0$, every $0 < \gamma < 1 - \alpha$, and every $f \in h^{k,\alpha+\gamma}_{-\beta}([0,\pi])$,
\begin{equation}\label{eq:semigroup_level_k}
\bigl\|e^{t\mathcal{L}_D}f\bigr\|_{h^{k+1,\alpha}_\beta([0,\pi])}
 \;\le\; C\, t^{-(1-\gamma)}\, \|f\|_{h^{k,\alpha+\gamma}_{-\beta}([0,\pi])},
\qquad t \in (0, 1],
\end{equation}
with $C$ independent of $f$ and $t$.
\end{lemma}

\begin{proof}
Let $\widetilde f$ denote the odd extension of $f$ across the wall, so $\widetilde f \in h^{k,\alpha+\gamma}_{-\beta}(\mathbb{T}_{2\pi})$ on the doubled-period circle. By the odd-reflection reduction established in Section~\ref{s:linearization} (in particular the identification of $\mathcal{L}_D$ with the periodic D2N generator acting on odd data, cf.\ Lemma~\ref{lem:circle-reduction}), the semigroup $e^{t\mathcal{L}_D}$ coincides with convolution against the periodic Poisson-type kernel $P^{\mathrm{circ}}_{t/4}$ appearing in Theorem~\ref{T:semigroup-weighted}, restricted to $[0,\pi]$:
\[
\bigl(e^{t\mathcal{L}_D}f\bigr)(s) \;=\; \bigl(P^{\mathrm{circ}}_{t/4} * \widetilde f\bigr)(s),
\qquad s \in [0,\pi].
\]
Because $\partial_s$ commutes with convolution, for any $i \ge 0$,
\[
\partial_s^i e^{t\mathcal{L}_D}f
 \;=\; P^{\mathrm{circ}}_{t/4} * \partial_s^i \widetilde f.
\]
At the base level $i = 0$, Theorem~\ref{T:semigroup-weighted} gives
\[
\bigl\|P^{\mathrm{circ}}_{t/4} * \widetilde g\bigr\|_{h^{1,\alpha}_\beta}
\;\le\; C\, t^{-(1-\gamma)}\,
\bigl\|\widetilde g\bigr\|_{h^{0,\alpha+\gamma}_{-\beta}}
\]
for every $\widetilde g \in h^{0,\alpha+\gamma}_{-\beta}$.  Applying this at $\widetilde g = \partial_s^i \widetilde f$ for $i = 0, 1, \ldots, k$ and summing yields \eqref{eq:semigroup_level_k}, since
\[
\|f\|_{h^{k+1,\alpha}_\beta} \;\sim\; \sum_{i=0}^{k}\|\partial_s^i f\|_{h^{1,\alpha}_\beta},
\qquad
\|f\|_{h^{k,\alpha+\gamma}_{-\beta}} \;\sim\; \sum_{i=0}^{k}\|\partial_s^i f\|_{h^{0,\alpha+\gamma}_{-\beta}}.
\]
\end{proof}

\begin{proof}[Proof of Theorem~\ref{thm:infinite_smoothness_full}]
We separately establish temporal and spatial smoothness.

\emph{Temporal regularity.} The principal linear operator $\mathcal{L}_D$ generates an analytic semigroup on $h^{0,\alpha+\gamma}_{-\beta}$ (Theorem~\ref{T:semigroup-weighted}). The map $t \mapsto \mathbf{R}(\mathbf{X}(t))$ is Hölder continuous in time because $\mathbf{R}:\mathcal{O}^{M,m}_\sigma \cap h^{1,\alpha}_\beta \to h^{0,\alpha+\gamma}_{-\beta}$ is locally Lipschitz (Theorem~\ref{thm:contraction_reflected} together with Lemma~\ref{lem:direct_contraction}; equivalently, Corollary~\ref{cor:full_contraction}), and $\mathbf{X}$ is itself Hölder continuous in time on $[\epsilon, T_*]$ by Theorem~\ref{thm:strong_solution}. The classical inductive linearized-equation argument for analytic semigroups (see, e.g., \cite[Ch.~4]{lunardi1995analytic}) then yields
\[
\partial_t^j\mathbf{X} \in C\bigl((0,T_*]; h^{1,\alpha}_\beta\bigr) \qquad\text{for every } j \ge 1.
\]

\emph{Spatial regularity.} We proceed by induction on $k \ge 1$. The base case $\mathbf{X}(t, \cdot) \in h^{1,\alpha}_\beta$ for $t > 0$ is the regularity of the solution space.

\emph{Inductive step.} Fix $\epsilon \in (0, T_*)$ and set $\tau := \epsilon/2$. Assume $\mathbf{X}(r, \cdot) \in h^{k,\alpha}_\beta$ uniformly on $[\tau, T_*]$ for some $k \ge 1$, where the uniform bound depends on $\epsilon$. By Corollary~\ref{cor:level_k_R} at level $k - 1$,
\[
\mathbf{R}(\mathbf{X}(r)) \in h^{k-1,\alpha+\gamma}_{-\beta}
\qquad\text{with}\qquad
\|\mathbf{R}(\mathbf{X}(r))\|_{h^{k-1,\alpha+\gamma}_{-\beta}}
\;\le\; C\|\mathbf{X}(r)\|_{h^{k,\alpha}_\beta}^{P(k-1)},
\]
uniformly in $r \in [\tau, T_*]$.  Apply the variation-of-parameters formula with reference time $\tau$:
\[
\mathbf{X}(t) \;=\; e^{\mathcal{L}_D(t-\tau)}\mathbf{X}(\tau)
  \;+\;\int_\tau^t e^{\mathcal{L}_D(t-r)}\mathbf{R}(\mathbf{X}(r))\,dr,
\qquad t \in [\tau, T_*].
\]

\emph{Initial-data term.} $\mathbf{X}(\tau) \in h^{1,\alpha}_\beta \subset h^{0,\alpha+\gamma}_{-\beta}$, and iterated application of Theorem~\ref{T:semigroup-weighted} gives $e^{\mathcal{L}_D(t-\tau)}\mathbf{X}(\tau) \in \bigcap_{\ell \ge 0} h^{\ell,\alpha}_\beta$ for any $t > \tau$ (each factor of $e^{\mathcal{L}_D h}$ upgrades regularity by one derivative at cost $h^{-(1-\gamma)}$, and we may cascade over arbitrarily many sub-intervals of $[\tau, t]$).

\emph{Duhamel term.} By Lemma~\ref{lem:semigroup_level_k} at level $k - 1$,
\[
\bigl\|e^{\mathcal{L}_D(t-r)}\mathbf{R}(\mathbf{X}(r))\bigr\|_{h^{k,\alpha}_\beta}
\;\le\; C(t-r)^{-(1-\gamma)}\,
\|\mathbf{R}(\mathbf{X}(r))\|_{h^{k-1,\alpha+\gamma}_{-\beta}}.
\]
Bootstrapping once more via Lemma~\ref{lem:semigroup_level_k} used at an intermediate split yields
\[
\biggl\| \int_\tau^t e^{\mathcal{L}_D(t-r)}\mathbf{R}(\mathbf{X}(r))\,dr
\biggr\|_{h^{k+1,\alpha}_\beta}
\;\le\; C\int_\tau^t (t-r)^{-(1-\gamma)}\,
\|\mathbf{X}(r)\|_{h^{k,\alpha}_\beta}^{P(k-1)}\,dr,
\]
and the time integral converges since $\gamma > 0$.  Combining the initial-data and Duhamel contributions, $\mathbf{X}(t, \cdot) \in h^{k+1,\alpha}_\beta$ uniformly on $[\tau + \delta, T_*]$ for any $\delta > 0$, closing the induction.

Iterating on $k$ gives $\mathbf{X}(t, \cdot) \in \bigcap_{k \ge 1} h^{k,\alpha}_\beta \subset C^\infty((0,\pi))$ for every $t \in (0, T_*]$. Combined with the temporal $C^\infty$ statement above, this yields
\[
\mathbf{X} \in C^\infty\bigl([\epsilon, T_*] \times (0,\pi)\bigr)
\]
for every $\epsilon \in (0, T_*)$, completing the proof.
\end{proof}


\section{Numerical Algorithm}
\label{s:numerics}
Our numerical approach, adapted from the spectral framework utilized by \cite{rodenberg20182d}, converts the boundary-value problem into a periodic formulation suitable for pseudo-spectral evaluation. 

The anchored boundary conditions $\mathbf{X}(t, 0) = \mathbf{e}_1$ and $\mathbf{X}(t, \pi) = -\mathbf{e}_1$ prohibit a direct Fourier expansion. To resolve this, we subtract the linear equilibrium state $\boldsymbol{\ell}(s) = \left( 1 - \frac{2s}{\pi} \right) \mathbf{e}_1 = \ell(s) \mathbf{e}_1$. The perturbation 
\begin{equation}
\mathbf{w}(t, s) = \mathbf{X}(t, s) - \boldsymbol{\ell}(s)
\end{equation}
satisfies homogeneous Dirichlet boundary conditions $\mathbf{w}(t, 0) = \mathbf{w}(t, \pi) = \mathbf{0}$. Consequently, $\mathbf{w}(t, s)$ can be continuously extended to the domain $[0, 2\pi]$ via an odd reflection across the boundaries. This allows us to evaluate the non-local spatial operators using the Fast Fourier Transform (FFT) on the periodic domain $\mathbb{S}^1$.

\begin{figure}[htbp]
    \centering
    \begin{subfigure}{0.48\textwidth}
        \centering
        \includegraphics[width=\linewidth]{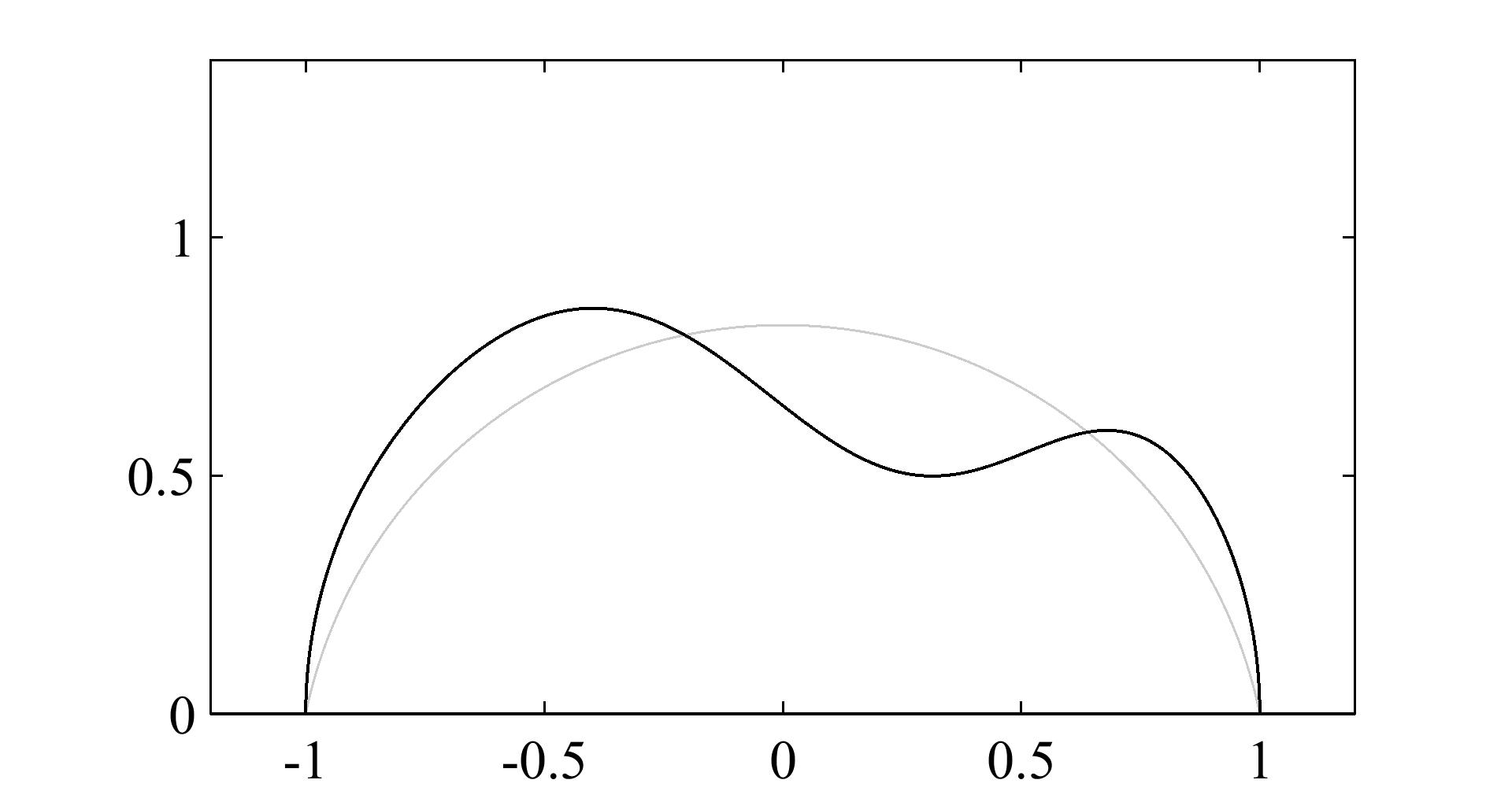}
        \caption{$T=0$}
        \label{fig:AsymmetricT0}
    \end{subfigure}
    \begin{subfigure}{0.48\textwidth}
        \centering
        \includegraphics[width=\linewidth]{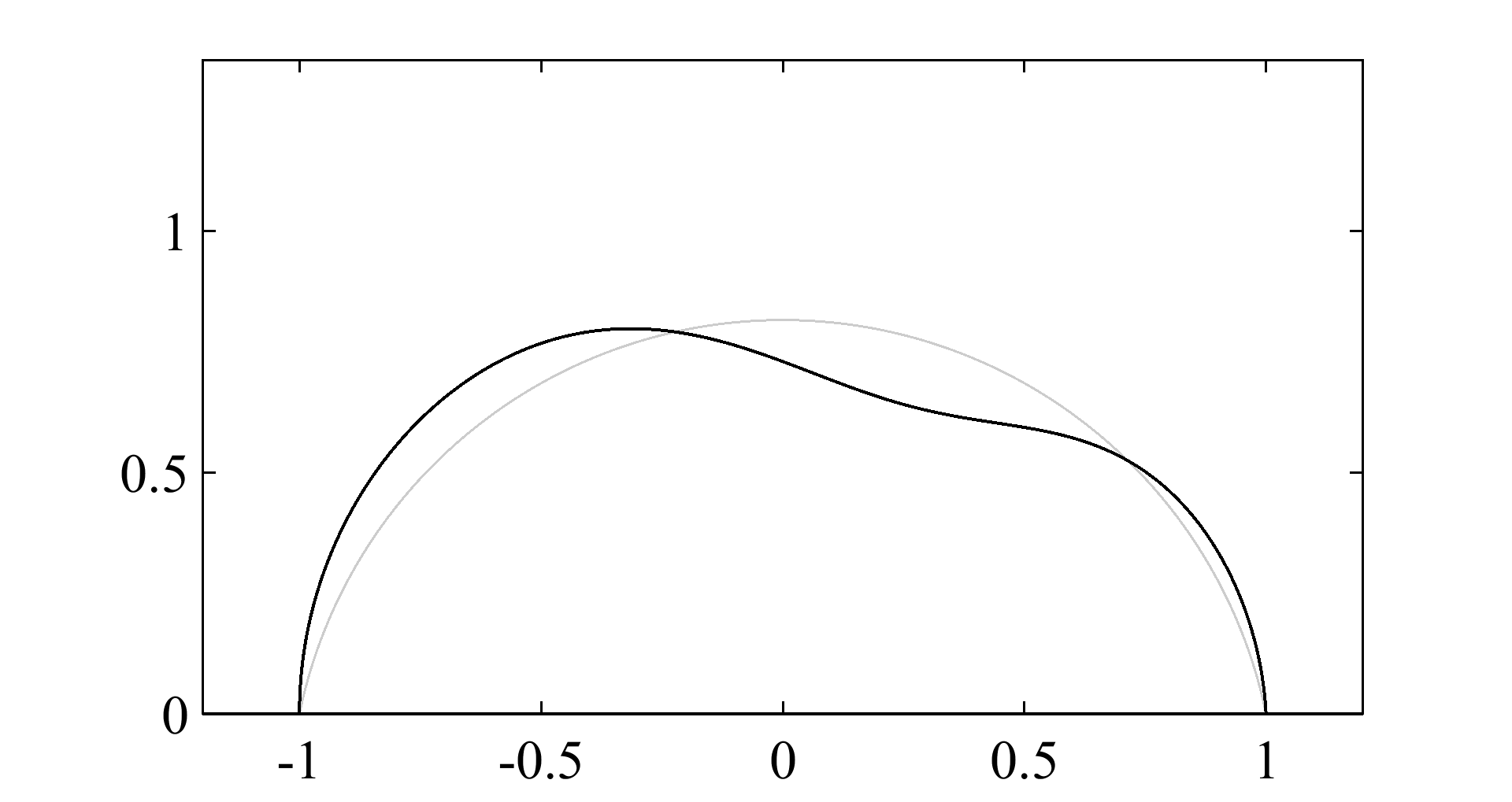}
        \caption{$T=1$}
        \label{fig:AsymmetricT1}
    \end{subfigure}
        \begin{subfigure}{0.48\textwidth}
        \centering
        \includegraphics[width=\linewidth]{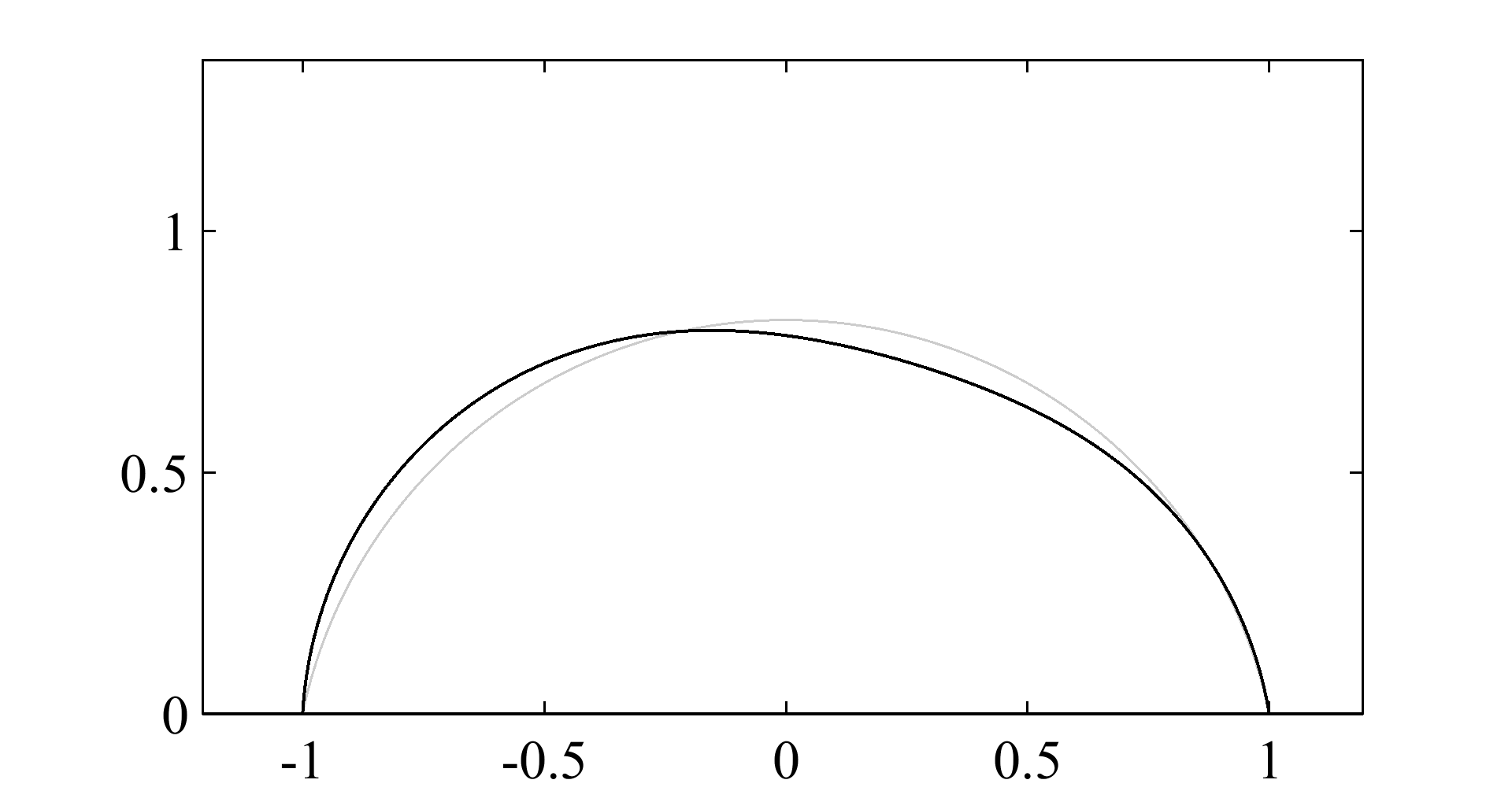}
        \caption{$T=2.5$}
        \label{fig:AsymmetricT1}
    \end{subfigure}
        \begin{subfigure}{0.48\textwidth}
        \centering
        \includegraphics[width=\linewidth]{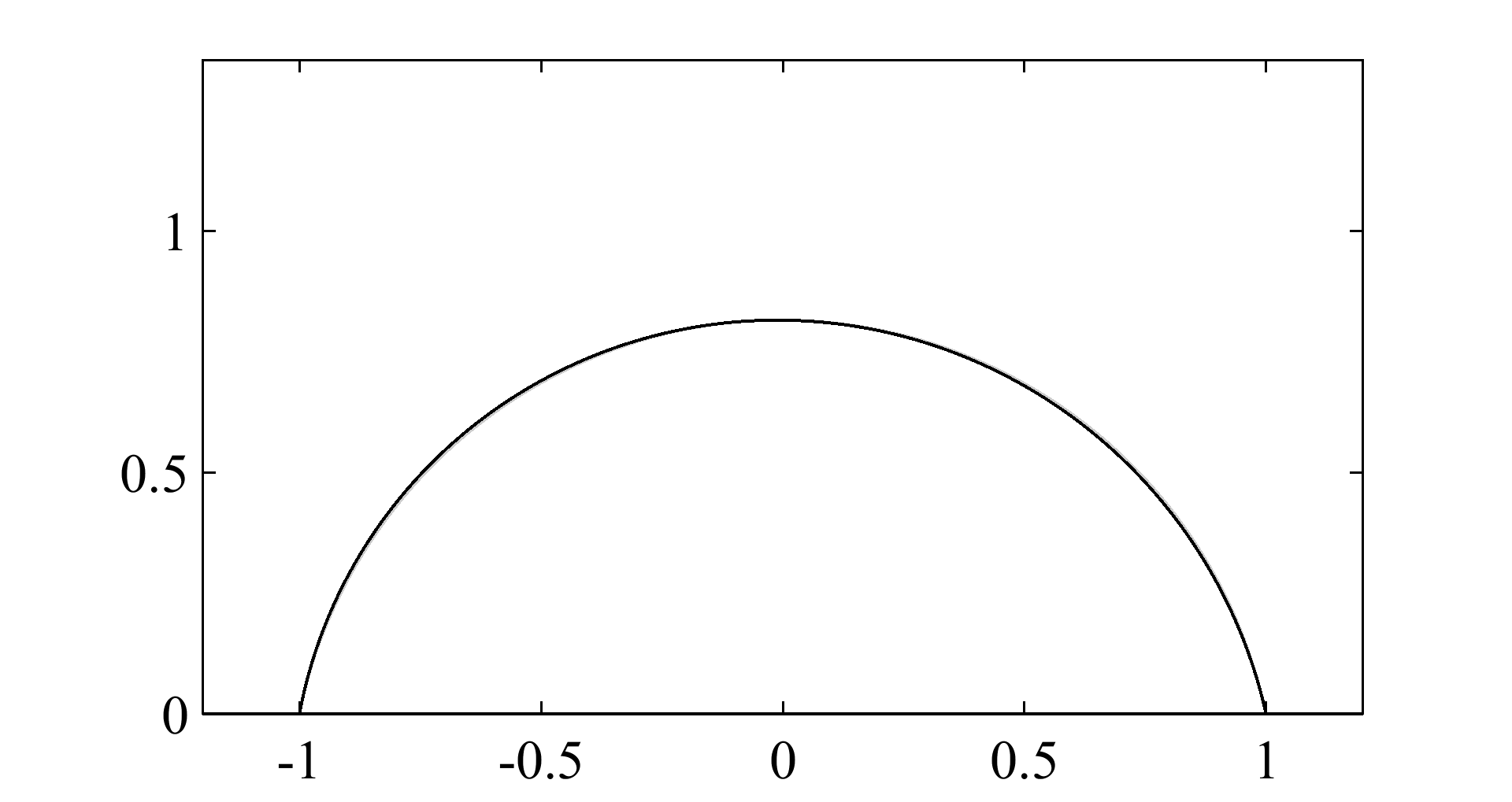}
        \caption{$T=10$}
        \label{fig:AsymmetricT1.5}
    \end{subfigure}
\caption{Sample evolution of an asymmetric initial profile defined by $X(s)=\cos(s)$ and $Y(s)=\sin(s)-a\exp[-(s-d)^2/(2w^2)]\sin^2(s)$ with $a=0.5,\,d=0.4\pi,\,w=0.12\pi$.  Simulation parameters are $\Delta t=0.01$ and $N=512$. The filament evolves to its circular-arc equilibrium (light gray) with the area conserved to spatial-quadrature accuracy.}
\label{fig:AsymmetricEvolution}
\end{figure}

\subsection{Discrete Odd Reflection}

We discretize the spatial domain $[0, \pi]$ into $N$ equispaced points $s_j = j \Delta s$ for $j = 0, \ldots, N-1$, with $\Delta s = \frac{\pi}{N-1}$. Let $\mathbf{u}_j$ denote a discrete vector field on these nodes. If $\mathbf{u}_0 \neq \mathbf{0}$ or $\mathbf{u}_{N-1} \neq \mathbf{0}$, we first isolate the homogeneous component:
\[
\mathbf{w}_j = \mathbf{u}_j - \left[ \mathbf{u}_0 + \frac{j}{N-1} \left( \mathbf{u}_{N-1} - \mathbf{u}_0 \right) \right].
\]
By construction, $\mathbf{w}_0 = \mathbf{w}_{N-1} = \mathbf{0}$. We define the discrete odd reflection $\widetilde{\mathbf{w}}$ on $M = 2N - 2$ points as:
\begin{align*}
    \widetilde{\mathbf{w}}_j = \left\{ 
    \begin{array}{ll} 
    \mathbf{w}_j & \quad \text{if } 0 \leq j \leq N-1, \\
    -\mathbf{w}_{M-j} & \quad \text{if } N \leq j \leq M-1.
    \end{array}
    \right.
\end{align*}
This maps the physical configuration onto a periodic grid suitable for spectral differentiation and the application of the linear propagator.

\subsection{Spectral Operators and Time Integration}

We approximate functions on $\mathbb{S}^1$ using trigonometric polynomials. For a discrete periodic signal $\mathbf{v}$ of length $M$, the discrete Fourier transform and its inverse are defined using wave numbers $k \in \{-M/2 + 1, \dots, M/2\}$. 

The spectral derivative operator $\mathcal{D}_h$ and the discrete linear semigroup $\mathcal{S}_h(t)$ are defined via their Fourier multipliers:
\[
\mathcal{D}_h \mathbf{v} = \mathcal{F}^{-1}_M \left[ \widehat{\mathcal{D}}_h(k) \mathcal{F}_M \mathbf{v} \right], \quad \text{where} \quad \widehat{\mathcal{D}}_h(k) = \left\{ 
\begin{array}{ll}
i k & \text{if } k \neq M/2, \\
0 & \text{if } k = M/2,
\end{array}\right.
\]
and
\[
\mathcal{S}_h(t) \mathbf{v} = \mathcal{F}^{-1}_M \left[ \widehat{\mathcal{S}}_h(k) \mathcal{F}_M \mathbf{v} \right], \quad \text{where} \quad \widehat{\mathcal{S}}_h(k) = \left\{ 
\begin{array}{ll}
e^{-t|k|/4} & \text{if } k \neq M/2, \\
0 & \text{if } k = M/2.
\end{array}\right.
\]
The symbol $-|k|/4$ corresponds precisely to the spectrum of our principal linear operator $\mathcal{L}_D = -\frac{1}{4}\Lambda$.

To advance the system from time $t_n$ to $t_{n+1} = t_n + \Delta t$, we employ a second-order exponential midpoint integrator. Let $\mathbf{X}_n$ denote the configuration at $t_n$. We first calculate a half-step prediction:
\begin{equation}
\mathbf{X}_{n+1/2} \approx \boldsymbol{\ell} + \mathcal{S}_h\left(\frac{\Delta t}{2}\right) \left[ \widetilde{\mathbf{X}_n - \boldsymbol{\ell}} \right] + \frac{\Delta t}{2} \mathcal{S}_h\left(\frac{\Delta t}{2}\right) \widetilde{\mathbf{R}(\mathbf{X}_n)}.
\end{equation}
We then use this intermediate state to compute the remainder and take the full time step:
\begin{equation}
\mathbf{X}_{n+1} \approx \boldsymbol{\ell} + \mathcal{S}_h(\Delta t) \left[ \widetilde{\mathbf{X}_n - \boldsymbol{\ell}} \right] + \Delta t \, \mathcal{S}_h\left(\frac{\Delta t}{2}\right) \widetilde{\mathbf{R}(\mathbf{X}_{n+1/2})}.
\end{equation}

\subsection{Quadrature of the Non-Local Remainder}

The non-local remainder $\mathbf{R}(\mathbf{X})$ requires the integration of singular kernels. Let $\mathbf{X}_{h, k}$ denote the evaluation of the configuration at node $s_k$, and let $\mathbf{X}^r_{h, \ell}$ denote the reflected configuration corresponding to the method of images. We discretize the integral kernels $H^1, H^2,$ and $H^3$ as matrices evaluating the interaction between nodes $k$ and $\ell$.

\begin{figure}[htbp]
    \centering
    \begin{subfigure}{0.48\textwidth}
        \centering
        \includegraphics[width=\linewidth]{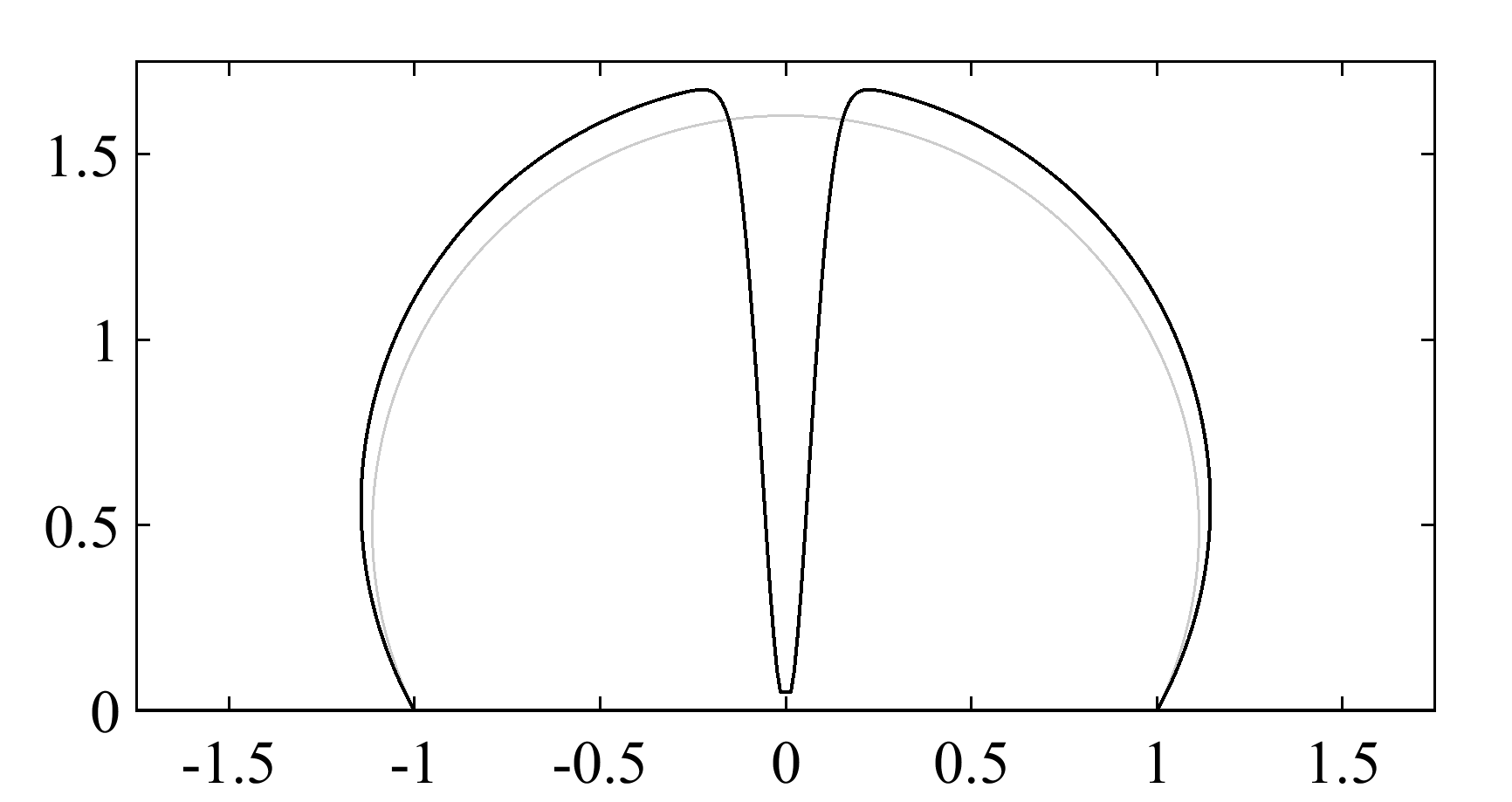}
        \caption{$T=0$}
        \label{fig:AsymmetricT0}
    \end{subfigure}
    \begin{subfigure}{0.48\textwidth}
        \centering
        \includegraphics[width=\linewidth]{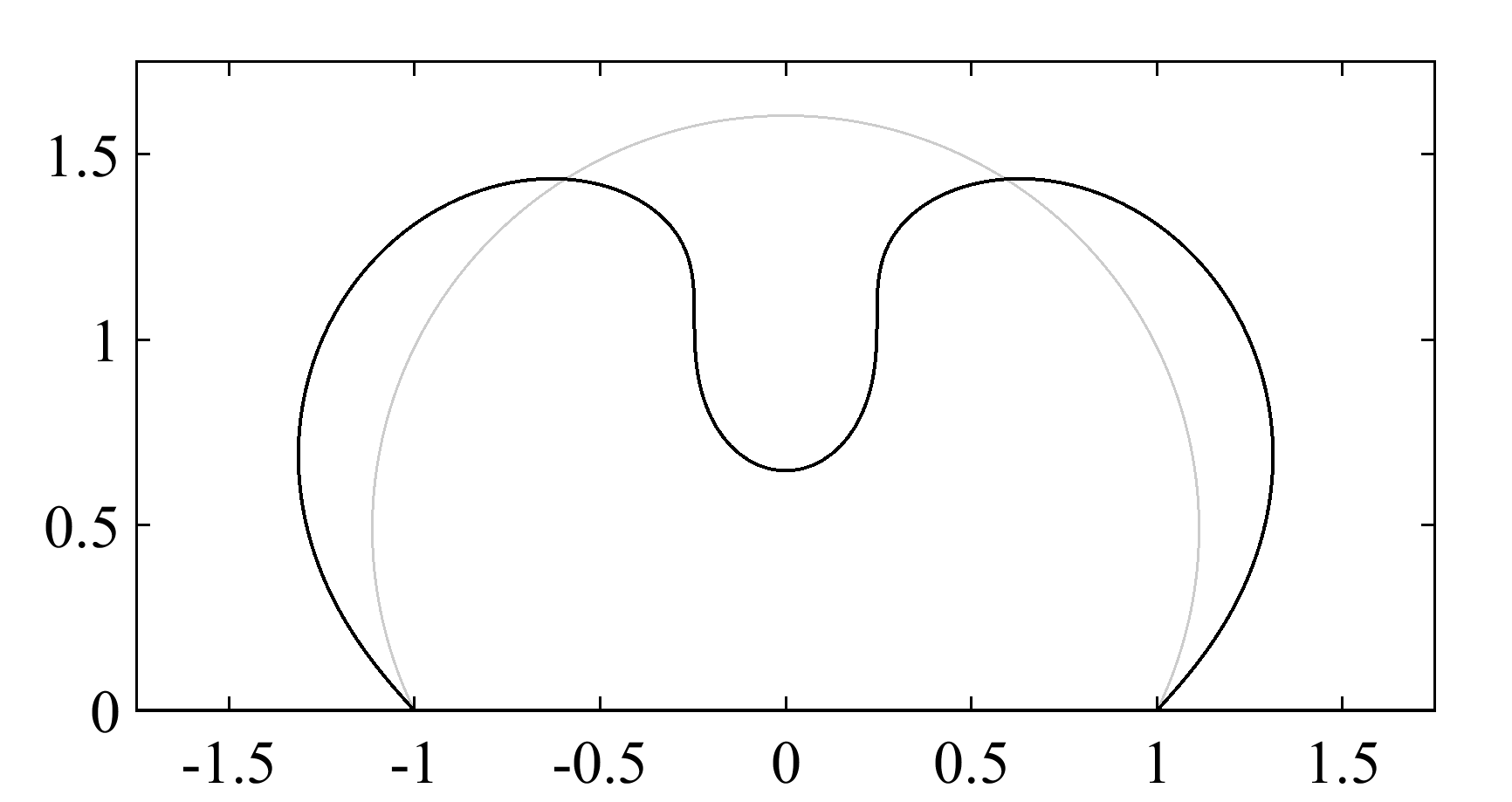}
        \caption{$T=0.4$}
        \label{fig:AsymmetricT0.4}
    \end{subfigure}
        \begin{subfigure}{0.48\textwidth}
        \centering
        \includegraphics[width=\linewidth]{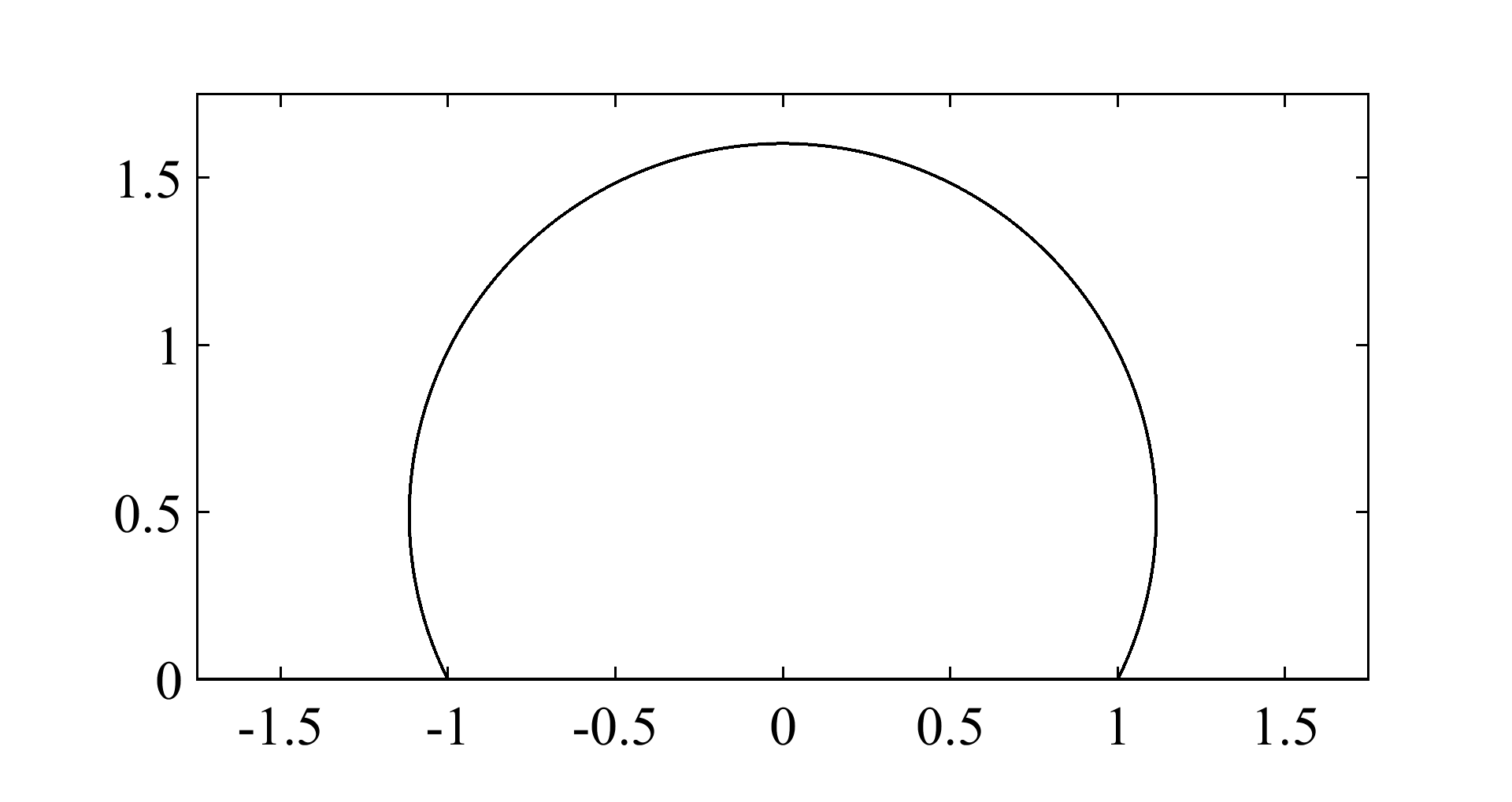}
        \caption{$T=10$}
        \label{fig:AsymmetricT10}
        \end{subfigure}
\caption{Evolution of a deeply notched obtuse arc toward its circular-arc equilibrium. We define the initial curves, parametrizing by  $\Phi \in \left[\arctan(\frac{-c}{1}),\pi-\arctan(\frac{-c}{1})\right]$ with $R_c=\sqrt{1+c^2}$, $c=\frac{h^2-1}{2y}$, and $h=1.7$ . The initial conditions are given by $X(\Phi)=X_{arc}$ and $Y(\Phi)=\max{(0.05\sin(s),Y_{arc})}$ where $X_{arc} = R_c\cos(\Phi)$ and $Y_{arc} = c + R_c\sin(\Phi)-s_d\cdot d_p$ where $d_p=\exp\left[-\left(s-\pi/2)/s_w\right)^2\right] \sin(s)$ and \(s_w=0.06\). Simulation parameters are $\Delta t=0.01$ and $N=512$ with $\Phi$ and $s$ discretized by $N$ points. At start time the notch is within 0.05 of the $x$-axis. The slit relaxes rapidly while the enclosed area is conserved to relative error. }
\label{fig:SlitEvolution}
\end{figure}
For the primary Stokeslet term, the removable logarithmic singularity is regularized using the spectral derivative at the diagonal $k = \ell$:
\begin{align*}
H^1_{k\ell}  = \left\{ \begin{array}{ll} 
- \log\left( \frac{|\mathbf{X}_{h,k} - \mathbf{X}_{h,\ell} |}{2 \sin(|s_k - s_\ell|/2)} \right) I_2 + \frac{ (\mathbf{X}_{h,k} - \mathbf{X}_{h,\ell}) \otimes (\mathbf{X}_{h,k} - \mathbf{X}_{h,\ell}) }{ |\mathbf{X}_{h,k} - \mathbf{X}_{h,\ell}|^2 } & \text{if } k \neq \ell, \\
- \log\left( |(\mathcal{D}_{h} \mathbf{X}_h)_k| \right) I_2 + \frac{ (\mathcal{D}_h \mathbf{X}_{h,k} ) \otimes (\mathcal{D}_h \mathbf{X}_{h,k} ) }{ |(\mathcal{D}_h \mathbf{X}_{h,k} )|^2} & \text{if } k = \ell.
\end{array} \right.
\end{align*}
The boundary-reflected wall terms are strictly non-singular on the interior. Following the structure of \eqref{e:halfspacestokeslet}, the reflected log--tensor piece $H^2$ and the wall-normal piece $H^3$ act on the \emph{reflected} tangent $\mathbf{f}^r=R\mathbf{f}$ where $R=\operatorname{diag}(1,-1)$. The matrix kernels acting on the unreflected tangent $\mathbf{f}=\partial_{s'}\mathbf{X}$ are therefore obtained by composing with $R$:
\begin{align*}
H^2_{k\ell}  &=  \left[-\log\!\left(\frac{|\mathbf{X}_{h,k}-\mathbf{X}^r_{h,\ell}|}{2\sin(|s_k+s_\ell|/2)}\right)I_2 + \frac{(\mathbf{X}_{h,k}-\mathbf{X}_{h,\ell})\otimes(\mathbf{X}_{h,k}-\mathbf{X}^r_{h,\ell})}{|\mathbf{X}_{h,k}-\mathbf{X}^r_{h,\ell}|^2}\right]R, \\
H^3_{k\ell} &= -\frac{(\mathbf{e}_2\cdot\mathbf{X}_{h,k})}{|\mathbf{X}_{h,k}-\mathbf{X}^r_{h,\ell}|^2}\,(\mathbf{X}_{h,k}-\mathbf{X}^r_{h,\ell})\otimes\mathbf{e}_2 \\
&\quad - (\mathbf{e}_2\cdot\mathbf{X}_{h,k})(\mathbf{e}_2\cdot\mathbf{X}_{h,\ell}) \left[\frac{I_2}{|\mathbf{X}_{h,k}-\mathbf{X}^r_{h,\ell}|^2} - \frac{2(\mathbf{X}_{h,k}-\mathbf{X}^r_{h,\ell})\otimes(\mathbf{X}_{h,k}-\mathbf{X}^r_{h,\ell})}{|\mathbf{X}_{h,k}-\mathbf{X}^r_{h,\ell}|^4}\right]R,
\end{align*}
where the first term of $H^3$ involves the rank-one structure $(\Delta^r\otimes\mathbf{e}_2)$ acting on $\mathbf{f}$, contracting only the second component $\mathbf{f}_2$, and so does not require composition with $R$.

The discrete non-local remainder $\mathbf{R}_h$ at each node $k$ is computed by applying the spectral derivative to the trapezoidal quadrature of the kernels acting on the tangent vectors:
\begin{equation}
\mathbf{R}_{h, k} = \left( \mathcal{D}_h \left[ \sum_{\ell=0}^{N-1} \left( \sum_{q=1}^3 H^q_{\cdot, \ell} \right) (\mathcal{D}_h \mathbf{X}_{h, \ell}) \Delta s \right] \right)_k.
\end{equation}
Thus we have fully described the numerical scheme.\\


\subsection{Area conservation and convergence to equilibrium}
\label{ss:area_conservation}

The numerical scheme conserves area to spatial-quadrature accuracy, with the cancellation between the linear semigroup and the nonlinear remainder operating in the discrete formulation as it does in the continuous one.  We begin by recalling the relevant continuous-flow result.


In the discrete scheme, the operator splitting evaluates $e^{\mathcal{L}_D \Delta t}$ and $\mathbf{R}^h$ sequentially.  Because $\mathcal{L}_D$ is diagonalized analytically on the odd-extended grid via the FFT, the linear semigroup is applied without spatial discretization error, and $\delta A_{\mathcal{L}_D}$ matches its continuous counterpart exactly.  The discretization error therefore enters only through $\mathbf{R}^h$.  Crucially, the discrete remainder produced by the kernel \eqref{e:halfspacestokeslet} preserves the algebraic structure required for the continuous cancellation: each piece of $\mathbf{R}^h$ contributes the same surface flux as its continuous counterpart up to the trapezoidal-rule defect.  Consequently $\delta A_{\mathcal{L}_D}$ (negative) and $\delta A_{\mathbf{R}^h}$ (positive) cancel at the discrete level up to a spatial-quadrature error that converges to zero as $N\to\infty$.


To stress-test the discrete cancellation away from near-equilibrium configurations, we run the scheme on a far-from-equilibrium initial datum: an obtuse circular arc with a deep, narrow Gaussian slit at the apex that descends to within roughly $0.05$ of the wall (Figure~\ref{fig:SlitEvolution}).  This configuration has a large $\|\mathbf{R}\|$ and high tangent slopes near the slit walls, so any failure of the cancellation would manifest as a large per-step area drift at the start.  In numerical experiments at $N = 512$ and $\Delta t = 0.01$, we observe the slit relaxing rapidly through a dumbbell shape and converging to the obtuse-arc equilibrium predicted by Section~\ref{s:equilibria}, with the enclosed area conserved to relative error $\sim 5\times 10^{-5}$ over the full evolution.
\begin{figure}[htbp]
    \centering
    \begin{subfigure}{0.48\textwidth}
        \centering
        \includegraphics[width=\linewidth]{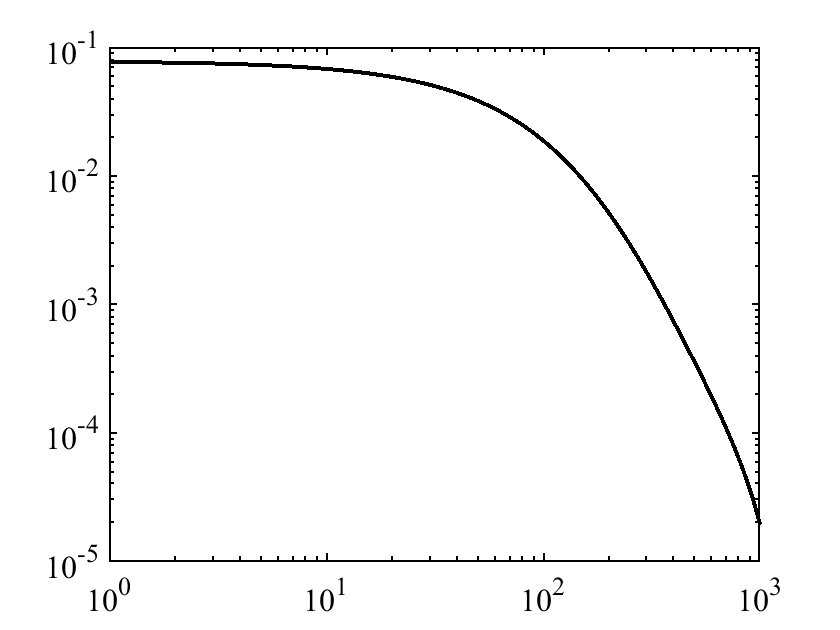}
        \caption{Convergence error of the isoperimetric ratio to $1$ in  Figure~\ref{fig:AsymmetricEvolution}}
        \label{fig:IsoErrorAsymmetric}
    \end{subfigure}
    \begin{subfigure}{0.48\textwidth}
        \centering
        \includegraphics[width=\linewidth]{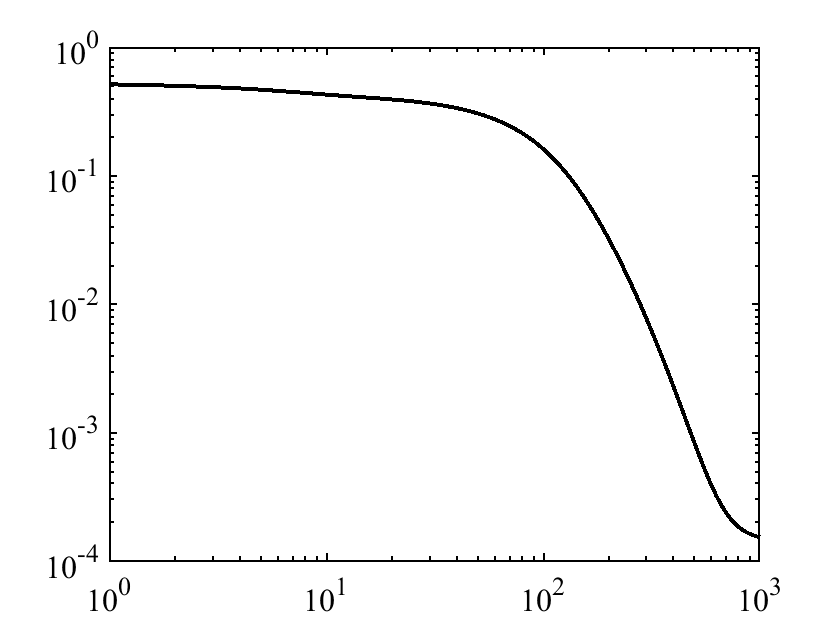}
        \caption{Convergence error of the isoperimetric ratio to $1$  in  Figure~\ref{fig:SlitEvolution}}
        \label{fig:IsoErrorObtuseSlit}
    \end{subfigure}
    \caption{ $\log / \log$ plots of the error of the convergence of the isoperimetric ratio $R$ to 1 of the asymmetric arc and deeply notched obtuse arc over 1000 time steps. We observe consistent convergence trends in both cases.}
\end{figure}

\appendix

\renewcommand\thetheorem{\Alph{section}.\arabic{theorem}}
\renewcommand\thelemma{\Alph{section}.\arabic{theorem}}

\section{Reflected Stokeslet}
\label{appendix:stokeslet}
We recall the derivation of the half-space Stokeslet following \cite{gimbutas2015simple}.
We begin with the free-space Stokeslet in $\mathbb{R}^2$,
\[
S[\mf](\mx,\my)
= \frac{1}{4\pi}
\left(
-\log|\mx-\my|\,\mf
+ \frac{(\mf\cdot(\mx-\my))(\mx-\my)}{|\mx-\my|^2}
\right).
\]
Let $\my^r=(y_1,-y_2)$ denote the reflection of $\my$ across the boundary
$\partial\mathbb{R}^2_+=\{x_2=0\}$, and define the reflected Stokeslet
\[
S^r[\mf](\mx,\my)=S[\mf^r](\mx,\my^r), \qquad \mf^r=(f_1,-f_2).
\]
Both $S$ and $S^r$ are divergence free in $\mx$ and solve the Stokes equations
away from their singularities.

On the boundary $x_2=0$, the tangential components cancel, but the normal
component does not. A direct computation gives
\begin{equation}\label{eq:boundary-mismatch}
\left.(S-S^r)[\mf](\mx,\my)\right|_{x_2=0}
= -\frac{\mf_2}{2\pi}\log|\mx-\my|\,\mathbf e_2
-\frac{y_2}{2\pi}
\frac{\mf^r\cdot(\mx-\my^r)}{|\mx-\my|^2}\,\mathbf e_2 .
\end{equation}
Thus, the image construction $S-S^r$ fails to satisfy the no-slip condition,
leaving a purely vertical residual velocity on the boundary.

To cancel \eqref{eq:boundary-mismatch}, we introduce a harmonic scalar
potential. Let
\[
H(\mx,\my)=-\frac{1}{2\pi}\log|\mx-\my|, \qquad
\nabla_y H(\mx,\my)=\frac{1}{2\pi}\frac{\mx-\my}{|\mx-\my|^2}.
\]
Define
\[
\Phi[\mf](\mx,\my)
= \mf_2 H(\mx,\my^r)
- y_2 \sum_{j=1}^2 \mf_j^r \,\partial_{y_j}H(\mx,\my^r).
\]
Equivalently,
\[
\Phi[\mf](\mx,\my)
= -\frac{\mf_2}{2\pi}\log|\mx-\my^r|
-\frac{y_2}{2\pi}
\frac{\mf^r\cdot(\mx-\my^r)}{|\mx-\my^r|^2}.
\]
Since $\Phi$ is harmonic in $\mathbb{R}^2_+$, we may define the correction field
\[
S^T[\mf](\mx,\my)
=
\begin{bmatrix}
0 \\ -\Phi[\mf](\mx,\my)
\end{bmatrix}
+ x_2 \nabla_x \Phi[\mf](\mx,\my).
\]
On the boundary $x_2=0$,
\[
\left.S^T[\mf]\right|_{x_2=0}
=
-\begin{bmatrix}
0 \\ \Phi[\mf]
\end{bmatrix},
\]
which exactly cancels the residual velocity in \eqref{eq:boundary-mismatch}.
Moreover,
\[
\operatorname{div} S^T
= -\partial_{x_2}\Phi + x_2 \Delta \Phi + \partial_{x_2}\Phi
= 0,
\]
so $S^T$ is divergence free. Choosing an associated pressure
$p[\mf]$ ensures that $S^T$ satisfies the Stokes equations.
 
The half-space Green’s function is therefore
\[
S^+[\mf](\mx,\my)
= S[\mf](\mx,\my) - S^r[\mf](\mx,\my) + S^T[\mf](\mx,\my).
\]
Combining terms yields
\begin{equation}\label{e:halfspacestokeslet_appendix}
\begin{split}
S^+[\mf](\mx,\my)
&= \frac{1}{4\pi}
\big[
-\log|\mx-\my|+\log|\mx-\my^r|
\big]\mf \\
&\quad + \frac{1}{4\pi}
\left[
\frac{(\mx-\my)\otimes(\mx-\my)}{|\mx-\my|^2}\mf
-\frac{(\mx-\my)\otimes(\mx-\my^r)}{|\mx-\my^r|^2}\mf^r
\right] \\
&\quad - \frac{x_2\mf_2}{2\pi}\frac{\mx-\my^r}{|\mx-\my^r|^2}
- \frac{x_2y_2}{2\pi}
\left[
\frac{\mf^r}{|\mx-\my^r|^2}
-2\frac{(\mx-\my^r)\otimes(\mx-\my^r)}{|\mx-\my^r|^4}\mf^r
\right].
\end{split}
\end{equation}
Although \eqref{e:halfspacestokeslet_appendix} is algebraically involved, the logarithmic
term is the principal singular part of the Green's function. The remaining terms
are lower order and decay faster. In our analysis and numerics, this principal
operator will serve as the generator of the associated semigroup.

To furnish $S^+$ with a pressure satisfying
$-\Delta_x S^+[\mf](\mx,\my)+\nabla_x P^+[\mf](\mx,\my)=\mf\,\delta(\mx-\my)$
in $\mathbb{R}^2_+$, we treat each of the three components $S$, $S^r$,
$S^T$ in turn, with the convention $\mu=1$.
The free-space Stokeslet pressure paired with $S$ is
\[
P[\mf](\mx,\my) \;=\; \frac{1}{2\pi}\,\frac{\mf\cdot(\mx-\my)}{|\mx-\my|^2},
\]
which together with $S$ satisfies $-\Delta_x S+\nabla_x P=\mf\,\delta(\mx-\my)$
and $\nabla_x\!\cdot\! S=0$ in $\R^2$. The reflected counterpart is
\[
P^r[\mf](\mx,\my) \;=\; P[\mf^r](\mx,\my^r)
\;=\; \frac{1}{2\pi}\,\frac{\mf^r\cdot(\mx-\my^r)}{|\mx-\my^r|^2}.
\]
Since $\my^r\in\R^2_-$, the pair $(S^r,P^r)$ solves the homogeneous Stokes
equations classically in $\overline{\R^2_+}$.
For the harmonic correction $S^T$, recall
\[
S^T[\mf](\mx,\my) \;=\; \begin{pmatrix}0\\-\Phi\end{pmatrix} + x_2\,\nabla_x\Phi,
\qquad \Delta_x \Phi[\mf](\mx,\my) = 0 \ \text{in }\R^2_+.
\]
Using $\Delta_x(x_2\,g) = 2\,\p_{x_2}g + x_2\,\Delta_x g$ for smooth $g$
and the harmonicity of $\Phi$, a direct computation gives
\[
\Delta_x S^T[\mf](\mx,\my) \;=\; 2\,\nabla_x\!\bigl(\p_{x_2}\Phi[\mf](\mx,\my)\bigr).
\]
The choice
\[
P^T[\mf](\mx,\my) \;:=\; 2\,\p_{x_2}\Phi[\mf](\mx,\my)
\]
therefore satisfies $-\Delta_x S^T+\nabla_x P^T=0$ in $\R^2_+$, while
$\nabla_x\!\cdot\! S^T = x_2\,\Delta_x \Phi = 0$ as already noted.
 
Differentiating $\Phi$ explicitly with $\mathbf{R}:=\mx-\my^r=(x_1-y_1,\,x_2+y_2)$,
\[
\p_{x_2}\Phi[\mf](\mx,\my)
\;=\; -\frac{f_2\,x_2}{2\pi\,|\mathbf{R}|^2}
\;+\; \frac{y_2(x_2+y_2)\,\mf^r\cdot\mathbf{R}}{\pi\,|\mathbf{R}|^4},
\]
yielding
\begin{equation}\label{eq:PT}
P^T[\mf](\mx,\my)
\;=\; -\frac{x_2\,f_2}{\pi\,|\mx-\my^r|^2}
\;+\; \frac{2\,y_2(x_2+y_2)\,\mf^r\cdot(\mx-\my^r)}{\pi\,|\mx-\my^r|^4}.
\end{equation}
Combining the three contributions, the half-space Stokeslet pressure is
\begin{equation}\label{eq:P+}
\begin{split}
P^+[\mf](\mx,\my)
&= \frac{1}{2\pi}\,\frac{\mf\cdot(\mx-\my)}{|\mx-\my|^2}
\;-\; \frac{1}{2\pi}\,\frac{\mf^r\cdot(\mx-\my^r)}{|\mx-\my^r|^2} \\
&\quad\;-\; \frac{x_2\,f_2}{\pi\,|\mx-\my^r|^2}
\;+\; \frac{2\,y_2(x_2+y_2)\,\mf^r\cdot(\mx-\my^r)}{\pi\,|\mx-\my^r|^4}.
\end{split}
\end{equation}
By construction, $(S^+, P^+)$ satisfies
\[
-\Delta_x S^+[\mf](\mx,\my) + \nabla_x P^+[\mf](\mx,\my)
\;=\; \mf\,\delta(\mx-\my),
\qquad \nabla_x\!\cdot\!S^+[\mf](\mx,\my) = 0
\]
in $\R^2_+$, and $S^+[\mf](\mx,\my)\big|_{x_2=0}=\mathbf 0$.
 
Two consistency checks: as $y_2\to 0$ (source approaching the wall), one
finds $P^+\equiv 0$, reflecting the fact that the no-slip Green's function
cannot transmit a point force placed exactly on the wall; and as
$|\mx|\to\infty$, the leading $|\mx|^{-1}$ contribution from each of the
four terms in \eqref{eq:P+} cancels, so $P^+(\mx,\my) = O(|\mx|^{-2})$,
the expected boundary-induced far-field suppression.

\section{Trigonometric inequalities}

\begin{lemma} \label{lem:sinequiv}
For all $0 \leq x,y \leq \pi$ we have $\sin\left( {x+y \over 2}\right) \sim  \sin\left( x \right)+ |x-y| $,  in the sense that 
\[
{1\over C}\left( \sin\left(x\right) + |x-y| \right) \leq \sin \left( {x+y \over 2}\right) \leq C \left( \sin\left( x \right)+ |x-y| \right)
\]
with $C = 8\pi$.
\end{lemma}

\begin{proof}
We begin with elementary trigonometric bounds:
\[
\begin{cases}
{2\over \pi } z \leq \sin(z) \leq z &  \hbox{ for } 0 \leq z \leq {\pi \over 2} \\
1 - {2\over \pi } \left( z - {\pi \over 2} \right) \leq \sin(z) \leq \pi - z & \hbox{ for } {\pi \over 2} \leq z \leq {\pi }  \\
{1\over \sqrt{2}}  \leq \cos(z) \leq 1 &  \hbox{ for } 0 \leq z \leq {\pi \over 4}.
\end{cases}
\]
We proceed by cutting $[0,\pi]^2$ into four regions.

\textbf{Case 1.} We first consider $0\leq x,y \leq {\pi \over 4}$.  
$0 \leq {x + y \over 2} \leq {\pi \over 4}$.
We now consider two separate cases.  First suppose $0 \leq x \leq y \leq {\pi \over 4}$ and set $h = y-x$, then $0 \leq h \leq {\pi \over 4}$.  Then 
$x + y = 2x + h$, or ${x + y \over 2} = x + {h\over 2}$.  Returning to our estimate, note that
\[
\sin \left( {x + y\over 2} \right)
= \sin \left( x \right) \cos \left( {h\over 2} \right) + \cos(x) \sin\left( {h\over 2} \right).
\]
Using our earlier bounds, we find
\begin{align*}
\sin \left( {x + y\over 2} \right)
& \leq  \sin\left( x \right)  + {h \over 2}  \leq  \sin\left( x \right) + h ,
\end{align*}
and 
\begin{align*}
\sin \left( {x + y\over 2} \right) 
& \geq  \sin\left(x \right) \cos\left( {\pi \over 8} \right)+ \cos\left({\pi \over 4} \right) {h \over \pi}   \geq {3 \over 4} \sin\left( x\right) + {1 \over \sqrt{2} \pi}h \\
& \geq {1 \over \sqrt{2}\pi} \left( \sin\left( x\right) + h \right).
\end{align*}

We now consider, $0 \leq y \leq x \leq {\pi \over 4}$.  In this case we can write $x + |x-y| = 2x - y$, and so
\begin{align*}
    x &\leq x+ y \leq 2x\\
     x & \leq x + |x - y| \leq 2x.  
\end{align*}
Therefore, $x+ y  \leq 2 x \leq 2 ( 2 x - y)= 2 ( x + |x -y|)$, 
and $x+ y \geq x \geq {1\over 2} ( 2 x - y) = {1\over 2} ( x + |x -y|)$, 
which implies
\begin{align*}
    {1\over 2} ( x + |x -y|) \leq x + y \leq 2 ( x + |x -y|).
\end{align*}
In turn we can estimate 
\begin{align*}
\sin \left( {x + y\over 2} \right)  
 \leq x + |x -y| 
\leq {\pi\over 2} \left( \sin\left( x\right) + |x - y|  \right). 
\end{align*}
Likewise, 
\begin{align*}
\sin \left( {x + y\over 2} \right)  
 \geq {2 \over \pi} \left( {x + y \over 2} \right) 
\geq {1\over 2 \pi} \left( x + |x -y| \right)
\geq {1\over 2\pi} \left( \sin\left( x \right) + |x - y|  \right). 
\end{align*}
A similar argument holds in the region ${3 \pi \over4 } \leq x,y \leq \pi$. In particular, noting $\sin({x+y\over2}) = \sin({x-\pi+y-\pi\over2}+\pi) = \sin({(\pi - x )+ (\pi-y)\over2})$, and we can repeat the argument from above by substituting $x$ with $\pi - x$, etc.

\textbf{Case 2.} We now consider the domain $0 \leq x \leq {\pi \over 8}$ and ${\pi \over 4} \leq y \leq \pi$.  In this set 
${\pi \over 8} \leq {x+y\over 2} \leq {9 \pi \over 16}$,  ${\pi \over 8} \leq |x -y| \leq \pi$, and $0\leq \sin(x) \leq {\pi \over 8}$.  In particular,
\[
{1\over 4} = {2 \over \pi} {\pi \over 8} \leq {2 \over \pi}\left( {x + y\over 2} \right) \leq \sin \left( {x + y\over 2} \right)
\leq 1 ;
\]
consequently, 
\begin{align*}
\sin\left( {x + y\over 2} \right)
\leq 1 \leq {8 \over \pi} |x - y| \leq {8 \over \pi} \left( \sin\left( x \right) + |x - y|  \right)
\end{align*}
Likewise,
\begin{align*}
\sin\left( {x + y\over 2} \right)
\geq {1\over 8} + {1\over 8} \geq
{1\over \pi} \sin\left( x\right) + {1\over 8 \pi} |x-y| \geq {1\over 8 \pi}\left( \sin\left( x\right) + |x - y| \right).
\end{align*}
Similar estimates hold in the domain ${7\pi \over 8} \leq x \leq {\pi }$ and ${0} \leq y \leq {3\pi \over 4}$.  

\textbf{Case 3.} We now consider the domain $0 \leq y \leq {\pi \over 8}$ and ${\pi \over 4} \leq x \leq \pi$.  In this set 
${\pi \over 8} \leq {x+y\over 2} \leq {9 \pi \over 16}$ and ${\pi \over 8} \leq |x -y| \leq \pi$. Similar two sided bounds hold.
  
\textbf{Case 4.} Finally, in the set ${\pi \over 8} \leq x, y \leq {7 \pi \over 8}$, we have ${\pi \over 8} \leq {x+ y\over 2} \leq {7 \pi \over 8}$, $0 \leq |x-y| \leq {3 \pi \over 4}$, and ${1\over 4} \leq \sin\left( {x}\right) \leq 1$.  Then
\begin{align*}
\sin\left( {x + y\over 2} \right) \leq {x + y\over 2} \leq {7\pi \over 8} \leq {7\pi \over 2} \sin\left( {x}\right) \leq {7\pi \over 2} \left(\sin\left( {x}\right) + |x - y| \right) 
\end{align*}
and
\begin{align*}
\sin\left( {x + y\over 2} \right) & \geq {2 \over \pi} \left( {x + y\over 2} \right) \geq  
{1\over 8} + {1\over 8} \geq 
{1\over 8} \sin\left( {x}\right) 
+ {1 \over 6\pi} |x-y| \\
& 
\geq {1\over 6\pi}\left(\sin\left( {x }\right) + |x - y| \right) 
\end{align*}

Combining the estimates from Cases 1--4 together yields the estimate for $C = 8\pi$.

\end{proof}

\bibliographystyle{abbrv}

\bibliography{PeskinWall.bib}

\end{document}